\documentclass[12pt]{amsart}
\usepackage{amssymb,amsmath, amsthm,latexsym,verbatim}
\usepackage{a4wide}
\usepackage{mathrsfs}
\usepackage[hypertex]{hyperref}
\usepackage{enumitem}
\usepackage{color}

\theoremstyle{plain}

\newtheorem{lem}{Lemma}[section]
\newtheorem{theo}[lem]{Theorem}
\newtheorem{prop}[lem]{Proposition}
\newtheorem{corollary}[lem]{Corollary}
\newtheorem{remark}[lem]{Remark}
\newtheorem{definition}[lem]{Definition}

\parskip=\medskipamount
\font\k=cmr7
\font\rm=cmr12

  \newcommand {\di}{\mbox{\k dis}}
  \newcommand {\fin}{\mbox{\k fin}}
  \newcommand{\unip}{\operatorname{unip}}

  \newcommand {\reg}{\mbox{\k reg}}
  
  \newcommand {\spec}{\mbox{\k spec}}
  \newcommand {\geo}{\mbox{\k geo}}

  \newcommand {\tors}{\mbox{\k tors}}
  \newcommand {\C}{{\mathbb C}}
  \newcommand {\bH}{{\mathbb H}}
  \newcommand {\N}{{\mathbb N}}
  \newcommand {\R}{{\mathbb R}}
  \newcommand {\Z}{{\mathbb Z}}
  \newcommand {\Q}{{\mathbb Q}}
  \newcommand {\A}{{\mathbb A}}
  
  \newcommand {\af}{{\mathfrak a}}
  \newcommand {\gf}{{\mathfrak g}}
  
  \newcommand {\kf}{{\mathfrak k}}

  \newcommand {\of}{{\mathfrak o}}
  \newcommand {\nf}{{\mathfrak n}}
  
  \newcommand {\pg}{{\mathfrak p}}
   \newcommand {\pf}{{\mathfrak p}}

 \newcommand{\ho}{{\mathfrak o}}
 \newcommand{\fS}{{\mathfrak S}}
 
\renewcommand {\H}{{\mathcal H}}
  \newcommand {\M}{{\mathcal M}}
  \newcommand {\cF}{{\mathcal F}}
  \newcommand {\Co}{{\mathcal C}}
 \newcommand {\cO}{{\mathcal O}}
 \newcommand {\ccC}{{\mathscr C}}
  \newcommand {\G}{{\bf G}}

 \newcommand {\cH}{{\mathcal H}}
 \newcommand {\cP}{{\mathcal P}}
 \newcommand {\cL}{{\mathcal L}}
 \newcommand {\cA}{{\mathcal A}}
 \newcommand {\cM}{{\mathcal M}}
\newcommand {\cT}{{\mathcal T}}
\newcommand  {\cZ}{{\mathcal Z}}
\newcommand  {\cU}{{\mathcal U}}
\newcommand {\bs}{\backslash}

\newcommand {\bx}{{\bf x}}
 
\newcommand{\cV}{{\mathcal V}}

\newcommand{\levis}{{\mathcal L}}
\newcommand{\Ai}{A_M(\R)^0}
\newcommand{\Ag}{A_G(\R)^0}

\renewcommand{\Re}{\operatorname{Re}}
\newcommand{\Tr}{\operatorname{Tr}}

\newcommand{\End}{\operatorname{End}}

\newcommand{\tr}{\operatorname{tr}}
\newcommand{\Id}{\operatorname{Id}}
\newcommand{\Hom}{\operatorname{Hom}}

\newcommand{\Ind}{\operatorname{Ind}}

\newcommand{\rk}{\operatorname{rank}}
\newcommand{\vol}{\operatorname{vol}}

\newcommand{\SL}{\operatorname{SL}}
\newcommand{\GL}{\operatorname{GL}}
\newcommand{\SO}{\operatorname{SO}}

\newcommand{\Ad}{\operatorname{Ad}}

\newcommand{\rO}{\operatorname{O}}
\newcommand{\supp}{\operatorname{supp}}

\renewcommand{\det}{\operatorname{det}}

\newcommand{\arcsinh}{\operatorname{arcsinh}}
\newcommand{\Rep}{\operatorname{Rep}}

\newcommand{\ve}{\varepsilon}

\newcommand{\norm}[1]{\lVert#1\rVert}
\newcommand{\abs}[1]{\lvert#1\rvert}
\newcommand{\eps}{\epsilon}
\newcommand{\one}{\mathbf 1}

\newcommand{\aaa}{\mathfrak{a}}
  \newcommand {\K}{{\bf K}}
  
  \newcommand {\bU}{{\bf U}}
  
  \newcommand{\Ht}{H}
\newcommand{\sprod}[2]{\left\langle#1,#2\right\rangle}
\newcommand{\PPP}{\mathcal{P}}
\newcommand{\FFF}{{\mathcal F}}
\newcommand{\rts}{\Sigma}

\newcommand{\disc}{\operatorname{disc}}
\newcommand{\srts}{\Delta}
\newcommand{\modulus}{\delta}
\newcommand{\AF}{{\mathcal A}}
\newcommand{\zzz}{\mathfrak{z}}
\newcommand{\iii}{{\mathrm i}}
\newcommand{\LieG}{\mathfrak{g}}
\newcommand{\bases}{\mathfrak{B}}
\newcommand{\bss}{\underline{\beta}}
\newcommand{\dtup}{\mathcal{X}}
\newcommand{\card}[1]{\lvert#1\rvert}

\newcommand{\ka}{\mathfrak{a}}
\newcommand{\kg}{\mathfrak{g}}
\newcommand{\kn}{\mathfrak{n}}
\newcommand{\CmP}{\mathcal{P}}
\newcommand{\CmU}{\mathcal{U}}
\newcommand{\Mid}{\operatorname{id}}
\newcommand{\Wt}{\operatorname{Wt}}
\newcommand{\diag}{\operatorname{diag}}
\newcommand{\CmO}{\mathcal{O}}
\newcommand{\CmF}{\mathcal{F}}
\newcommand{\CmN}{\mathcal{N}}
\newcommand{\km}{\mathfrak{m}}
\newcommand{\kv}{\mathfrak{v}}
\newcommand{\cpt}{\mathbf{K}}

\newcommand{\FP}{\operatorname{FP}}
\newcommand{\Mat}{\operatorname{Mat}}

\newcommand{\1}{{\bf 1}}

\begin{document}

\title[]
{Analytic torsion of arithmetic quotients of the symmetric space 
$\SL(n, \R)/\SO(n)$}
\date{\today}

\author{Jasmin Matz}
\address{The Hebrew University of Jerusalem\\
Einstein Institute of Mathematics}
\email{jasmin.matz@mail.huij.ac.il}

\author{Werner M\"uller}
\address{Universit\"at Bonn\\
Mathematisches Institut\\
Endenicher Allee 60\\
D -- 53115 Bonn, Germany}
\email{mueller@math.uni-bonn.de}


\begin{abstract}
In this paper we define a regularized version of the analytic torsion for 
arithmetic 
quotients of the symmetric space $\SL(n,\R)/\SO(n)$. The definition is based on
the study of the renormalized trace of the corresponding heat operators, which 
is defined as the geometric side of the Arthur trace formula applied to the 
heat operator. 

\end{abstract}

\maketitle
\setcounter{tocdepth}{1}
\tableofcontents

\section{Introduction}
In various papers \cite{BV}, \cite{MaM}, \cite{MP4} the Ray-Singer
analytic torsion \cite{RS} has been used to study the 
growth of torsion in the cohomology of cocompact arithmetic groups. Since many
important arithmetic groups are not cocompact, it is very desirable to extend
these results to the noncompact case. There exist some results for hyperbolic 
3-manifolds. In  \cite{PR}, Pfaff and Raimbault obtained upper and lower bounds
for the growth of torsion in the cohomology of congruence subgroups of Bianchi 
groups if the local system varies. 
In  \cite{Ra1}, \cite{Ra2}, J. Raimbault has studied the case of sequences
$(\Gamma_i)$ of congruence subgroups of Bianchi groups such that 
$\vol(\Gamma_i\bs\bH^3)\to\infty$ as $i\to\infty$. 

The approach in the cocompact case relies on the equality of analytic torsion
and Reidemeister torsion of the corresponding locally symmetric manifolds. We
briefly recall the definition of the Ray-Singer analytic torsion. 
Let $X$ be a compact Riemannian manifold of dimension $n$  and  
$\rho\colon \pi_1(X)\to\GL(V)$
a finite dimensional representation of its fundamental group. 
Let $E_\rho\to X$ be the flat vector bundle associated with $\rho$. Choose a
Hermitian fiber metric in $E_\rho$. Let $\Delta_p(\rho)$ be the Laplace 
operator on
$E_\rho$-valued $p$-forms with respect to the  metrics on $X$ and in $E_\rho$. 
It is
an elliptic differential operator, which is formally self-adjoint and 
non-negative. Let $h_p(\rho):=\dim\ker\Delta_p(\rho)$. Using the trace of the
heat operator $e^{-t\Delta_p(\rho)}$, the zeta function $\zeta_p(s;\rho)$ of 
$\Delta_p(\rho)$ can be defined by
\begin{equation}\label{zeta-fct}
\zeta_p(s;\rho):= \frac{1}{\Gamma(s)}\int_0^\infty 
\left(\Tr\left(e^{-t\Delta_p(\rho)}\right)-h_p(\rho)\right)t^{s-1} \;dt.
\end{equation} 
The integral converges for $\Re(s)>n/2$ and admits a meromorphic extension to
the whole complex plane, which is holomorphic at $s=0$. 
Then the Ray-Singer analytic torsion $T_X(\rho)\in\R^+$ is defined by
\begin{equation}\label{analtor0}
\log T_X(\rho)=\frac{1}{2}\sum_{p=1}^d (-1)^p p 
\frac{d}{ds}\zeta_p(s;\rho)\big|_{s=0}.
\end{equation}
The analytic torsion has a topological counterpart. This is
the Reidemeister torsion $\tau_X(\rho)$, which is defined in terms of a 
smooth triangulation of $X$ \cite{RS}, \cite{Mu5}. It is known that for 
unimodular representations $\rho$ (meaning that $|\det\rho(\gamma)|=1$ for
all $\gamma\in\pi_1(X)$) one has the equality $T_X(\rho)=\tau_X(\rho)$
\cite{Ch}, \cite{Mu4}, \cite{Mu5}. In the general case of a non-unimodular 
representation the equality does not hold, but the defect can be described 
\cite{BZ}. This equality  has the following interesting consequence. 
Assume that the space of the representation $\rho$
contains a lattice which is invariant under $\pi_1(X)$. Let $\cM$ be the 
associated local system of free $\Z$-modules. Let $H^p(X,\cM)_{\tors}$ be the
torsion subgroup of $H^p(X,\cM)$. Then
\begin{equation}\label{tor-coh}
T_X(\rho)= R\cdot\prod_{p=0}^d |H^p(X,\cM)_{\tors}|^{(-1)^{p+1}},
\end{equation}
where $R$ is the so called ``regulator'', defined in terms of the free part
of the cohomology $H^p(X,\cM)$ (see \cite{BV}, \cite{MP4}). In particular, if 
$\rho$ is acyclic, i.e., $H^\ast(X,E_\rho)=0$, then $R=1$. The equality 
\eqref{tor-coh} is the starting point for the application of the analytic 
torsion to the study of the torsion in the cohomology of cocompact arithmetic 
groups.

The definition of the analytic torsion \eqref{analtor0} obviously depends on
the compactness of the underlying manifold. Without this assumption, the heat 
operator $e^{-t\Delta_p(\rho)}$ is, in general, not a trace class operator.
 If one attempts 
to generalize the above method to non-cocompact arithmetic groups, the first 
problem is to define an appropriate regularized trace of the heat operators.
For hyperbolic manifolds of finite volume one can proceed as in 
Melrose \cite{Me} to define
the regularized trace by means of the renormalized trace of the heat kernel.
This method has been used in \cite{CV}, \cite{PR}, \cite{MP1}, 
\cite{MP3}, \cite{MP4}. One uses an appropriate height function to truncate 
the hyperbolic
manifold $X$ at height $T>0$. This amounts to cut off the cusps at sufficiently
high level $T>T_0$. Then one integrates the point wise trace of the heat 
kernel over
the truncated manifold $X(T)$. This integral has an asymptotic expansion in 
$\log T$. The constant term is defined to be the renormalized trace of the 
heat operator.

The purpose of the present paper is to start the investigation of the case of
finite volume locally symmetric spaces of any rank by defining a regularized
analytic torsion for arithmetic quotients associated to split forms of type
$A_n$ over $\Q$. In the higher rank case we proceed in the same 
way as in the case of hyperbolic manifolds. The first problem is to
define the truncation in the right way. For this we can build on Arthur's work.
The definition of the truncation operator is an important issue in Arthur's 
trace formula \cite{Ar1}, which we will use for our purpose. To this end we
need to switch to the adelic framework. 

Now we will describe the approach in more detail. For simplicity assume that
$G$ is a connected semisimple algebraic group defined over $\Q$. Assume that
$G(\R)$ is not compact. Let $K_\infty$ be a maximal compact subgroup of $G(\R)$.
Put $\widetilde X=G(\R)/K_\infty$. Let $\A$ be
the ring of adeles of $\Q$ and $\A_f$ the ring of finite adeles. Let 
$K_f\subset G(\A_f)$
be an open compact subgroup. We consider the adelic quotient
\begin{equation}
X(K_f)=G(\Q)\bs (\widetilde X\times G(\A_f))/K_f.
\end{equation}
This is the adelic version of a locally symmetric space. In fact, 
$X(K_f)$ is the disjoint union of finitely many locally symmetric spaces 
$\Gamma_i\bs\widetilde X$, $i=1,\dots,l$, (see section \ref{sec-arithm-mfd}). 
If $G$ is simply connected, then by strong approximation we 
\[
X(K_f)=\Gamma\bs \widetilde X,
\]
where $\Gamma=(G(\R)\times K_f)\cap G(\Q)$. We will assume that $K_f$ is neat, that is, the eigenvalues of any element in $\Gamma$ generate a torsion free subgroup in $\C^\times$, 
so that $X(K_f)$ is a manifold. Let $\nu\colon K_\infty\to \GL(V_\nu)$ be a 
finite dimensional unitary representation. It induces a homogeneous 
Hermitian vector bundle $\widetilde E_\nu$ over $\widetilde X$,  which is 
equipped with the canonical connection $\nabla^\nu$. Being homogeneous,
$\widetilde E_\nu$  can be pushed down to a locally homogeneous Hermitian 
vector bundle over each component $\Gamma_i\bs \widetilde X$ of $X(K_f)$. Their 
disjoint union is a Hermitian vector bundle $E_\nu$ over $X(K_f)$. 
Let $\widetilde\Delta_\nu$ (resp. $\Delta_\nu$) be the associated
Bochner-Laplace operator acting in the space of 
smooth section of $\widetilde E_\nu$ (resp. $E_\nu$).  
Let $e^{-t\widetilde\Delta_\nu}$ (resp. $e^{-t\Delta_\nu}$), $t>0$,  be the 
heat semigroup generated by $\widetilde\Delta_\nu$ (resp. $\Delta_\nu$). Since
$\widetilde \Delta_\nu$ commutes with the action of $G(\R)$, it follows that
$e^{-t\widetilde\Delta_\nu}$ is
a convolution operator with kernel given by a smooth map $H_t^\nu\colon 
G(\R)\to \End(V_\nu)$. Let $h_t^\nu(g)=\tr H_t^\nu(g)$, $g\in G(\R)$. In fact,
$h_t^\nu$ belongs to Harish-Chandra's Schwartz space $\Co(G(\R))$. Let
$\chi_{K_f}$ be the characteristic function of $K_f$ in $G(\A_f)$. We define
the function $\phi_t^\nu\in C^\infty(G(\A))$ by
\[
\phi_t^\nu(g_\infty g_f)=h_t^\nu(g_\infty)\chi_{K_f}(g_f),\quad g_\infty\in G(\R),\;
g_f\in G(\A_f).
\]
In fact, $\phi_t^\nu$ belongs to $\Co(G(\A);K_f)$, the adelic version of the 
Schwartz space (see section \ref{sec-trunc} for its definition). 
If $X(K_f)$ is compact, then one has
\begin{equation}\label{trcocom}
\Tr\left(e^{-t\Delta_\nu}\right)=\int_{G(\Q)\bs G(\A)}\sum_{\gamma\in G(\Q)} 
\phi_t^\nu(x^{-1}\gamma x)\;dx.
\end{equation}
This is our starting point for defining the renormalized trace in the noncompact
case. We fix a minimal Levi subgroup $M_0$ of $G$. If $M\subseteq G$ is a Levi subgroup containing $M_0$, let $A_M$ be the split 
component of the center of $M$. Let $\af_0:=\af_{M_0}$ be the Lie algebra of
$A_{M_0}(\R)$. Let $J_{\geo}$ be the geometric side of the Arthur trace 
formula introduced in
\cite{Ar1}; see also \cite{Ar} for an introduction to the trace formula. For $f\in C_c^\infty(G(\A))$, Arthur defines $J_{\geo}(f)$ as the 
value at a point $T_0\in\af_0$, specified in \cite[Lemma 1.1]{Ar3}, 
of a polynomial $J^T(f)$ on $\af_0$.  In fact,
by \cite[Theorem 7.1]{FL1}, $J^T(f)$ is defined for all $f\in\Co(G(\A);K_f)$.
Furthermore,  we use an appropriate height function to 
truncate $G(\A)$.  For $T\in\af_0$ let $G(\A)_{\le T}$ be obtained by
truncating $G(\A)$ at level $T$ (see \eqref{trunc1}). This is a compact subset
of $G(\A)$. By \cite[Theorem 7.1]{FL1} it follows that for sufficiently regular
$T\in\af_0$ we have
\begin{equation}\label{truncker}
\int_{G(\Q)\bs G(\A)_{\le T}}\sum_{\gamma\in G(\Q)} \phi_t^\nu(x^{-1}\gamma x)\;dx =
J^T(\phi_t^\nu)+O\left(e^{-c\|T\|}\right).
\end{equation}
Since $J^T(\phi_t^\nu)$ is a polynomial in $T$, we get an asymptotic expansion 
in $T$ of the truncated integral. Under additional assumption on $G$, which 
are satisfied for $\GL(n)$ and $\SL(n)$, the point $T_0\in
\af_0$, determined by \cite[Lemma 1.1]{Ar1}, is equal to 0. Thus in this case
$J_{\geo}(\phi_t^\nu)$ is the constant term of the polynomial $J^T(\phi_t^\nu)$. 
This leads to our definition of the regularized trace
\begin{equation}
\Tr_{\reg}\left(e^{-t\Delta_\nu}\right):=J_{\geo}(\phi_t^\nu).
\end{equation}
In general, $J_{\geo}(\phi_t^\nu)$ is not the constant term of the polynomial
$J^T(\phi_t^\nu)$. Nevertheless, we prefer this definition, because of its
independence on the choice of the minimal parabolic subgroup $P_0$.

The next goal is to determine the asymptotic behavior of 
$\Tr_{\reg}\left(e^{-t\Delta_\nu}\right)$ as
$t\to 0$ and $t\to\infty$, respectively. To this end we use the Arthur trace 
formula.  Currently we are only able to deal with these problems
for the groups $G=\GL(n)$ or $G=\SL(n)$. For $N\in\N$ let $K(N)\subset G(\A_f)$
be the principal congruence subgroup of level $N$. Recall that $K(N)$ is neat
for $N\ge 3$. Our first main result is the following proposition. 
\begin{theo}\label{prop-asymp3}
Let $G=\GL(n)$ or $\SL(n)$. Let $K_f\subset G(\A_f)$ be an open compact 
subgroup. Assume that $K_f$ is contained in $K(N)$ for some $N\ge 3$. 
 Let $\nu$ be finite dimensional unitary representation of $K_\infty$
and let $\Delta_\nu$ be the associated Bochner-Laplace operator.
Let $d=\dim X(K_f)$. As $t\to +0$, there is an asymptotic expansion 
\begin{equation}\label{asex}
\Tr_{\reg}\left(e^{-t\Delta_\nu}\right)\sim t^{-d/2}\sum_{j=0}^\infty a_j(\nu) t^j
+t^{-(d-1)/2}\sum_{j=0}^\infty\sum_{i=0}^{r_j} b_{ij}(\nu)t^{j/2}(\log t)^i.
\end{equation}
Moreover $r_j\le n-1$ for all $j\in\N_0$.
\end{theo}

For hyperbolic manifolds a similar result was proved in \cite{Mu6}.

To study the large time behavior we restrict attention to twisted Laplace
operators, which are relevant for studying the analytic torsion with coefficients in local systems. 
Let $\tau\colon G(\R)\to \GL(V_\tau)$ be a finite dimensional complex 
representation. Let $\Gamma_i\bs \widetilde X$, $i=1,\dots,l$,  be the
components of $X(K_f)$. The restriction of $\tau$ to $\Gamma_i$ induces a flat 
vector bundle $E_{\tau,i}$ over $\Gamma_i\bs \widetilde X$. The disjoint union is
 a flat vector bundle $E_\tau$ over $X(K_f)$. By \cite{MM} it is isomorphic to 
the locally homogeneous vector bundle associated to $\tau|_{K_\infty}$. It 
can be equipped with a fiber metric induced from the homogeneous 
bundle. Let $\Delta_p(\tau)$ be the corresponding twisted Laplace operator on
$p$-forms with values in $E_\tau$. Let $\Ad_{\pf}\colon K_\infty\to \GL(\pf)$ 
be the adjoint representation of $K_\infty$ on $\pf$, where $\pf=\kf^\perp$, and 
$\nu_p(\tau)=\Lambda^p\Ad_{\pf}^\ast\otimes\tau$. 
Up to a vector bundle endomorphism, $\Delta_p(\tau)$ equals the 
Bochner-Laplace operator $\Delta_{\nu_p(\tau)}$. So 
$\Tr_{\reg}\left(e^{-t\Delta_p(\tau)}\right)$ is well defined. Let $\theta$
be the Cartan involution of $G(\R)$ with respect to $K_\infty$. Put $\tau_\theta:=
\tau\circ\theta$. The large time 
behavior of the regularized trace is described by the following proposition.

\begin{theo}\label{theo-lt}
Let  $G=\GL(n)$ or $\SL(n)$. Let $K_f\subset G(\A_f)$ be an open compact 
subgroup which is contained in $K(N)$ for some $N\ge 3$. 
Let $\tau$ be finite dimensional representation
of $G(\R)$. Assume that $\tau\not\cong\tau_\theta$. Then we have
\begin{equation}\label{largetime}
\Tr_{\reg}\left(e^{-t\Delta_p(\tau)}\right)=O(e^{-ct})
\end{equation}
as $t\to\infty$ for all $p=0,\dots,d$. 
\end{theo} 
The proof is an immediate consequence of Proposition \ref{asympinf} together
with the trace formula. Without the assumption $\tau\not\cong\tau_\theta$ the
behavior of $\Tr_{\reg}\left(e^{-t\Delta_p(\tau)}\right)$ as $t\to\infty$ is more
complicated and it is definitely not exponentially decreasing. This condition
is also relevant in  \cite{BV}. It implies that the representation $\tau$ is
strongly acyclic \cite[Lemma 4.1]{BV}, which is a necessary condition to
establish the main results of \cite{BV}. It is a very challenging problem
to eliminate this condition. We also
note that the condition $\tau\not\cong\tau_\theta$ 
implies the vanishing theorem of Borel-Wallach for
the cohomology of a cocompact lattice in a semisimple Lie group 
\cite[Theorem 6.7, Ch. VII]{BW}.

By Theorems \ref{prop-asymp3} and \ref{theo-lt} we can define the zeta
function of $\Delta_p(\tau)$ as in \eqref{zetafct}, using the regularized
trace of $e^{-t\Delta_p(\tau)}$ in place of the usual trace. The corresponding 
Mellin transform converges absolutely
and uniformly on compact subsets of the half-plane $\Re(s)>d/2$ and admits
a meromorphic extension to the whole complex plane. Because of the presence
of the log-terms in the expansion \eqref{asex}, the zeta function may have a
pole at $s=0$. Let $f(s)$ be a meromorphic function on $\C$. For $s_0\in\C$ 
let $f(s)=\sum_{k\ge k_0}a_k(s-s_0)^k$
be the Laurent expansion of $f$ at $s_0$. Put $\FP_{s=s_0}f(s):=a_0$. Now we define
the analytic torsion $T_{X(K_f)}(\tau)\in\C\setminus\{0\}$ by
\begin{equation}\label{analtor}
\log T_{X(K_f)}(\tau)=\frac{1}{2}\sum_{p=0}^d (-1)^p p 
\left(\FP_{s=0}\frac{\zeta_p(s;\tau)}{s}\right).
\end{equation}
In the case of $G=\GL(3)$ we are able to determine the coefficients of the 
log-terms. This shows that the 
zeta functions definitely have a pole at $s=0$. However, the combination
$\sum_{p=1}^5 (-1)^p p \zeta_p(s;\tau)$ turns out to be holomorphic at $s=0$ and
we can define the logarithm of the analytic torsion by
\[
\log T_{X(K_f)}(\tau)=\frac{d}{ds}
\left(\frac{1}{2}\sum_{p=1}^5 (-1)^p p \zeta_p(s;\tau)\right)\bigg|_{s=0}.
\]
Let $\{K_f(N)\}_{N\in\N}$  be the family of principal congruence subgroups
of $\GL(n,\A_f)$, and $X(N):=X(K_f(N))$, $N\in\N$. The next problem is to 
study the limiting behavior of $\log T_{X(N)}(\tau)/\vol(X(N))$ as $N\to\infty$ which we do in subsequent work. 
In consideration of the results for the cocompact case in \cite{BV}, one can expect a different behavior of $\log T_{X(N)}(\tau)/\vol(X(N))$ in the limit $N\to\infty$ for different $n$. More precisely, the fundamental rank $\rk G(\R)- \rk K_\infty$ determined in \cite{BV} whether the limit vanishes, which it does unless if the rank equals $1$. In our case of $\SL_n(\R)$, the fundamental rank is $1$ precisely when $n=3$ or $n=4$.

An even more difficult problem is the question if there is a combinatorial
counterpart of $T_{X(K_f)}(\tau)$ as there is in the compact case. 

Now we briefly explain our method to prove Theorems \ref{prop-asymp3}
and \ref{theo-lt}. To determine the asymptotic 
behavior of the regularized trace as $t\to +0$, we use the geometric side of 
trace formula. The first step is to show that $\phi_t^\nu$ can be replaced by 
a  compactly supported function 
$\widetilde \phi_t^\nu\in C_c^\infty(G(\A))$ without changing the asymptotic
behavior. Next we use the coarse geometric expansion 
of the geometric side, which expresses $J_{\geo}(f)$, $f\in C_c^\infty(G(\A))$,
 as a sum of distributions $J_{\of}(f)$ associated to semisimple conjugacy
classes of $G(\Q)$. Let $J_{\unip}(f)$ be the distribution associated to the 
class of $1$.  If the support of $\widetilde \phi_t^\nu$ is a sufficiently 
small neighborhood of $1$, it follows that
\begin{equation}
\Tr_{\reg}\left(e^{-t\Delta_\nu}\right):=J_{\unip}(\widetilde\phi_t^\nu)+
O\left(e^{-c/t}\right)
\end{equation}
as $t\to +0$. To analyze $J_{\unip}(\widetilde\phi_t^\nu)$, we use the fine
geometric expansion \cite{Ar4} which expresses 
$J_{\unip}(\widetilde\phi_t^\nu)$ in terms
of weighted orbital integrals. If the  real rank of $G(\R)$  is one, the 
weighted orbital integrals are rather simple and the weight factors are 
explicitly known (see \cite{Wa}). 
In order to deal with the weighted orbital integrals in the higher rank case,
we need to restrict to the groups $\GL(n)$ or $\SL(n)$. In this case all
unipotent orbits are Richardson, which simplifies the analysis considerably. 
We are only interested in the situation over the field $\R$. Let $M$ be a
Levi subgroup of $G$. Let $\cU_M$ be the unipotent variety in $M$ and 
$\cV\in (\cU_M)$ a conjugacy class. Let $U$ be an $M(\R)$ conjugacy class
in $\cV(\R)$. There exists a standard parabolic subgroup $Q=LN\in\cF$ and
a constant $c>0$ such 
that for every $\rO(n)$-conjugation invariant function $f\in C_c^\infty(G(\R))$
the weighted orbital integral $J_M(U,f)$ is given by 
\begin{equation}\label{orbint0}
J_M(U,f)=c\int_{N(\R)}f(n) w_{M,\cV}(n)\;dn,
\end{equation}
where $w_{M,\cV}(n)$ is a certain weight function. The main problem is now to determine
the structure of the weight function. For $G(\R)=\SO_0(n,1)$ the weighted
orbital integral is of the same form with weight function $w(n)=\log\|\log n\|$,
where the inner log is the isomorphism $\log \colon N\to\nf$. 
This fact has been exploited in \cite{Mu6} in order to establish the asymptotic
expansion of the regularized trace in the case of hyperbolic manifolds of 
finite volume. It turns out that $w_{M,\cV}$ has a similar
behavior with respect to scaling. Note that the map $x\mapsto X=x-\Mid$ defines
a bijection between the variety of unipotent elements in $G(\R)$ and the
nilpotent cone in the Lie algebra $\kg(\R)$. For $s\in\R$ let 
$x_s:=\Mid+s(x-\Mid)$. Let $x\in\cV^G(\R)$ such that $w_{M,\cV}(x)$ is defined.
Then by Proposition \ref{prop}, $w_{M,\cV}(x_s)$ is well-defined for every
$s>0$ and $s\mapsto w_{M,\cV}(x_s)$ is a polynomial in 
$\log s$ of degree at most $\dim\af_M^G$. Inserting a standard parametrix for
the heat kernel into \eqref{orbint0} and using the structure of $w_{M,\cV}$,
we obtain Theorem \ref{prop-asymp3}. To eliminate the assumption that $K_f$ 
is contained in some $K(N)$ with $N\ge 3$, we would have to consider orbital 
integrals associated to classes of finite order. For $\GL(2)$ and $\GL(3)$ we 
discuss this issue in section \ref{sec-finite-ord}. 

To prove Theorem \ref{theo-lt}, we use the spectral side of the trace formula.
Let $\phi_t^{\tau,p}$ be the function in $\Co(G(\A);K_f)$, which is defined in 
the same way
as $\phi_t^\nu$ in terms of the kernel of the heat operator on the universal
covering. Then by the trace formula
\[
\Tr_{\reg}\left(e^{-t\Delta_p(\tau)}\right)=J_{\spec}(\phi^{\tau,p}_t).
\]
The key input to deal with the spectral side is the refinement of the spectral
 expansion of the Arthur trace formula established in \cite{FLM1}
(see Theorem \ref{thm-specexpand}). 
For $f\in\Co(G(\A))$ we have
\[
J_{\spec}(f)=\sum_{[M]}J_{\spec,M}(f),
\]
where $[M]$ runs over the conjugacy classes of Levi subgroups of $G$ and 
$J_{\spec,M}(f)$ is a distribution associated to $M$. The distribution 
associated to $G$ is $\Tr R_{\di}(f)$, where
$R_{\di}$ denotes the restriction of the regular representation of $G(\A)$ 
in $L^2(G(\Q)\bs G(\A))$ to the discrete subspace. 
For a proper Levi subgroup $M$ of $G$, $J_{\spec,M}(f)$ is an integral whose
main ingredient are logarithmic derivatives of intertwining operators. Using
our assumption that $\tau\neq\tau_\theta$, we obtain
$\dim\ker\Delta_p(\tau)=0$. Then it follows as in the compact case
that there exists $c>0$ such that 
\[
\Tr R_{\di}(\phi^{\tau,p}_t)=O(e^{-ct}), \quad\text{as}\;\; t\to\infty.
\]
For a proper Levi subgroup $M$, the determination of the asymptotic behavior 
of $J_{\spec,M}(\phi^{\tau,p}_t)$ as $t\to\infty$ 
relies on two conjectural properties, one global and one local, of the
intertwining operators. The global property is a uniform 
estimate on the winding number of the normalizing factors of the intertwining
operators in the co-rank one case. For $\GL(n)$ and $\SL(n)$, this property
follows from known, but delicate, properties of the Rankin-Selberg $L$-functions
\cite{FLM2}. The local property is concerned with the estimation 
of logarithmic derivatives of normalized local intertwining operators, which
are uniform in $\pi$. For $\GL(n)$ the pertinent estimates have been
established in \cite[Proposition 0.2]{MS}. They are a consequence of a weak
version of the Ramanujan conjecture. The case of $\SL(n)$ can be reduced to
$\GL(n)$ in the same way as in the proof of \cite[Lemma 5.14]{FLM2}. Let
$\theta\colon G\to G$ be the Cartan involution and let $\tau_\theta:=\tau\circ
\theta$. Using these estimations, it follows that for 
$G=\GL(n)$ or $G=\SL(n)$, a proper Levi 
subgroup $M$ of $G$ and a finite dimensional representation $\tau$ of $G(\R)$
such that $\tau\not\cong\tau_\theta$, one has $J_{\spec,M}(\phi^{\tau,p}_t)=
O(e^{-ct})$ as $t\to\infty$. Putting everything together, we obtain Theorem
\ref{theo-lt}.

We end this introduction with some remarks on the possible extension of the 
our results to other groups $G$. 
First of all, Theorem \ref{theo-lt} depends on the estimations of 
logarithmic derivatives of global normalizing factors and normalized local  
intertwining operators. Using functoriality, T. Finis and E. Lapid \cite{FL2}
have recently established similar estimates of the logarithmic derivatives of
global normalizing factors associated to intertwining operators for the 
following reductive groups over number fields: inner forms of $\GL(n)$, 
quasi-split classical groups and their similitude groups, and the exceptional
groups $G_2$. One can expect that the estimates of the logarithmic derivatives 
of the normalized local intertwining operators can be established by the same
methods. This would lead to an extension of Theorem \ref{theo-lt}
to these groups. It remains to deal with the unipotent orbital integrals for 
the groups above.

The paper is organized as follows. In section \ref{sec-prelim} we fix notations 
and recall some basic facts. In section \ref{sec-arithm-mfd} we introduce the
locally symmetric manifolds as adelic quotients. In section \ref{sec-trunc} 
we compare two different methods of truncation. One of them is based on the
truncation of kernels of integral operators which leads to the geometric side 
of the trace formula. The other one consists in the truncation of the underlying
manifold, which is the basis for the renormalization of the trace of the heat
operator.    
 In section \ref{sec-trform} we recall the spectral side of the Arthur
trace formula. In section \ref{sec-unipotent} we are assuming that $G=\GL(n)$ or
$G=\SL(n)$. We discuss the unipotent contribution to the trace formula and 
derive a simplified formula for the weighted orbital integral. Section 
\ref{sec-wfct} is devoted to the study of the weight functions for the
groups $\GL(n)$ and $\SL(n)$. The main result is Proposition \ref{prop}, which
is the key result that enables us to determine the asymptotic behavior as
$t\to +0$ of the corresponding orbital integrals. Examples of low rank are
discussed in section \ref{sec-lowrank}. These are cases where the weight 
function is given explicitly. In section \ref{sec-bochlapl} we collect some
basic facts concerning Bochner-Laplace operators. The regularized trace of 
the corresponding heat operators is introduced in section \ref{sec-regtr}. 
The definition is based on section \ref{sec-trunc}, which deals with 
truncation. In section \ref{sec-heatkernel} we establish some estimates of 
the heat kernel for Bochner-Laplace operators on the symmetric space
$\widetilde X$. Combined with the analysis of the weight functions in section
\eqref{sec-wfct}, the estimations are used in section \eqref{sec-asymp} to
prove Theorem \ref{prop-asymp3}. In section \eqref{sec-analtor} we first
use the spectral side of the Arthur trace formula to establish Theorem 
\ref{theo-lt}, which concerns the large time asymptotic behavior of the 
regularized trace of the heat operators. This finally enables us to define
the regularized analytic torsion. In section \ref{sect-gl3} we 
assume that $G=\GL(3)$. Using the explicit form of the weight functions
described in section \ref{sec-lowrank}, we determine the coefficients of the 
possible poles at $s=0$ of the zeta functions. It turns out that the 
combination of the zeta functions, which is used to define the analytic
torsion, is holomorphic at $s=0$. In the final section \ref{sec-finite-ord}
we consider for $G=\GL(2)$ or $G=\GL(3)$ an arbitrary subgroup $K_f$ of
$G(\hat\Z)$ and study the additional weighted orbital integrals that arise 
in this case.

\noindent
{\bf Acknowledgment.}  The authors  would like to thank the
referees for the careful reading of the manuscript and for their very helpful 
suggestions and comments.

\section{Preliminaries}\label{sec-prelim}

Let $G$ be a reductive algebraic group defined
over $\Q$. We fix a minimal parabolic subgroup $P_0$ of $G$ 
defined over $\Q$ and a Levi decomposition $P_0=M_0\cdot N_0$, both defined 
over $\Q$.  Let $\cF$ be the set of parabolic subgroups of $G$ which contain 
$M_0$ and are defined over $\Q$. Let $\cL$ be the set of subgroups of $G$ 
which contain $M_0$ and are Levi components of groups in $\cF$. 
For any $P\in\cF$ we write
\[
P=M_PN_P,
\]
where $N_P$ is the unipotent radical of $P$ and $M_P$ belongs to $\cL$. 

Let $M\in\cL$. Denote by $A_M$ the $\Q$-split component of the center of $M$. 
Put $A_P=A_{M_P}$. Let $L\in\cL$ and assume that $L$ contains $M$. Then $L$ is
a reductive group defined over $\Q$ and $M$ is a Levi subgroup of $L$. We 
shall denote the set of Levi subgroups of $L$ which contain $M$ by $\cL^L(M)$.
We also write $\cF^L(M)$ for the set of parabolic subgroups of $L$, defined 
over $\Q$, which contain $M$, and $\cP^L(M)$ for the set of groups in $\cF^L(M)$
for which $M$ is a Levi component. Each of these three sets is finite. If 
$L=G$, we shall usually denote these sets by $\cL(M)$, $\cF(M)$ and $\cP(M)$.

Let $X(M)_\Q$ be the group of characters of $M$ which are defined over $\Q$. 
Put
\begin{equation}\label{liealg}
\af_{M}:=\Hom(X(M)_\Q,\R).
\end{equation}
This is a real vector space whose dimension equals that of $A_M$. Its dual 
space is
\[
\af_{M}^\ast=X(M)_\Q\otimes \R.
\]
 We shall write, 
\begin{equation}\label{liealg1}
\af_P=\af_{M_P},\;A_0=A_{M_0}\quad\text{and}\quad \af_0=\af_{M_0}.
\end{equation}
For $M\in\cL$ let $A_M(\R)^0$ be the connected component of the identity of
the group $A_M(\R)$. 
Let $W_0=N_{\G(\Q)}(A_0)/M_0$ be the Weyl group of $(G,A_0)$,
where $N_{G(\Q)}(H)$ is the normalizer of $H$ in $G(\Q)$.
For any $s\in W_0$ we choose a representative $w_s\in G(\Q)$.
Note that $W_0$ acts on $\levis$ by $sM=w_s M w_s^{-1}$. For $M\in\cL$ let
$W(M)=N_{\G(\Q)}(M)/M$, which can be identified with a subgroup of $W_0$.

For any $L\in\cL(M)$ we identify $\af_L^\ast$ with a subspace of $\af_M^\ast$.
We denote by $\af_M^L$ the annihilator of $\af_L^\ast$ in $\af_M$. 
We set
\[
\levis_1(M)=\{L\in\levis(M):\dim\aaa_M^L=1\}
\]
and
\begin{equation}\label{f1}
\cF_1(M)=\bigcup_{L\in\levis_1(M)}\cP(L).
\end{equation}
We shall denote the simple roots of $(P,A_P)$ by $\Delta_P$. They are
elements of $X(A_P)_\Q$ and are canonically embedded in $\af_P^\ast$. 
Let
$\Sigma_P\subset \af_P^\ast$ be the set of reduced roots of $A_P$ on the Lie
algebra of $G$. The
set $\Delta_0=\Delta_{P_0}$ is a base for a root system. In particular, for 
every $\alpha\in\Delta_P$ we have a co-root $\alpha^\vee\in\af_{P_0}$.

Let $P_1$ and $P_2$ be parabolic subgroups with
$P_1\subset P_2$. Then $\af_{P_2}^\ast$ is embedded into $\af_{P_1}^\ast$, while
$\af_{P_2}$ is a natural quotient vector space of $\af_{P_1}$. The group
$M_{P_2}\cap P_1$ is a parabolic subgroup of $M_{P_2}$. Let $\Delta_{P_1}^{P_2}$
denote the set of simple roots of $(M_{P_2}\cap P_1,A_{P_1})$. It is a subset
of $\Delta_{P_1}$. For a parabolic subgroup $P$ with $P_0\subset P$ we write
$\Delta_0^P:=\Delta_{P_0}^P$. 

Let $\A$ (resp. $\A_f$) be the ring of adeles (resp. finite adeles) of $\Q$. 
We fix a maximal compact subgroup $\K=\prod_v K_v = K_\infty\cdot 
\K_{f}$ of $G(\A)=G(\R)\cdot G(\A_{f})$. We assume that the maximal 
compact subgroup $\K \subset G(\A)$ is admissible with respect to 
$M_0$ \cite[\S 1]{Ar5}. Let $\Ht_M: M(\A)\rightarrow\aaa_M$ be the 
homomorphism given by
\begin{equation}\label{homo-M}
e^{\sprod{\chi}{\Ht_M(m)}}=\abs{\chi (m)}_\A = \prod_v\abs{\chi(m_v)}_v
\end{equation}
for any $\chi\in X(M)$ and denote by $M(\A)^1 \subset M(\A)$ the kernel 
of $\Ht_M$. Then $M(\A)$ is the direct product of $M(\A)^1$ and $A_M(\R)^0$, the
component of $1$ in $A_M(\R)$. By the conditions on $\K$, we have $G(\A)=P(\A)
\K$. Hence any $x\in G(\A)$ can be written as
\[
nmak,\quad n\in N(\A),\;m\in M(\A)^1,\; a\in A_M(\R)^0,\; k\in\K.
\]
Define $H_P\colon G(\A)\to\af_P$ by
\begin{equation}\label{hp}
H_P(x):=H_M(a),
\end{equation}
where $x=nmak$ as above.
Let $\gf$ and $\kf$ denote the Lie algebras of $G(\R)$ and $K_\infty$,
respectively. Let $\theta$ be the Cartan involution of $G(\R)$ with respect to
$K_\infty$. It induces a Cartan decomposition $\mathfrak{g}= 
\mathfrak{p} \oplus \mathfrak{k}$. 
We fix an invariant bi-linear form $B$ on $\mathfrak{g}$ which is positive 
definite on $\mathfrak{p}$ and negative definite on $\mathfrak{k}$.
This choice defines a Casimir operator $\Omega$ on $G(\R)$. Let $\Pi(G(\R))$ denote the set of equivalence classes of irreducible unitary representations of $G(\R)$. 
We denote the Casimir eigenvalue of any $\pi \in \Pi (G(\R))$ by 
$\lambda_\pi$. Similarly, we obtain
a Casimir operator $\Omega_{K_\infty}$ on $K_\infty$ and write $\lambda_\tau$ for 
the Casimir eigenvalue of a
representation $\tau \in \Pi (K_\infty)$ (cf.~\cite[\S 2.3]{BG}).
The form $B$ induces a Euclidean scalar product $(X,Y) = - B (X,\theta(Y))$ on 
$\mathfrak{g}$ and all its subspaces.
For $\tau \in \Pi (K_\infty)$ we define $\norm{\tau}$ as in 
\cite[\S 2.2]{CD}. Note that the restriction of the scalar product 
$(\cdot,\cdot)$ on $\gf$ to $\af_0$ gives $\af_0$ the structure of a 
Euclidean space. In particular, this fixes Haar measures on the spaces 
$\af_M^L$ and their duals $(\af_M^L)^\ast$. We follow Arthur in the 
corresponding normalization of Haar measures on the groups $M(A)$ 
(\cite[\S 1]{Ar1}). 

Let $L^2_{\disc}(\Ai M(\Q)\bs M(\A))$ be the discrete part of 
$L^2(\Ai M(\Q)\bs M(\A))$, i.e., the
closure of the sum of all irreducible subrepresentations of the regular 
representation of $M(\A)$.
We denote by $\Pi_{\disc}(M(\A))$ the countable set of equivalence classes of 
irreducible unitary
representations of $M(\A)$ which occur in the decomposition of the discrete 
subspace $L^2_{\disc}(\Ai M(\Q)\bs M(\A))$ into irreducible representations.

\section{Arithmetic manifolds}\label{sec-arithm-mfd}
Let $G$ be a reductive algebraic group over $\Q$.  Let 
$K_f\subset G(\A_f)$ be an open compact 
subgroup. The double coset space $\Ag G(\Q)\bs G(\A)/G(\R)K_f$ is known to
be finite (see \cite[\S 5]{Bo1}). Let $x_1=1, x_2,\dots,x_l$ be a set of 
representatives in $G(\A_f)$ of the double cosets. Then the groups
\[
\Gamma_i:=\left( G(\R)\times x_i K_f x_i^{-1}\right)\cap G(\Q),\quad 1\le i\le l,
\]
are arithmetic subgroups of $G(\R)$ and the action of $G(\R)$ on the space of 
double
cosets $\Ag G(\Q)\bs G(\A)/K_f$ induces the following decomposition into
$G(\R)$-orbits:
\begin{equation}\label{adel-quot1}
\Ag G(\Q)\bs G(\A)/K_f\cong \bigsqcup_{i=1}^l
\left(\Gamma_i\bs G(\R)^1\right),
\end{equation}
where $G(\R)^1=G(\R)/\Ag$. Thus we get an isomorphism of $G(\R)$-modules
\begin{equation}\label{g-modules}
L^2(\Ag G(\Q)\bs G(\A))^{K_f}\cong \bigoplus_{i=1}^lL^2(\Gamma_i\bs G(\R)^1).
\end{equation}
We note that, in general, $l>1$. However, if
$G$  is semisimple,
simply connected, and without any $\Q$-simple factors $H$ for which $H(\R)$
is compact, then by strong approximation we have 
\[
G(\Q)\bs G(\A)/K_f\cong \Gamma\bs G(\R),
\] 
where $\Gamma=(G(\R)\times K_f)\cap G(\Q)$. In particular this is the case for
$G=\SL(n)$. Let $K_\infty\subset G(\R)$ be a maximal compact subgroup. Let
\begin{equation}\label{symspace}
\widetilde X:=G(\R)^1/K_\infty
\end{equation}
be the associated global Riemannian symmetric space.
Given an open compact subgroup $K_f\subset G(\A_f)$, we  define the
arithmetic manifold $X(K_f)$ by
\begin{equation}\label{adel-quot2}
X(K_f):= G(\Q)\bs (\widetilde X \times G(\A_f))/K_f.
\end{equation}
By \eqref{adel-quot1} we have
\begin{equation}
X(K_f)=\bigsqcup_{i=1}^l \left(\Gamma_i\bs \widetilde X\right),
\end{equation}
where each component $\Gamma_i\bs \widetilde X$ is a locally symmetric space.
We will assume that $K_f$ is neat. Then $X(K_f)$ is a locally symmetric
manifold of finite volume. 

Now consider $G=\GL(n)$ as algebraic group over $\Q$. Then $\Ag$ is the
group of scalar matrices with a positive real scalar and $K_\infty=\rO(n)$. 
Let $N=\prod_pp^{r_p}$, 
$r_p\ge 0$. Put
\[
K_p(N):=\{k\in G(\Z_p)\colon k\equiv 1\;\text{mod}\;p^{r_p}\Z_p\}
\]
and
\[
K(N):=\prod_{p<\infty}K_p(N).
\]
Then $K(N)$ is an open compact subgroup of $G(\A_f)$ and
\begin{equation}\label{gln-iso}
\Ag G(\Q)\bs G(\A)/K(N)\cong \bigsqcup_{i=1}^{\varphi(N)}\Gamma(N)\bs SL(n,\R)
\end{equation}
where $\varphi(N)=\#[(\Z/N\Z)^\ast]$ (see \cite{Ar6}). Hence we have
\begin{equation}\label{l2-gln}
L^2(\Ag G(\Q)\bs G(\A))^{K(N)}\cong \bigoplus_{i=1}^{\varphi(N)} 
L^2(\Gamma(N)\bs \SL(n,\R))
\end{equation}
as $\SL(n,\R)$-modules. We have
\[
\widetilde X=\SL(n,\R)/\SO(n). 
\]
Let
\begin{equation}
X(N):=G(\Q)\bs (\widetilde X\times G(\A_f))/K(N).
\end{equation}
Let $\nu\colon K_\infty\to\GL(V_\nu)$ be a finite dimensional unitary 
representation of $K_\infty$. Let $\widetilde E_\nu$ be the associated 
homogeneous Hermitian vector  bundle over $\widetilde X$. Over each component 
of $X(K_f)$, $\widetilde E$ induces a locally homogeneous Hermitian vector 
bundle $E_{i,\nu}\to\Gamma_i\bs\widetilde X$.  Let
\begin{equation}\label{vectbdl}
E_\nu:=\bigsqcup_{i=1}^l E_{i,\nu}.
\end{equation}
Then $E_\nu$ is a vector bundle over $X(K_f)$, which is locally homogeneous.

\section{Truncation and the geometric side of the trace formula}
\label{sec-trunc}

The Arthur trace formula is obtained by truncating the kernels of integral 
operators associated to functions in $C^\infty_c(G(\A)^1)$.  On the other hand, 
the regularization of the trace of heat operators is based on the truncation
of the underlying locally symmetric space. In this section we compare the two 
methods. 
Let $P_0$ be the fixed minimal parabolic subgroup of $G$. 

For $f\in C^\infty_c(G(\A)^1)$ let
\[
K_f(x,y)=\sum_{\gamma\in G(\Q)}f(x^{-1}\gamma y).
\]
This is the kernel of an integral operator. In general, $K_f(x,x)$ is not
integrable over $G(\Q)\bs G(\A)^1$ and needs to be truncated to get an
integrable function. To define the truncated kernel we need to introduce some 
notations.

Let $P=M_PN_P$ be a standard parabolic subgroup
and let $Q$ be a parabolic subgroup containing $P$. Let $\Delta^Q_P$ be the set
of simple roots of $(M_Q\cap P,A_P)$. 
 Similarly, we have the set of coroots $\Delta_0^\vee$ and, more
generally and, the set $(\Delta_P^Q)^\vee$ which forms a basis of
$\af_P^Q:=\af_P\cap \af_0^Q$. We denote the basis of $(\af_P^Q)^\ast$ (resp. 
$\af_P^Q$) dual to $(\Delta_P^Q)^\vee$ (resp. $\Delta_P^Q$) by $\hat\Delta_P^Q$
(resp. $(\hat\Delta_P^Q)^\vee$. Let $\tau_P^Q$ and $\widehat\tau_P^Q$ denote the 
characteristic functions of the set
\[
\{X\in\af_0\colon \langle\alpha,X\rangle>0\;\text{for}\;\text{all}\;
\alpha\in\Delta_P^Q\}
\]
and
\[
\{X\in\af_0\colon \langle\varpi,X\rangle>0\;\text{for}\;\text{all}
\;\varpi\in\hat\Delta_P^Q\},
\]
respectively. If $Q=G$, we will suppress the superscript. Moreover we 
put $\tau_0:=\tau_0^G$ and $\hat\tau_0:=\widehat\tau_0^G$. Now we can define
the truncated kernel. Put
\[
K^P_f(x,y):=\int_{N_P(\Q)\bs N_P(\A)}\sum_{\gamma\in P(\Q)} f(x^{-1}\gamma ny)\;dn.
\]
Let $H_P\colon G(\A)\to\af_P$ be the map defined by \eqref{hp}. For any 
$T\in \af_0^+$  define
\begin{equation}\label{truncker1}
k^T(x,f):=\sum_P (-1)^{\dim(A_P/A_G)}\sum_{\delta\in P(\Q)\bs G(\Q)}
K^P_f(\delta x,\delta x)\widehat\tau_P(H_P(\delta x)-T_P),
\end{equation}
where $T_P$ denotes the projection of $T$ on $\af_P$.
Note that the term in \eqref{truncker1} which corresponds to $P=G$ is 
$K_f(x,x)$. If $G(\Q)\bs G(\A)^1$ is compact, there are no proper parabolic
subgroups of $G$ over $\Q$. Thus, in this case we have $k^T(x,f)=K_f(x,x)$, and
the truncation operation is trivial. By \cite[Theorem 6.1]{Ar4} the integral
\begin{equation}\label{trf1}
J^T(f):=\int_{G(\Q)\bs G(\A)^1} k^T(x,f)\;dx
\end{equation}
converges absolutely. This is the first step toward the trace formula. 
As shown by 
Hoffmann \cite{Ho}, $J^T(f)$ is defined  for a larger class of functions $f$
and $J^T(f)$ is a polynomial in $T\in\af_0$ of degree at most 
$d_0=\dim\af_{P_0}^G$. There is a distinguished point $T_0\in\af_0$ specified 
by \cite[Lemma 1.1]{Ar3}, and Arthur defines the distribution $J$ on $G(\A)^1$
by 
\begin{equation}\label{j-distr}
J(f):=J^{T_0}(f),\quad f\in C^\infty_c(G(\A)^1).
\end{equation}
This is the geometric side of the trace formula. To distinguish it from the
spectral side, we will denote it by $J_{\geo}$. 

In \cite{Ar1} Arthur has 
introduced the coarse geometric expansion of $J^T(f)$. 
To define it, one has to introduce an equivalence relation in $G(\Q)$. 
Define two elements $\gamma$ and $\gamma^\prime$ in $G(\Q)$ to be equivalent,
if the semisimple components $\gamma_s$ and $\gamma_s^\prime$ of their Jordan 
decompositions are $G(\Q)$-conjugate.
Let $\cO$ be the set of equivalence classes. Note that the set $\cO$ is in 
obvious bijection with the semisimple conjugacy classes in $G(\Q)$. Furthermore,
in case $G=\GL(n)$, the Jordan decomposition is given by the Jordan normal form.
For $\of\in\cO$ and $f\in C^\infty_c(G(\A)^1)$ let
\[
K_\of^P(x,y):=\int_{N_P(\Q)\bs N_P(\A)}\sum_{\gamma\in P(\Q)\cap \of} f(x^{-1}\gamma ny)\;
dn.
\]
Given $T\in\af_0$ and $x\in G(\A)^1$, let 
\begin{equation}\label{truncker2}
k^T_\of(x,f):=\sum_{P} (-1)^{\dim(A_P/A_G)}\sum_{\delta\in P(\Q)\bs G(\Q)}
K_\of^P(\delta x,\delta x)\widehat\tau_P\left(H_P(\delta x)-T_P\right).
\end{equation}
Let
\begin{equation}\label{distrib5}
J^T_\of(f):=\int_{G(\Q)\bs G(\A)^1} k^T_\of(x,f)\;dx.
\end{equation}
The integral converges absolutely and one obtains an absolutely convergent
expansion
\begin{equation}\label{coarseexp}
J^T(f)=\sum_{\of\in\cO} J^T_\of(f), \quad f\in C_c^\infty(G(\A)^1).
\end{equation}
This is the coarse geometric expansion, introduced in \cite{Ar1}. In \cite{FL1},
Finis and Lapid have shown that the coarse geometric expansion 
\eqref{coarseexp} extends continuously to the space of Schwartz functions
$\Co(G(\A)^1)$ which is defined as follows. For any compact open subgroup
$K_f$ of $G(\A_f)$ the space $G(\A)^1/K_f$ is the countable disjoint union of
copies of $G(\R)^1=G(\R)\cap G(\A)^1$ and therefore, it is a differentiable
manifold. Any element $X\in\mathcal{U}(\gf^1_\infty)$ of the universal 
enveloping algebra of the Lie algebra $\gf_\infty^1$ of $G(\R)^1$ defines a
left invariant differential operator $f\mapsto f\ast X$ on $G(\A)^1/K_f$. Let
$\Co(G(\A)^1;K_f)$ be the space of smooth right $K_f$-invariant functions on
$G(\A)^1$ which belong, together with all their derivatives, to $L^1(G(\A)^1)$.
The space $\Co(G(\A)^1;K_f)$ becomes a Fr\'echet space under the seminorms
\[
\|f\ast X\|_{L^1(G(\A)^1)},\quad X\in\mathcal{U}(\gf^1_\infty).
\]
Denote by $\Co(G(\A)^1)$ the union of the spaces $\Co(G(\A)^1;K_f)$ as $K_f$ 
varies over the compact open subgroups of $G(\A_f)$ and endow 
$\Co(G(\A)^1)$ with the inductive
limit topology. For $f\in\Co(G(\A)^1;K_f)$ and $\of\in\cO$ let $J^T(f)$ 
and $J^T_\of(f)$  be defined by \eqref{trf1} and \eqref{distrib5}, respectively.
By \cite[Theorem 7.1]{FL1}, the integrals defining $J^T(f)$ and $J^T_\of(F)$
are absolutely convergent and we have
\begin{equation}\label{coarseexp1}
J^T(f)=\sum_{\of\in\cO}J_\of^T(f),\quad f\in\Co(G(\A)^1;K_f).
\end{equation}
We shall now discuss how $J^T(f)$ is related the integral of the
kernel over the truncated manifold, where the truncated manifold is defined
by a certain height function. For $T\in\af_0$ let 
\begin{equation}\label{trunc1}
G(\A)^1_{\le T}=\{g\in G(\A)^1\colon \hat\tau_0(T-H_0(\gamma g))=1,
 \;\text{for}\;\text{all}\;\gamma\in G(\Q)\}.
\end{equation}
Note that by definition, $G(\A)^1_{\le T}$ is $G(\Q)$-invariant.
Furthermore, for $T_1\in\af_0$ let
\[
\fS_{T_1}=\{x\in G(\A)\colon \tau_0(H_0(x)-T_1)=1\}
\]
and more generally
\[
\fS^P_{T_1}=\{x\in G(\A)\colon \tau_0^P(H_0(x)-T_1)=1\}
\]
for any $P\supset P_0$. Note that these sets are left $P_0(\A)^1$-invariant.
By reduction theory, there exists $T_1\in\af_0$ such that
\[
P(\Q)\fS^P_{T_1}=G(\A)
\]
for all $P\supset P_0$, in particular for $P=G$. We fix such $T_1$. 
 Let
\begin{equation}\label{dT}
d(T)=\min_{\alpha\in\Delta_0}\langle\alpha,T\rangle.
\end{equation}
There exists $d_0>0$, which depends only on $G$, $P_0$ and ${\bf K}$, such 
that for all $T\in\af_0$ with $d(T)>d_0$ one has
\[
G(\A)^1_{\le T}\cap \fS_{T_1}=\{g\in G(A)^1\colon \tau_0(H_0(g)-T_1)
\hat\tau_0(T-H_0(g))=1\}.
\]
For $f\in \Co(G(\A)^1)$  recall that
\begin{equation}\label{kernel4}
K_f(x,y)=\sum_{\gamma\in G(\Q)}f(x^{-1}\gamma y)
\end{equation}
The series converges absolutely and uniformly on compact subsets. 
Then the following theorem, which  is an immediate consequence of 
\cite[Theorem 7.1]{FL1}, establishes the relation between $J^T(f)$ and naive
truncation.
\begin{theo}\label{theo-trunc}
For every open compact subgroup $K_f$ of $G(\A_f)$ there exists $r\ge 0$ and a 
continuous seminorm $\mu$ on $\Co(G(\A)^1;K_f)$ such that
\[
\left|\int_{G(\Q)\bs G(\A)^1_{\le T}}K_f(x,x)\;dx-J^T(f)\right|\le\mu(f)
(1+\|T\|)^r e^{-d(T)}
\]
for all $f\in \Co(G(\A)^1;K_f)$ and $T\in\af_0$ such that $d(T)>d_0$.
\end{theo}
\begin{proof}
For $\of\in\cO$ let
\[
 K_\of(x,y):=\sum_{\gamma\in\of} f(x^{-1}\gamma y).
\]
We have
\begin{equation}\label{kernel-sum}
K_f(x,y)=\sum_{\of\in\cO} K_\of(x,y),
\end{equation}
where the series converges absolutely.
Using \eqref{coarseexp1}, we get
\[
\left|\int_{G(\Q)\bs G(\A)^1_{\le T}}K_f(x,x)\;dx-J^T(f)\right|\le \sum_{\of\in\cO}
\left|\int_{G(\Q)\bs G(\A)^1_{\le T}}K_\of(x,x)\;dx - J_\of^T(f)\right|.
\]
and the theorem follows from \cite[Theorem 7.1]{FL1}. We note that for the case
of compactly supported functions $f$ this is due to Arthur 
(see \cite[\S 7]{Ar1}).
\end{proof}

\section{The non-invariant trace formula}\label{sec-trform}
\setcounter{equation}{0}

Arthur's (non-invariant) trace formula is the equality 
\begin{equation}\label{tracef1}
J_{\geo}(f)=J_{\spec}(f),\quad f\in C_c^\infty(G(\A)^1),
\end{equation}
of the geometric side $J_{\geo}(f)$ and the spectral side $J_{\spec}(f)$ of the
trace formula. The geome\-tric side has been described in the previous section.
In this section we recall the definition of the spectral side, and in particular
the refinement of the spectral expansion obtained in \cite{FLM1}. Combining
\cite{FLM1} and \cite{FL1}, it follows that \eqref{tracef1} extends continuously
to $f\in\Co(G(\A)^1)$.

The main ingredient of the spectral side  are logarithmic derivatives of 
intertwining operators. We briefly recall the structure of the intertwining 
operators.

Let $P\in\cP(M)$. 
Let $U_P$ be the unipotent radical of $P$. 
Recall that we denote  by $\rts_P\subset\af_P^*$ the set of reduced roots of 
$A_M$ of the Lie algebra $\mathfrak{u}_P$ of $U_P$.
Let $\srts_P$ be the subset of simple roots of $P$, which is a basis for 
$(\af_P^G)^*$.
Write $\af_{P,+}^*$ for the closure of the Weyl chamber of $P$, i.e.
\[
\aaa_{P,+}^*=\{\lambda\in\aaa_M^*:\sprod{\lambda}{\alpha^\vee}\ge0
\text{ for all }\alpha\in\rts_P\}
=\{\lambda\in\aaa_M^*:\sprod{\lambda}{\alpha^\vee}\ge0\text{ for all }
\alpha\in\srts_P\}.
\]
Denote by $\modulus_P$ the modulus function of $P(\A)$.
Let $\bar\AF_2(P)$ be the Hilbert space completion of
\[
\{\phi\in C^\infty(M(\Q)U_P(\A)\bs G(\A)):\modulus_P^{-\frac12}\phi(\cdot x)\in
L^2_{\disc}(\Ai M(\Q)\bs M(\A)),\ \forall x\in G(\A)\}
\]
with respect to the inner product
\[
(\phi_1,\phi_2)=\int_{\Ai M(\Q)\bU_P(\A)\bs \G(\A)}\phi_1(g)
\overline{\phi_2(g)}\ dg.
\]
Let $\alpha\in\rts_M$.
We say that two parabolic subgroups $P,Q\in\cP(M)$ are \emph{adjacent} along 
$\alpha$, and write $P|^\alpha Q$, if $\rts_P\cap-\rts_Q=\{\alpha\}$.
Alternatively, $P$ and $Q$ are adjacent if the group $\langle P,Q\rangle$
generated by $P$ and $Q$ belongs to $\cF_1(M)$ (see \eqref{f1} for its
definition).
Any $R\in\cF_1(\M)$ is of the form $\langle P,Q\rangle$, where $P,Q$ are
the elements of $\cP(M)$ contained in $R$. We have $P|^\alpha Q$ with 
$\alpha^\vee\in\rts_P^\vee \cap\af^R_M$.
Interchanging $P$ and $Q$ changes $\alpha$ to $-\alpha$.

For any $P\in\cP(M)$ let $\Ht_P\colon G(\A)\rightarrow\af_P$ be the 
extension of $\Ht_M$ to a left $U_P(\A)$-and right $\K$-invariant map.
Denote by $\cA^2(P)$ the dense subspace of $\bar\cA^2(P)$ consisting of its 
$\K$- and $\zzz$-finite vectors,
where $\zzz$ is the center of the universal enveloping algebra of 
$\mathfrak{g} \otimes \C$.
That is, $\cA^2(P)$ is the space of automorphic forms $\phi$ on 
$U_P(\A)M(\Q)\bs G(\A)$ such that
$\modulus_P^{-\frac12}\phi(\cdot k)$ is a square-integrable automorphic form on
$\Ai M(\Q)\bs M(\A)$ for all $k\in\K$.
Let $\rho(P,\lambda)$, $\lambda\in\af_{M,\C}^*$, be the induced
representation of $G(\A)$ on $\bar\cA^2(P)$ given by
\[
(\rho(P,\lambda,y)\phi)(x)=\phi(xy)e^{\sprod{\lambda}{\Ht_P(xy)-\Ht_P(x)}}.
\]
It is isomorphic to the induced representation 
\[
\Ind_{P(\A)}^{G(\A)}\left(L^2_{\disc}(\Ai M(\Q)\bs M(\A))
\otimes e^{\sprod{\lambda}{\Ht_M(\cdot)}}\right).
\]
For alternative descriptions see \cite[\S 1]{Ar8}, \cite[I.2.17, I.2.18]{MW}.

For $P,Q\in\cP(M)$ let
\[
M_{Q|P}(\lambda):\cA^2(P)\to\cA^2(Q),\quad\lambda\in\af_{M,\C}^*,
\]
be the standard \emph{intertwining operator} \cite[\S 1]{Ar9}, which is the 
meromorphic continuation in $\lambda$ of the integral
\[
[M_{Q|P}(\lambda)\phi](x)=\int_{U_Q(\A)\cap U_P(\A)\bs U_Q(\A)}\phi(nx)
e^{\sprod{\lambda}{\Ht_P(nx)-\Ht_Q(x)}}\ dn, \quad \phi\in\cA^2(P), \ x\in G(\A).
\]
Given $\pi\in\Pi_{\di}(M(\A))$, let $\cA^2_\pi(P)$ be the space of all 
$\phi\in\cA^2(P)$ for which the function 
$M(\A)\ni x\mapsto \delta_P^{-\frac{1}{2}}\phi(xg)$,
$g\in G(\A)$, belongs to the $\pi$-isotypic subspace of the space
$L^2(\Ai M(\Q)\bs M(\A))$.
For any $P\in\cP(\M)$ we have a canonical isomorphism of 
$G(\A_f)\times(\LieG_{\C},K_\infty)$-modules
\[
j_P:\Hom(\pi,L^2(\Ai M(\Q)\bs M(\A)))\otimes 
\Ind_{P(\A)}^{G(\A)}(\pi)\rightarrow\cA^2_\pi(P).
\]
If we fix a unitary structure on $\pi$ and endow 
$\Hom(\pi,L^2(\Ai M(\Q)\bs M(\A)))$ with the inner product 
$(A,B)=B^\ast A$
(which is a scalar operator on the space of $\pi$), the isomorphism $j_P$ 
becomes an isometry. 

Suppose that $P|^\alpha Q$.
The operator $M_{Q|P}(\pi,s):=M_{Q|P}(s\varpi)|_{\cA^2_\pi(P)}$, where $\varpi\in
\af^\ast_M$ is such that $\langle\varpi,\alpha^\vee\rangle=1$, admits a 
normalization by a global factor
$n_\alpha(\pi,s)$ which is a meromorphic function in $s$. We may write
\begin{equation} \label{normalization}
M_{Q|P}(\pi,s)\circ j_P=n_\alpha(\pi,s)\cdot j_Q\circ(\Id\otimes R_{Q|P}(\pi,s))
\end{equation}
where $R_{Q|P}(\pi,s)=\otimes_v R_{Q|P}(\pi_v,s)$ is the product
of the locally defined normalized intertwining operators and 
$\pi=\otimes_v\pi_v$
\cite[\S 6]{Ar9}, (cf.~\cite[(2.17)]{Mu2}). In many cases, the 
normalizing factors can be expressed in terms automorphic $L$-functions 
\cite{Sha1}, \cite{Sha2}. For example, let $G=\GL(n)$. Then
the global normalizing factors $n_\alpha$ can be expressed in terms
of Rankin-Selberg $L$-functions. The
known properties of these functions are collected and analyzed in 
\cite[\S\S 4,5]{Mu1}.
Write $M \simeq \prod_{i=1}^r \GL (n_i)$, where the root $\alpha$ is trivial on 
$\prod_{i \ge 3} \GL (n_i)$,
and let $\pi \simeq \otimes \pi_i$ with representations $\pi_i \in 
\Pi_{\disc}(\GL(n_i,\A))$.
Let $L(s,\pi_1\times\tilde\pi_2)$ be the completed Rankin-Selberg $L$-function 
associated to $\pi_1$ and $\pi_2$. It satisfies the functional equation
\begin{equation}\label{functequ}
L(s,\pi_1\times\tilde\pi_2)=\eps(\frac12,\pi_1\times\tilde\pi_2)
N(\pi_1\times\tilde\pi_2)^{\frac12-s}L(1-s,\tilde\pi_1\times\pi_2)
\end{equation}
where $\abs{\eps(\frac12,\pi_1\times\tilde\pi_2)}=1$ and $N(\pi_1\times\tilde
\pi_2)\in\N$ is the conductor. Then we have
\begin{equation}\label{rankin-selb}
n_\alpha (\pi,s) = \frac{L(s,\pi_1\times\tilde\pi_2)}{\eps(\frac12,\pi_1\times
\tilde\pi_2)N(\pi_1\times\tilde\pi_2)^{\frac12-s}L(s+1,\pi_1\times\tilde\pi_2)}.
\end{equation}

We now turn to the spectral side. Let $L \supset M$ be Levi subgroups in 
$\levis$, $P \in \PPP (M)$, and
let $m=\dim\aaa_L^G$ be the co-rank of $L$ in $G$.
Denote by $\bases_{P,L}$ the set of $m$-tuples $\bss=(\beta_1^\vee,\dots,
\beta_m^\vee)$
of elements of $\rts_P^\vee$ whose projections to $\af_L$ form a basis for 
$\af_L^G$.
For any $\bss=(\beta_1^\vee,\dots,\beta_m^\vee)\in\bases_{P,L}$ let
$\vol(\bss)$ be the co-volume in $\af_L^G$ of the lattice spanned by $\bss$ 
and let
\begin{align*}
\Xi_L(\bss)&=\{(Q_1,\dots,Q_m)\in\cF_1(M)^m: \ \ \beta_i^{\vee}\in\af_M^{Q_i}, 
\, i = 1, \dots, m\}\\&=
\{\langle P_1,P_1'\rangle,\dots,\langle P_m,P_m'\rangle): \ \ 
P_i|^{\beta_i}P_i', \, 
i = 1, \dots, m\}.
\end{align*}

For any smooth function $f$ on $\af_M^*$ and $\mu\in\af_M^*$ denote by 
$D_\mu f$ the directional derivative of $f$ along $\mu\in\af_M^*$.
For a pair $P_1|^\alpha P_2$ of adjacent parabolic subgroups in $\cP(M)$ write
\begin{equation}\label{intertw2}
\delta_{P_1|P_2}(\lambda)=M_{P_2|P_1}(\lambda)D_\varpi M_{P_1|P_2}(\lambda):
\cA^2(P_2)\rightarrow\cA^2(P_2),
\end{equation}
where $\varpi\in\af_M^*$ is such that $\sprod{\varpi}{\alpha^\vee}=1$.
\footnote{Note that this definition differs slightly from the definition of
$\delta_{P_1|P_2}$ in \cite{FLM1}.} Equivalently, writing
$M_{P_1|P_2}(\lambda)=\Phi(\sprod{\lambda}{\alpha^\vee})$ for a
meromorphic function $\Phi$ of a single complex variable, we have
\[
\delta_{P_1|P_2}(\lambda)=\Phi(\sprod{\lambda}{\alpha^\vee})^{-1}
\Phi'(\sprod{\lambda}{\alpha^\vee}).
\]
For any $m$-tuple $\dtup=(Q_1,\dots,Q_m)\in\Xi_L(\bss)$
with $Q_i=\langle P_i,P_i'\rangle$, $P_i|^{\beta_i}P_i'$, denote by 
$\Delta_{\dtup}(P,\lambda)$
the expression
\begin{equation}\label{intertw3}
\frac{\vol(\bss)}{m!}M_{P_1'|P}(\lambda)^{-1}\delta_{P_1|P_1'}(\lambda)M_{P_1'|P_2'}(\lambda) \cdots
\delta_{P_{m-1}|P_{m-1}'}(\lambda)M_{P_{m-1}'|P_m'}(\lambda)\delta_{P_m|P_m'}(\lambda)M_{P_m'|P}(\lambda).
\end{equation}
In \cite[pp. 179-180]{FLM1} the authors defined a (purely combinatorial) map $\dtup_L: \bases_{P,L} \to \FFF_1 (M)^m$ with the property that
$\dtup_L(\bss) \in \Xi_L (\bss)$ for all $\bss \in \bases_{P,L}$.\footnote{The map $\dtup_L$ depends in fact on the additional choice of
a vector $\underline{\mu} \in (\aaa^*_M)^m$ which does not lie in an explicit finite
set of hyperplanes. For our purposes, the precise definition of $\dtup_L$ is immaterial.}

For any $s\in W_M$ let $L_s$ be the smallest Levi subgroup in $\levis(M)$
containing $w_s$. We recall that $\aaa_{L_s}=\{H\in\aaa_M\mid sH=H\}$.
Set
\[
\iota_s=\abs{\det(s-1)_{\aaa^{L_s}_M}}^{-1}.
\]
For $P\in\FFF(M_0)$ and $s\in W(M_P)$ let
$M(P,s):\AF^2(P)\to\AF^2(P)$ be as in \cite[p.~1309]{Ar3}.
$M(P,s)$ is a unitary operator which commutes with the operators $\rho(P,\lambda,h)$ for $\lambda\in\iii\aaa_{L_s}^*$.
Finally, we can state the refined spectral expansion.

\begin{theo}[\cite{FLM1}] \label{thm-specexpand}
For any $h\in C_c^\infty(G(\A)^1)$ the spectral side of Arthur's trace formula is given by
\begin{equation}\label{specside1}
J_{\spec}(h) = \sum_{[M]} J_{\spec,M} (h),
\end{equation}
$M$ ranging over the conjugacy classes of Levi subgroups of $G$ (represented by members of $\mathcal{L}$),
where
\begin{equation}\label{specside2}
J_{\spec,M} (h) =
\frac1{\card{W(M)}}\sum_{s\in W(M)}\iota_s
\sum_{\bss\in\bases_{P,L_s}}\int_{\iii(\aaa^G_{L_s})^*}
\tr(\Delta_{\dtup_{L_s}(\bss)}(P,\lambda)M(P,s)\rho(P,\lambda,h))\ d\lambda
\end{equation}
with $P \in \PPP(M)$ arbitrary.
The operators are of trace class and the integrals are absolutely convergent
with respect to the trace norm and define distributions on $\Co(G(\A)^1)$.
\end{theo}
Note that the term corresponding to $M=G$ is $J_{\spec,G} (h) = \tr R_{\disc}(h)$.

\section{The unipotent contribution to the trace formula}\label{sec-unipotent}
\setcounter{equation}{0}

In this section we assume that $G=\GL(n)$ or $G=\SL(n)$ as algebraic groups over $\Q$, and we specialize to one of these groups at some points. 
The purpose is to analyze  the unipotent contribution to the geometric side of
the trace formula. The point of departure is the coarse geometric expansion 
of $J_{\geo}$ as a sum of distributions
\begin{equation}\label{geoside}
J_{\geo}(f)=\sum_{\ho\in\cO}J_{\ho}(f),\quad f\in C^\infty_c(G(\A)^1),
\end{equation} 
parametrized by the set $\cO$ of semisimple conjugacy classes of $G(\Q)$.
The distribution $J_{\ho}(f)$ is the value at $T=0$
of the polynomial $J_{\ho}^T(f)$ defined in \cite{Ar1}. 
In particular, following Arthur, we write $J_{\unip}(f)$ for the 
contribution corresponding to the class of $\{1\}$. Let $K(N)$ be a
principal congruence subgroup of level $N\ge 3$. By \cite[Corollary 5.2]{LM}
there exists a bi-$K_\infty$-invariant compact neighborhood $\omega$ of $K(N)$ in
$G(\A)^1$ such that 
\begin{equation}\label{unip-contr1}
J_{\geo}(f)=J_{\unip}(f).
\end{equation}
for all $f\in C_c^\infty(G(\A)^1)$ supported in $\omega$. (\cite[Corollary 5.2]{LM} was only stated for $\GL(n)$, but its proof holds also for  $\SL(n)$ without modification.)  For our applications
we can choose $f$ such that \eqref{unip-contr1} holds if we restrict to $K_f$ with $K_f\subseteq K(N)$ for some $N\ge3$. See \S~\ref{sec:examples} for computations for $n=2,3$ regarding orbits of finite order which appear in the  case $K_f=G(\hat\Z)$.

To analyze $J_{\unip}(f)$ we use Arthur's fundamental result 
(\cite[Corollaries 8.3 and 8.5]{Ar4}) to 
express $J_{\unip}(f)$ in terms of weighted orbital integrals. To state the 
result we recall some facts about weighted orbital integrals. Let $S$ be a 
finite set of places of $\Q$ containing $\infty$. Set
\[
\Q_S=\prod_{v\in S}\Q_v,\quad \mathrm{and}\quad G(\Q_S)=\prod_{v\in S}G(\Q_v).
\]
Let
 \[
G(\Q_S)^1=G(\Q_S)\cap G(\A)^1
\]
and write $C^\infty_c(G(\Q_S)^1)$ for the space of functions on $G(\Q_S)^1$
obtained by restriction of functions in $C^\infty_c(G(\Q_S))$ to $G(\Q_S)^1$. Further, let $\A^S=\prod_{v\not\in S}\Q_v$ be the restricted product over all places outside of $S$, and define $G(\A^S)$ similarly as above. 

\subsection{The fine geometric expansion}

Let $M\in\levis$ and $\gamma\in M(\Q_S)$. 
The general weighted orbital integrals $J_M(\gamma,f)$ defined in 
(\cite{Ar5}) are distributions on $G(\Q_S)$. 
Denote by $H_\gamma$ the centralizer of $\gamma$ in a subgroup $H$ of $G$.
If $\gamma$ is such that $M_\gamma=G_\gamma$ then
$J_M(\gamma,f)$ is given by an integral of the form
\begin{equation}\label{orbint}
J_M(\gamma,f)=\big|D(\gamma)\big|^{1/2}\int_{{G_\gamma}(\Q_S)\bs G(\Q_S)}
f(x^{-1}\gamma x) v_M(x)\ dx,
\end{equation}
where $D(\gamma)$ is defined in \cite[p. 231]{Ar5}
and $v_M(x)$ is the weight function associated to the $(G,M)$-family
$\{v_P(\lambda,x)\colon P\in\mathcal{P}(M)\}$ defined in 
\cite[p.230]{Ar5}. It is a left $M(\A)$-invariant and right $\mathbf{K}$-invariant function on $G(\A)$. 
In particular, in the case $M=G$ (in which $v_M\equiv1$) we obtain the usual
(invariant) orbital integral. Of course, implicit in $J_M(\gamma,f)$
is a choice of a Haar measure on $G_\gamma(\Q_S)$.
When the condition $G_\gamma\subset M$ is not satisfied (for example, if $\gamma$ is unipotent and $M\neq G$), 
the definition of $J_M(\gamma,f)$ is more complicated.
It is obtained as a limit of a linear
combination of integrals as above. For more details we refer to 
\cite{Ar5}, see also the description below.
 If $\gamma$ 
belongs to the intersection of $M(\Q_S)$ with $G(\Q_S)^1$, one can obviously
define the corresponding weighted orbital integral as a linear form on 
$C^\infty_c(G(\Q_S)^1)$.  Note that $J_M(\gamma,f)$ depends only on the $M(\Q_S)$-conjugacy class of $\gamma$.

To state the fine expansion of $J_{\unip}(f)$, we need to introduce a certain equivalence relation as defined in \cite{Ar4,Ar7}.
Let $\cU_M$ denote the variety of unipotent elements in $M$ and $\cU_M(A)$ its $A$-points for any $\Z$-algebra $A$. We say that $u_1, u_2\in \cU_M(\Q)$ are \emph{$(M,S)$-equivalent} if they are conjugate by some element in $M(\Q_S)$, cf.~\cite[\S7]{Ar4}. We denote by $\left(\cU_M(\Q)\right)_{M,S}$ the set of $(M,S)$-equivalence classes in $\cU_M(\Q)$. We note that $\left(\cU_M(\Q)\right)_{M,S}$ is finite for any $S$ but may get larger as $S$ grows. 
We get an injective map
\[
 \left(\cU_M(\Q)\right)_{M,S}\longrightarrow \left(\cU_M(\Q_S)\right)_{M(\Q_S)} 
\]
where $\left(\cU_M(\Q_S)\right)_{M(\Q_S)}$ denotes the set of $M(\Q_S)$-conjugacy classes in $\cU_M(\Q_S)$. Note that this last set is evidently the same as the direct product $\prod_{v\in S} \left(\cU_M(\Q_v)\right)_{M(\Q_v)}$ over the $M(\Q_v)$-conjugacy classes in $\cU_M(\Q_v)$, $v\in S$. In particular, we can identify an equivalence class $U\in\left(\cU_M(\Q)\right)_{M,S}$ with its image under the above map, that is, with a tuple $(U_v)_{v\in S}$ of $M(\Q_v)$-conjugacy classes in $\cU_M(\Q_v)$.

By~\cite[Theorem 8.1]{Ar4} we have
\begin{equation}\label{finegeoexp}
 J_{\unip}(f\otimes\one_{K^S})=\vol(G(\Q)\backslash G(\A)^1)f(1) + \sum_{(M,U)}a^M(S,U) J_M(U,f),
\end{equation}
where $(M,U)$ runs over all pairs of Levi subgroups $M\in \levis$ and $U\in \left(\cU_M(\Q)\right)_{M,S}$ with $(M,U)\neq (G, 1)$. Here $f\in C^\infty_c(G(\Q_S)^1)$, $\one_{K^S}$ is the characteristic
function of the standard maximal compact of $G(\A^S)$, and $a^M(S,U)$ are 
certain constants which depend on the normalization of measures (but are 
usually not known explicitly).
The distributions $J_M(U,f)$ can be written as weighted orbital integrals 
\cite[p. 256]{Ar5}. In our case the integrals simplify as we are going to see later. 

\subsection{Unipotent conjugacy classes}

Let $\left(\cU_M\right)$ denote the set of geometric (that is,  over an algebraic closure $\bar\Q$ of $\Q$) $M$-conjugacy classes in $\cU_M$. 
Then $\left(\cU_M\right)$ is a finite set. 
Each $\cV\in \left(\cU_M\right)$ is defined over $\Q$ and the set of $\Q$-points $\cV(\Q)$ is non-empty. More precisely, each $\cV\in\left(\cU_M\right)$ corresponds to a partition of $n$, and $\cV(\Q)$ contains the matrix with Jordan normal form corresponding to that partition.   
For any $\cV\in\left(\cU_M\right)$ the set $\cV(\Q)$ is closed under the $(M,S)$-equivalence relation and we write  $\left(\cV(\Q)\right)_{M,S}$ for the finite set of $(M,S)$-equivalence classes in $\cV(\Q)$. 

\begin{remark}
If $G=\SL(n)$,  there might be infinitely many $M(\Q)$-conjugacy classes in $\cV(\Q)$ depending on the type of $\cV(\Q)$. This is in contrast to the case of $\GL(n)$, where the (finite) set of geometric unipotent conjugacy classes is in bijection with the set of rational unipotent conjugacy classes. 
\end{remark}

Each class $\cV\in\left(\cU_M\right)$ is a Richardson class, that is, there exists a standard parabolic subgroup $Q=LV\in\levis(M)$ such that $\cV$ is induced from the trivial orbit in $L$ to $M$, see \cite[\S 5.5 Proposition]{Hu}. Equivalently, the intersection of  $\cV$ with $V$ is an open and dense subset of $V$. Note that every $(M,S)$-equivalence class $U\subseteq \cV(\Q)$ has a representative in $V(\Q)$. 

Now let $\cV\in \left(\cU_M\right)$ and $U=(U_v)_{v\in S}\in \left(\cV(\Q_S)\right)_{M(\Q_S)}$. Here we write $\left(\cV(\Q_S)\right)_{M(\Q_S)}=\prod_{v\in S}\left(\cV(\Q_v)\right)_{M(\Q_v)}$ for the set of $M(\Q_S)$-conjugacy classes in $\cV(\Q_S)$. To understand the $S$-adic integral $J_M(U,f)$, we decompose it into a sum of products of integrals at $\infty$ and at the finite places $S_{\fin}=S\backslash\{\infty\}$.  More precisely, for every pair of Levi subgroups $L_1, L_2\in \levis(M)$ there exists a coefficient $d_M^G(L_1, L_2)\in\C$ such that
\begin{equation}\label{orbint11}
 J_M(U,f)=\sum_{L_1, L_2\in\levis(M)} d_M^G(L_1, L_2) J_M^{L_1}(U_{\infty}, f_{\infty, Q_1}) J_M^{L_2}(U_{\fin}, f_{\fin, Q_2})
\end{equation}
(see~\cite{Ar3}) where $U_{\fin}=(U_v)_{v\in S_{\fin}}\in \left(\cV(\Q_{S_{\fin}})\right)_{M(\Q_{S_{\fin}})}$, and  $Q_1=L_1V_1\in \cP(L_1))$, $Q_2=L_2V_2\in\cP(L_2)$ 
are certain parabolic subgroups the exact choice of which does not matter to 
us. Moreover,  
\[
 f_{\infty, Q_1}(m)=\delta_{Q_1}(m)^{1/2} \int_{K_{\infty}}\int_{V_1(\R)} f(k^{-1} mv k)\, dk\, dv, \; m\in L_1(\R),
\]
with $f_{\fin, Q_2}$ being defined analogously, and  the coefficients $d_M^G(L_1, L_2)$ are independent of $S$ and they vanish unless the natural map $\ka_M^{L_1}\oplus\ka_M^{L_2}\longrightarrow \ka_M^G$ is an isomorphism. In case the coefficient does not vanish, it depends on the chosen measures on $\ka_M^{L_1}$, $\ka_M^{L_2}$ and $\ka_M^G$.
The distributions $J_M^{L}(U, F)$, for $L\in\levis(M)$, any finite set $S'$ of places of $\Q$,  $F\in C^\infty(L(\Q_{S'}))$, and a unipotent conjugacy class $U\subseteq L(\Q_{S'})$, are defined  similarly as the weighted orbital integrals $J_M^G=J_M$ but with $L$ in place of $G$.

Our test function at the places in $v\in S_{\fin}$ is fixed once and for all so that the integrals at those places can be viewed as constant for our purposes. Hence we need to understand the integral at the Archimedean place.  We therefore need to better understand the unipotent conjugacy classes over $\R$. If $G=\GL(n)$, the unipotent orbits in $\GL_n(\R)$ and all its Levi subgroups in $\levis$ are easy to describe as they are in one-to-one correspondence with the geometric unipotent conjugacy classes and therefore classified by partitions of $n$. We assume for the moment that $G=\SL(n)$. Note that $\SL_n(\R)$ is normalized by 
$\GL_n(\R)$. 
Moreover, each $M\in\levis$ is of the form $M=\bar M\cap\SL(n)$ for a unique Levi subgroup $\bar M\in \levis^{\GL(n)}$. Then $\bar K^M_{\infty}:=\rO(n)\cap \bar M(\R)$ is a maximal compact subgroup in $\bar M$, $\bar K^M_{\infty}\cap M(\R)=K^M_{\infty}$,  and $\bar M(\R)$ normalizes $M(\R)$. 
In particular, it makes sense to speak of $\bar K_{\infty}^M$-conjugation 
invariant functions on $M(\R)$.

\begin{lem}
Let $\cV\in\left(\cU_M\right)$.
For any two equivalence classes $U_1, U_2\in\left(\cV(\R)\right)_{M(\R)}$ there exists $k\in\bar K^M_{\infty}$ with $k^{-1} U_1 k=U_2$ and $\left(\cV(\R)\right)_{M(\R)}$ consists of at most two equivalence classes. More precisely, if $U_1, U_2\in\left(\cV(\R)\right)_{M(\R)}$ are the two distinct classes, we have $k^{-1} U_1 k=U_2$ with $k=\diag(-1, 1, \ldots, 1)\in \bar K^M_{\infty}$.
\end{lem}
\begin{proof}
Let $u_1\in U_1$ and  $u_2\in U_2$. In $\bar M(\R)$, $u_1$ and $u_2$ are conjugate, that is, there exists some $g\in \bar M(\R)$ with $g^{-1} u_1 g=u_2$. Without loss of generality we can assume that $|\det g|=1$. If $\det g=1$, then $g\in \bar M(\R)\cap\SL_n(\R)=M(\R)$ and $U_1=U_2$. If $\det g=-1$, let $k=\diag(-1, 1,\ldots, 1)\in \bar K^M_{\infty}$ and $g_1=gk^{-1}$. Then $g_1\in \bar M(\R)\cap \SL_n(\R)=M(\R)$, and $U_2= g^{-1} U_1 g= k^{-1} g_1^{-1} U_1 g_1 k= k^{-1} U_1 k$ as asserted.  
\end{proof}

\subsection{Measures on conjugacy classes}

\begin{corollary}\label{cor:invariance:arch}
 If $\cV\in\left(\cU_M\right)$  and $f_{\infty}\in C_c(G(\R))$ is conjugation invariant under $\bar K_{\infty}$, then we have $J_M(U_1,f_{\infty})=J_M(U_2, f_{\infty})$ for all $U_1, U_2\in\left(\cV(\R)\right)_{M(\R)}$.
\end{corollary}
\begin{proof}
By the previous lemma, there are at most two distinct classes in $\left(\cV(\R)\right)_{M(\R)}$. If there is only one class in $\left(\cV(\R)\right)_{M(\R)}$, there is nothing to show. If there are two distinct classes $U_1, U_2$, they are conjugate to each other via the element $k=\diag(-1, 1, \ldots, 1)\in\bar K_{\infty}^M$. Let $u_1\in U_1$ and $u_2=k^{-1}u_1k\in U_2$. Then the centralizers of $u_1$ and $u_2$ in $G(\R)$ are the same, since they are the same in $\GL_n(\R)$ and $k$ normalizes $G(\R)=\SL_n(\R)$ in $\GL_n(\R)$. Hence the invariant measures on the $G(\R)$-conjugacy classes of $u_1$ and $u_2$ coincide, in particular, $J_G(U_1,f_{\infty})=J_G(U_2, f_{\infty})$ for every $\bar K_{\infty}$-conjugation invariant $f_{\infty}\in C_c(G(\R))$.  

The non-invariant measure defining $J_M(U_i,\cdot)$ can be written as the product of some weight function times the invariant measure. We shall see in the next section (see \eqref{eq:weight:function:invariance:P} and \eqref{eq:definition:weight}) that the weight function is invariant under the action of $k=\diag(-1, 1, \ldots, 1)$ so that the claim of the lemma follows for the weighted orbital integrals as well.
\end{proof}

The following corollary now is valid for $G=\SL(n)$ as well as $G=\GL(n)$. However, we shall only prove it for $\SL(n)$. For $G=\GL(n)$ the proof in fact is easier and was already given in \cite[Lemma 5.3]{LM} (see also \cite[Proposition 5]{Howe74}).
\begin{corollary}
 For every $\cV\in\left(\cU_G\right)$ there exists a standard parabolic subgroup $P=LV$ and a constant $c>0$ such that for every $\rO(n)$-conjugation invariant function $f_{\infty}$ and every $U\in\left(\cV(\R)\right)_{G(\R)}$ the invariant orbital integral $J_G(U,f_{\infty})$ can be written as
 \[
  J_G(U, f_{\infty})= c\int_{V(\R)} f_{\infty}(v)\, dv,
 \]
 where $dv$ denotes the Haar measure on $V(\R)$ normalized such that it coincides with the measure obtained from the Lebesgue measure on $\R^{\dim V}$ when $V(\R)$ is identified with $\R^{\dim V}$ via its matrix coordinates.
\end{corollary}

\begin{proof}
  Let $P=LV\in\PPP$ be a Richardson parabolic subgroup for $\cV$ so that $\cV(\R)\cap V(\R)$ is dense in $V(\R)$. We have $\cV(\R)=\bigcup_{U\in\left(\cV(\R)\right)_{G(\R)}} U$ and this union is disjoint. Then 
  \begin{equation}\label{eq:decomposition}
   \cV(\R)\cap V(\R)=\bigcup_{U\in\left(\cV(\R)\right)_{G(\R)}} U\cap V(\R)
  \end{equation}
is also a disjoint union which is dense in $V(\R)$.
 For each $U\in\left(\cV(\R)\right)_{G(\R)}$ we can pick a representative $u\in V(\R)$. Then the orbit of $u$ under $P(\R)$ equals $U\cap V(\R)$.
 Since $f_{\infty}$ is $\rO(n)$-conjugation invariant, we have, using Iwasawa decomposition for $G(\R)$, that
  \[
   J_G(U,f_{\infty})= \int_{P_u(\R)\backslash P(\R)}\delta_P(p)^{-1} f_{\infty} (p^{-1} u p)\, dp.
  \]
It follows as in the proof of~\cite[Lemma 5.3]{LM} that
\[
C_c^{\infty}(V(\R))\ni h\longrightarrow \int_{P_u(\R)\backslash P(\R)}\delta_P(p)^{-1} h (p^{-1} u p)\, dp
\]
is absolutely continuous with respect to the Haar measure on $V(\R)$. Hence 
\begin{equation}\label{eq:orb:int:sum}
 C_c^{\infty}(V(\R))\ni h\longrightarrow \sum_{u} \int_{P_u(\R)\backslash P(\R)}\delta_P(p)^{-1} h (p^{-1} u p)\, dp
\end{equation}
is also absolutely continuous with respect to the Haar measure on $V(\R)$. Here $u\in V(\R)$ runs over a set of representatives for the classes $U\in\left(\cV(\R)\right)_{G(\R)}$. Since the right hand side of \eqref{eq:decomposition} is a disjoint union and dense in $V(\R)$, the measure defined by the right hand side of~\eqref{eq:orb:int:sum} must be proportional to the Haar measure on $V(\R)$. Hence there exists a constant $C>0$ such that for every $f_{\infty}\in C_c^{\infty}(G(\R))$ we have
\[
 \sum_{U\in\left(\cV(\R)\right)_{G(\R)}} J_G(U,f_{\infty})
 =C\int_{V(\R)} f_{\infty}(v)\, dv.
\]
By Corollary~\ref{cor:invariance:arch} and our assumption on $f_{\infty}$ we have  $J_G(U_1, f_{\infty})=J_G(U_2, f_{\infty})$ for all $U_1, U_2\in \left(\cV(\R)\right)_{G(\R)}$. Hence
\[
 J_G(U,f_{\infty})
 =\frac{C}{N}\int_{V(\R)} f_{\infty}(v)\, dv
\]
where $N$ is the number of classes in $\left(\cV(\R)\right)_{G(\R)}$. 
\end{proof}

The weighted integral $J_M(U,f_{\infty})$ from \cite[p. 256]{Ar5} can be written as an integral over $\Ind_M^G U$ against the invariant measure on $\Ind_M^G U$ weighted by a certain function. 
Hence using the corollary above it follows that the real orbital integral $J_M(U,f_{\infty})$  simplifies for every $\rO(n)$-conjugation invariant $f_{\infty}$ to
\begin{equation}\label{weiorbint}
J_M(U,f_{\infty})=\int_{N(\R)}f(x)w_{M,U}(x)\, dx
\end{equation}
where $Q=LN\in\PPP$ is a Richardson parabolic for $\Ind_M^G U$, the unipotent orbit induced from $M$ to $G$,
and the weight function $w_{M,U}(x)$ is described in \cite[Lemma 5.4]{Ar5}. (See also \cite[\S 5]{LM}.)
It is a finite linear combination
of functions of the form $\prod_{i=1}^r\log{\norm{p_i(x)}}$
where $p_i$ are polynomials on $N_Q(\R)$ into an affine space,
$i=1,\dots,r$ (not necessarily distinct) and
\[
\norm{(y_1,\dots,y_m)}_v=
\abs{y_1}_v^2+\dots+\abs{y_m}_v^2 
\]
(The fact that the product is over $r$ terms is implicit in \cite{Ar5}
but follows from the proof.)
For our purpose we need to describe the weight function in more detail which we shall do in the next section.

\section{The weight function}\label{sec-wfct}
\setcounter{equation}{0}
In this section, again $G=\GL(n)$ or $G=\SL(n)$.
We are only interested in the situation over the field $\R$, but most of this section holds over $\Q_p$, $p<\infty$, as well. As before, $\cU_M$ denotes the unipotent variety in $M$ and $\cV\in\left(\cU_M\right)$  a geometric conjugacy class. We write
\[
 \cpt^G=\begin{cases}
         \SO(n)				&\text{if }G=\SL(n),\\
         \rO(n)				&\text{if }G=\GL(n),
        \end{cases}
\]
and we write $\cpt$ for $\cpt^G$ if $G$ is clear from the context.

We first recall the definition of the local weight functions from \cite{Ar5}. 
Let $\cV\in\left(\cU_M\right)$ and $U\in \left(\cV(\R)\right)_{M(\R)}$. 
We denote by $\cV^G$ the conjugacy class in $\left(\cU_G\right)$ induced from $\cV$ along some parabolic subgroup in $\CmP(M)$ (they yield all the same induced class). Let $U^G\in \left(\cV^G(\R)\right)_{G(\R)}$ be  such that $U^G\cap M(\R)=U$. (As explained above there are at most two different elements in $\left(\cV^G(\R)\right)_{G(\R)}$ which are conjugate by $\diag(-1, 1, \ldots, 1)$.)

Arthur~\cite{Ar5} defines a weight function $w_{M,U}$ on a dense open subset of $\cV^G(\R)$ such that the local weighted orbital integral $J_M(U,f)=J_M^G(U, f)$ can be defined as
\begin{equation}\label{eq:integral}
 J_M^G(U,f) = \int_{U^G} f(x) w_{M,U}(x)\, dx
\end{equation} 
for any $f\in C^{\infty}(G(\R))$ of almost compact support with $dx$ denoting the invariant measure on $U^G$, cf. \cite{Rao}. 
Let $Q=LN\in\PPP$ be a Richardson parabolic subgroup for $\cV^G$. 
 By the results of the previous section, if $f$ is conjugation invariant under the group $\rO(n)$, we can write the above as 
\begin{equation}\label{realorbint}
 J_M^G(U,f) = c \int_{N(\R)} f(n) w_{M,U}(n)\, dn
\end{equation}
for some constant $c>0$.  The function $w_{M, U}$ actually only depends on $\cV$ and not on the specific $U\in \left(\cV(F)\right)_{M(\R)}$ so that we can also write $w_{M,\cV}$ for this function.  
Note that the map $x\mapsto X=x-\Mid$ defines a bijection between the variety $\cU_G(\R)$ and the nilpotent cone in the Lie algebra $\kg(\R)$. Moreover, if $x\in U^G$, then for any $s\in \R$, $s\neq0$, the element  
\begin{equation}\label{eq:def:xs}
x_s:=\Mid+ s(x-\Mid)
\end{equation}
is an element in $\cV^G(\R)$ (but not necessarily in $U^G$).

The goal is to show the following in this section:

\begin{prop}\label{prop}
Let $P=MV\in\CmF(M)$. Let $x\in \cV^G(\R)$ be such that $w_{M,\cV}(x)$ is defined. Then $w_{M,\cV}(x_s)$ is also well-defined for every $s>0$, and as a function of $s$, it is a polynomial in $\log s$ of degree at most $\dim\ka_M^G$. Moreover, there are $k\in\N$, homogeneous polynomials $P_1, \ldots, P_t: N(\R)\longrightarrow \R^k$, and coefficients $\alpha_I\in\R$ for each subset $I\subseteq \{1,\ldots, t\}$ such that $w_{M,\cV}(\Mid+X)=\sum_{I\subseteq \{1,\ldots, t\}} \alpha_I\prod_{i\in I} \log\|P_i(X)\|$ for $X$ in a dense subset of $N(\R)$. Here $\|\cdot\|$ denotes the vector norm on $\R^k$. 
\end{prop}

To prove this proposition, we follow along the lines of Arthur's construction of the weight functions.
 Let $\Phi(A_M, G)$ be the set of roots of $(A_M, G)$, and let $\beta\in\Phi(A_M, G)$. Note that every root in $\Phi(A_M, G)$ is reduced in $\Phi(A_M, G)$, that is, if $\gamma\in\Phi(A_M, G)$, then $m\gamma\not\in\Phi(A_M, G)$ for any integer $m\neq\pm1$.
Let $M_{\beta}\subseteq G$ be a Levi subgroup containing $M$ such that 
$\ka_{M_{\beta}}=\{X\in\ka_M\mid \beta(X)=0\}$.
Let $P_{\beta}\in\CmP^{M_{\beta}}(M)$ be the unique parabolic subgroup of $M_{\beta}$ such that the unique root of $A_M$ on the unipotent radical of $P_\beta$ equals $\beta$. Suppose that $P$, $P_1\in\CmP(M)$, $P_1=MN_1$, are such that $P\cap M_{\beta}=P_{\beta}$ and $P_1\cap M_{\beta}=\overline{P_{\beta}}$ (the opposite parabolic), and write $P_{\beta}=L_{\beta}N_{\beta}$ for the Levi decomposition of $P_{\beta}$ with $M\subseteq L_{\beta}$. 

Suppose that $\pi=u\nu\in \CmU_M(\R) \overline{N_{\beta}}(\R)$. Then for any $a\in A_{M, \text{reg}}$ there is a unique $n_{\beta}\in \overline{N_{\beta}}(\R)$ such that
\begin{equation}\label{eq:relation:unipotent}
 a\pi= n_{\beta}^{-1} a u n_{\beta}.
\end{equation}
Note that $n_{\beta}$ is independent of the $A_{M_{\beta}}$-part of $a$, that is, it only depends on $a^{\beta}$. Let $\Wt(\ka_M)\subseteq X(A_M)$ be the sublattice of all $\varpi$ which are extremal weights for some finite dimensional representation of $G(\R)$. 
 Let $\varpi\in \Wt(\ka_M)\subseteq\ka_M^*$ be such that 
$\varpi(\beta^{\vee})>0$. Consider 
\[
 v_P(\varpi, n_{\beta}):=e^{-\varpi(H_P(n_{\beta}))}
=e^{-\varpi(H_{P_{\beta}}(n_{\beta}))}=v_{P_{\beta}}(\varpi,n_{\beta})
\]
as a function of $a\in A_M/A_{M_{\beta}}\simeq A_M^{M_{\beta}}$, and $\pi$ as above.

Write $\nu=\Mid+X$ and $n_{\beta}=\Mid+Y_{\beta}$ with $X,Y_\beta\in\overline{\kn_{\beta}}(\R)$.  Note that $a^{-1}n_{\beta} a=\Mid+ a^{\beta} Y_{\beta}$. Further note that $Y_{\beta}^2=0$ since $2\beta\not\in\Phi(A_M, G)$.
Hence
\begin{align*}
\pi= u+uX & = a^{-1} n_{\beta}^{-1} au  n_{\beta}
= (\Mid + a^{\beta} Y_{\beta})^{-1} u (\Mid+Y_{\beta})\\
& = u + (1-a^{\beta}) Y_{\beta} u + [u, Y_{\beta}]  - a^{\beta} Y_{\beta} [u, Y_{\beta}]  
\end{align*}
where $[u,Y_{\beta}]=uY_{\beta}-Y_{\beta} u$ is again nilpotent and contained in the Lie algebra of $\overline{N_{\beta}}(\R)$.
Again, this implies that the term $Y_{\beta}[u,Y_{\beta}]$ vanishes since $2\beta$ is not a root of $(A_M, G)$. Hence
\[
 u+uX
=u +(1-a^{\beta}) Y_{\beta} u + [u, Y_{\beta}].
\]
Let $Q_0=M_0V_0$ be a semistandard minimal parabolic subgroup containing $N_1$, so in particular it contains $\overline{N_{\beta}}$ as well. 
Conjugating $u$ by some element in $\cpt^M:=\cpt\cap M(\R)$ if necessary, we can assume that $u\in V_0(\R)\cap M(\R)$. In particular, we can write $u=\Mid+X_0$ with $X_0$ a nilpotent matrix in the Lie algebra of $V_0(\R)\cap M(\R)$. Then $[u, Y_{\beta}]=[X_0, Y_{\beta}]$.
Hence the above equality becomes
\begin{equation}\label{eq:comparing:matrices}
uX= X+X_0X 
= (1-a^{\beta}) Y_{\beta} (\Mid+X_0) +  [X_0, Y_{\beta}]
= (1-a^{\beta}) Y_{\beta}  +  X_0Y_{\beta} -a^{\beta}Y_{\beta}X_0  .
\end{equation}

$Q_0$ determines a choice of positive  reduced roots $\Phi_{Q_0}^+:=\Phi(A_0,Q_0)$. Then there exists $\beta'\in\Phi_{Q_0}^+$ with $\beta'_{|A_M}=-\beta$, and we denote by $\Psi_{\beta}\subseteq \Phi_{Q_0}^+$ the subset of all such $\beta'$.
Let $\Phi_{Q_0}^{M,+}\subseteq\Phi_{Q_0}^+$ denote the subset of positive roots on $M$. Then $\Psi_{\beta}\cap\Phi^{M,+}_{Q_0}=\emptyset$.

We have a partial order $\prec_{\beta}$ on $\Psi_{\beta}$: If $\alpha_1, \alpha_2\in\Psi_{\beta}$, then $\alpha_1\prec_{\beta}\alpha_2$ if and only if there exists $\gamma\in\Phi_{Q_0}^+$ with $\alpha_2=\alpha_1+\gamma$. 
Note that then $\gamma\in\Phi_{Q_0}^{M,+} $.
We define subsets of $\Psi_{\beta}$ according to how much the elements fail to be minimal with respect to $\prec_{\beta}$: Let $\Psi_{\beta}^{(0)}$ be the set of all $\alpha\in\Psi_{\beta}$ which are minimal with respect to $\prec_{\beta}$. If $k\geq0$ is a non-negative integer, let $\Psi_{\beta}^{(k+1)}$ be the set of all $\alpha\in\Psi_{\beta}\backslash\Psi_{\beta}^{(k)}$ which can be written as $\alpha=\alpha_1+\gamma$ with $\alpha_1\in\Psi_{\beta}^{(k)}$ and $\gamma\in\Phi^{M,+}_{Q_0}$. Note that $\Psi_{\beta}^{(k)}\neq\emptyset$ for only finitely many $k$. Moreover,  
$\alpha_1\in \Psi^{(l)}_\beta$ implies that $\alpha_2\prec_{\beta}\alpha_1$ for any 
$\alpha_2\in\Psi_{\beta}^{(k)}$ with $k<l$.

To recover the matrix entries of $Y_{\beta}$ from~\eqref{eq:comparing:matrices} we now proceed inductively over $l$ by considering the matrix entries correspond to roots in $\Psi_{\beta}^{(l)}$. 
In the following we write 
\[
Z:=X+X_0X .
\]
Note that 
\[
\pi=\Mid +X_0+Z,\text{ so that }
\pi_s=\Mid +sX_0+sZ.
\]
If $Y\in\kg(\R)$ is a matrix in the Lie algebra, and $\alpha\in\Phi(A_0,G)$ we denote by $Y^{(\alpha)}\in \R$ the matrix entry of $Y$ corresponding to $\alpha$.

\begin{lem}
Let $l\geq0$, and $\alpha\in\Psi_{\beta}^{(l)}$. 
There are rational polynomials $P_{\alpha, i}$, $0\leq i\leq l$  in the variables $X_0^{(\gamma)}$ ($\gamma\in\Phi_{Q_0}^{M,+}$), $Z^{(\gamma)}$ ($\gamma\in\Psi_{\beta}$), and $a^{\beta}$ 
such that for $a\neq1$ the following holds
\begin{itemize}
\item as a polynomial of $Z$ and $X_0$, $P_{\alpha,i}(Z, X_0, a^{\beta})$ is homogeneous of degree $(i+1)$ in the matrix entries of $Z+X_0$, so in particular, for every $s\in \R$ we have $P_{\alpha,i}(sZ, sX_0, a^{\beta})=s^{i+1}P_{\alpha,i}(Z, X_0, a^{\beta})$,

\item and
\[
 Y_{\beta}^{(\alpha)}= 
\sum_{i=0}^{l} \frac{P_{\alpha,i}(Z, X_0, a^{\beta})}{(1-a^{\beta})^{i+1}}.
\]

\end{itemize}
\end{lem}

\begin{proof}
We prove the lemma by induction on $l$.
If $l=0$, then for any $\alpha\in\Psi_{\beta}^{(0)}$ we obtain from~\eqref{eq:comparing:matrices} that
\[
 Y_{\beta}^{(\alpha)}
= \frac{ Z^{(\alpha)}}{(1-a^{\beta})}
\]
so the assertion of the lemma is true for $l=0$.

Now suppose that for some non-negative integer $l\geq0$ we know that for every $0\leq k\leq l$ and every $\alpha\in \Psi_{\beta}^{(k)}$ we have
\begin{equation}\label{eq:inductive:def:Y}
 Y_{\beta}^{(\alpha)} 
=\sum_{i=0}^{k}\frac{P_{\alpha,i}(Z,X_0, a^{\beta})}{(1-a^{\beta})^{i+1}}
\end{equation}
with $P_{\alpha, i}$ polynomials satisfying the assertions of the lemma.
Then for  $\alpha\in\Psi_{\beta}^{(k+1)}$ the equation~\eqref{eq:comparing:matrices} gives
\[
 Z^{(\alpha)}
= (1-a^{\beta})Y_{\beta}^{(\alpha)}
+ (X_0Y_{\beta})^{(\alpha)} - a^{\beta} (Y_{\beta} X_0)^{(\alpha)}.
\]
By definition of $\Psi_{\beta}^{(k+1)}$ we have
\[
 (X_0Y_{\beta})^{(\alpha)} - a^{\beta} (Y_{\beta} X_0)^{(\alpha)}
=\sum_{\substack{ \gamma\in \Psi_{\beta}^{(k)},~\delta\in\Phi_{Q_0}^{M,+}: \\ \gamma+\delta=\alpha}}  e_{\alpha}^{\gamma,\delta} Y_{\beta}^{(\gamma)} X_{0}^{(\delta)} 
\]
with
\[
 e_{\alpha}^{\gamma,\delta}= \begin{cases}
                              1						&\text{if }[E_{\delta},E_{\gamma}]=E_{\delta+\gamma},\\
			      -a^{\beta}				&\text{if }[E_{\delta},E_{\gamma}]=-E_{\delta+\gamma}
                             \end{cases}
\]
where $E_{\gamma}, E_{\delta}$ denote the elements of the standard Chevalley basis 
attached to our root system $\Phi_{Q_0}^+$.
By the inductive assumption we can insert~\eqref{eq:inductive:def:Y} for  $Y_{\beta}^{(\gamma)}$ for every $\gamma$ occurring in the sum. 
Dividing both sides of the so obtained equality by $(1-a^{\beta})$ then yields the assertion of the lemma.
\end{proof}

\subsection{Weight functions}

We now consider the function 
$v_{P_{\beta}}(\varpi, n_{\beta})=e^{-\varpi(H_P(n_{\beta}))}$ for $\varpi\in\Wt(\ka_M)$. Note that $H_P$ is 
invariant under left and right multiplication with 
elements of $\cpt^M$. There is an irreducible representation $\Lambda_{\varpi}$ 
of $G$ on a finite dimensional vector space $V_{\varpi}$, defined over $\R$, 
together with an extremal vector $\phi_{\varpi}\in V_{\varpi}(\R)$ of weight 
$\varpi$ and a norm $\|\cdot\|$ on $V_{\varpi}(\R)$ such that 
$\|\phi_{\varpi}\|=1$ and 
\[
 v_{P_{\beta}}(\varpi, n)= \|\Lambda_{\varpi}(n^{-1}) \phi_{\varpi}\|.
\]
We can identify $V_{\varpi}(\R)$ with $\R^{m_{\varpi}}$ for $m_{\varpi}=\dim V_{\varpi}$ and we can assume that the norm is of the form
\[
 \|(x_1,\ldots, x_{m_{\varpi}})\|= \left(x_1^2+\ldots+x_{m_{\varpi}}^2\right)^{1/2}
\]
for $(x_1,\ldots, x_{m_{\varpi}})\in \R^{m_{\varpi}}$. 
Recall that $\Lambda_\varpi$
is an algebraic representation. Thus $\Lambda_\varpi\colon G\to\GL(V_\varpi)$ 
is a morphism of algebraic varieties.
Hence as a function of $n_{\beta}\in \overline{N_{\beta}}(\R)$ the function $v_{P_{\beta}}(\varpi, n_{\beta})$ is the vector norm applied to a polynomial function from $\overline{\kn_{\beta}}(\R)$ to $\R^{m_\varpi}$.
Using the above lemma (and the notation therein), we therefore get that
\begin{equation}\label{eq:weight}
v_{P_{\beta}}(\varpi, n_{\beta})^2= \sum_{\kappa\in R} \frac{f_{\kappa}(Z, X_0, a^{\beta})}{(a^{\beta}-1)^{\kappa}}
\end{equation}
where $\kappa$ runs over a finite set of integers $R\subseteq\Z$, and $f_{\kappa}$ is the norm of some rational function that is homogeneous of degree $\kappa$ in the matrix entries of $X_0$ and $Z$, and has a finite value  at $a=1$ so that we may write $f_{\kappa}(Z, X_0, 1)$. Here $n_\beta$, $Z$, $X_0$, and $a$ are related as in \eqref{eq:relation:unipotent}. The function $f_{\kappa}(Z, X_0, a^{\beta})$ in general depends on $\varpi$ and if we want to make this dependence explicit we write $f_{\kappa, \varpi}(Z, X_0, a^{\beta})$.

Now let $u\in U$, $U\subseteq M(\R)$, and $U^G\subseteq G(\R)$ be as before, and write $u=\Mid +X_0$ with $X_0$ nilpotent.
 Let $\kappa_0(\beta,X_0)\in R$ be the largest $\kappa\in R$ such that $f_{\kappa}(\cdot, \cdot,a^{\beta})$ does not vanish identically on $\overline{\kn_{\beta}}(\R)\times \CmN$, where $\CmN\subseteq \km(\R)$ is the nilpotent orbit defined by $U=\Mid+\CmN$. Let $\rho(\beta, X_0)$ be the product of $1/(2\varpi(\beta^{\vee}))$ with $\kappa_0(\beta, X_0)$.  It follows from~\cite[p. 238]{Ar5} that $\kappa_0(\beta,X_0)\ge0$ and $\rho(\beta, X_0)\ge0$. The $\rho(\beta,X_0)$ is independent of $\varpi$ as explained in \cite[p. 238]{Ar5} but $\kappa_0(\beta, X_0)$ in general depends on $\varpi$. If we want to emphasize this dependence, we write $\kappa_0(\beta, X_0)=\kappa_0(\varpi, \beta, X_0)$.

 Recall the definition of the weight function $w_P(\lambda, a, \pi)$ from~\cite[(3.6)]{Ar5}: Fix a parabolic subgroup $P_1\in\CmP(M)$. Then for any other $P\in\CmP(M)$ and any $P$-dominant $\varpi\in \Wt(\ka_M)$ Arthur defines for $\pi=u\nu\in U N_1(\R)$ the function
 \begin{equation}\label{eq:weight:function:definition}
  w_P(\varpi, a, \pi) 
  =\left(\prod_{\beta\in \Phi_P\cap\Phi_{\overline{P_1}}}  r_{\beta}(\varpi,u,a)\right)  v_P(\varpi,n)
 \end{equation}
where $\Phi_P=\Phi(A_M, P)$ and $\Phi_{\overline{P_1}}=\Phi(A_M, \overline{P_1})$. Here $u$, $\nu$, $a$, and $n\in N_1(\R)$ are related by $a\pi=n^{-1} au n$. 
The function $r_{\beta}$ is given by (see~\cite[(3.4)]{Ar5})
\[
 r_{\beta}(\varpi, u, a)
 =|a^{\beta}-a^{-\beta}|^{\rho(\beta,X_0)\varpi(\beta^{\vee})}.
\]
It follows from this definition that 
\begin{equation}\label{eq:weight:function:invariance:P} 
w_P(\varpi, a, k^{-1}\pi k)
= w_P(\varpi, a, \pi)
\end{equation}
for $k=\diag(-1,1, \ldots, 1)\in \bar\cpt^M_\infty$. 

If $\varpi_1,\ldots, \varpi_r\in\Wt(\ka_M)$ is a basis of $\ka_M^*$ consisting of $P$-dominant weights, and $\lambda=\lambda_1\varpi_1+\ldots+\lambda_r\varpi_r\in\ka_{M,\C}^*$ with $\lambda_1,\ldots,\lambda_r\in\C$, then
 \begin{equation}\label{eq:product:weight:functions}
  w_P(\lambda, a, \pi)=\prod_{i=1}^r w_P(\varpi_i, a, \pi)^{\lambda_i}
 \end{equation}
 By~\cite[Lemma 4.1]{Ar5} the limit 
\[
w_P(\lambda, \pi):=\lim_{a\rightarrow 1}w_P(\lambda, a, \pi)
\]
exists and is non-zero for all $\pi$ in an open and dense subset of $U N_1(\R)$.
 
 \subsection{The adjacent case}

For a unipotent element $\pi$ recall the definition of $\pi_s$ from~\eqref{eq:def:xs}. 
If $P$ and $P_1$ are adjacent, the function $w_P(\lambda, \pi)$ has the following behavior when $\pi$ is replaced by $\pi_s$: 

\begin{lem}\label{lemma:homog:weight:adjacent}
 Suppose $P$ and $P_1$ are adjacent via $\beta\in\Phi(A_M,P)$. 
Then for all $\pi$ in an open dense subset of $U N_1(\R)$ we have
\[
w_P(\lambda, \pi_s) 
= s^{\rho(\beta, X_0) \lambda(\beta^\vee)}\cdot  w_P(\lambda,\pi)
\]
for all $s>0$.
\end{lem}

\begin{proof}
Note that $w_P(\lambda, a, \pi)=w_P(\lambda, a, k^{-1}\pi k)$ for any $k\in \cpt^M=\cpt\cap M(\R)$ so that we can assume $\pi$ to be of the form $\Mid+X_0+Z$ with $X_0\in\CmN\cap \kv_0(\R)$, and $Z\in \kn_1(\R)$ as before. Here $\kv_0$, resp.\ $\kn_1$, denotes the Lie algebra of $V_0$ (the unipotent radical of $Q_0$), resp.\ $N_1$ (the unipotent radical of $P_1$). 
 Since $P$ and $P_1$ are assumed to be adjacent along $\beta$, we have $\Phi_P\cap \Phi_{\bar P_1}=\{\beta\}$. Hence it follows from \eqref{eq:weight:function:definition} that $ w_P(\lambda, a, \pi)
= r_{\beta}(\lambda, X_0, a) v_{P}(\lambda, n)$ where
\[
 r_{\beta}(\lambda, X_0, a)=|a^{\beta}-a^{-\beta}|^{\rho(\beta,X_0)\lambda(\beta^{\vee})}
\].  
for $a$ with $a^{\beta}\neq1$.
We can write $n\in N_1$ as $n=n_\beta \tilde n$ with $n_\beta$ in the unipotent radical of $\overline{P_\beta}$ and $\tilde n$ in the unipotent radical of $P_1\cap P$.  Then $ v_{P}(\lambda, n) = v_{P_\beta}(\lambda, n_\beta)$ so that
 \[
 w_P(\lambda, a, \pi)
= r_{\beta}(\lambda, X_0, a) v_{P_{\beta}}(\lambda, n_\beta).
\]
Hence, using \eqref{eq:weight} and \eqref{eq:product:weight:functions},
\begin{align*}
 w_P(\lambda, \pi)^2
& =  \lim_{a\rightarrow1}\prod_{i=1}^r\left(|a^{\beta}-a^{-\beta}|^{2\rho(\beta,X_0)\varpi_i(\beta^{\vee})} \sum_{\kappa\in R} \frac{f_{\kappa, \varpi_i}(Z, X_0, a^{\beta})}{(a^{\beta}-1)^{\kappa}}\right)^{\lambda_i} \\
& = \lim_{a\rightarrow1}\prod_{i=1}^r\left(|a^{\beta}-a^{-\beta}|^{\kappa_0(\varpi_i, \beta, X_0)} \sum_{\kappa\in R} \frac{f_{\kappa, \varpi_i}(Z, X_0, a^{\beta})}{(a^{\beta}-1)^{\kappa}}\right)^{\lambda_i} \\
& = \prod_{i=1}^ r 2^{\lambda_i} f_{\kappa_0(\varpi_i,\beta, X_0), \varpi_i}(Z, X_0, 1)^{\lambda_i}.
\end{align*}
Now $f_{\kappa_0(\varpi_i,\beta, X_0), \varpi_i}(Z, X_0, 1)$ is non-zero for generic $X_0$, $Z$, and  satisfies 
\[
 f_{\kappa_0(\varpi_i,\beta, X_0), \varpi_i} (sZ, sX_0, 1)
 = s^{\kappa_0(\varpi_i, \beta,X_0)}f_{\kappa_0(\varpi_i,\beta, X_0), \varpi_i}  (Z, X_0, 1).
\]
Hence
\[
  w_P(\lambda, \pi_s)^2
  =\prod_{i=1}^r s^{\lambda_i \kappa_0(\varpi_i, \beta,X_0)} \prod_{i=1}^r 2^{\lambda_i} f_{\kappa_0(\varpi_i,\beta, X_0), \varpi_i}(Z, X_0, 1)^{\lambda_i}
  = s^{2\rho(\beta, X_0) \lambda(\beta^\vee)}   w_P(\lambda, \pi)^2
\]
since $\sum_{i=1}^r \lambda_i \kappa_0(\varpi_i, \beta,X_0)= 2 \rho(\beta, X_0) \sum_{i=1}^r \lambda_i \varpi_i(\beta^\vee)= 2\rho(\beta, X_0) \lambda(\beta^\vee)$. Since $w_P(\lambda, \pi_s)$ is a real-valued function and continuous in $s$ with $w_P(\lambda, \pi_1)=w_P(\lambda, \pi)$, we can take the square-root on both sides of the equation and obtain the assertion of the lemma. 
\end{proof}

\subsection{The general case}

If $Q, Q'\in \CmP(M)$ are adjacent along some root $\beta\in\Phi(A_M,Q)$, we write $Q|_{\beta}Q'$.
\begin{corollary}
 Suppose that $P$ and $P_1$ are not necessarily adjacent. Choose a  minimal 
chain $P=Q_0|_{\beta_1}Q_1|_{\beta_2}\ldots|_{\beta_t} Q_t=P_1$  of adjacent 
parabolic subgroups $Q_1, \ldots, Q_t\in\CmP(M)$ from $P$ to $P_1$.
Then there exist  rational numbers $\rho_1, \ldots, \rho_t$ such that for all $\pi$ in an open dense subset of $U N_1(\R)$ we have
\[
w_P(\lambda, \pi_s) 
= s^{\rho_1 \lambda(\beta_1^{\vee})+\ldots+\rho_t\lambda(\beta_t^{\vee})}\cdot w_P(\lambda,\pi)
\]
for all $s>0$.
\end{corollary}

\begin{proof}
 This follows by induction on $t$ together with Lemma~\ref{lemma:homog:weight:adjacent} and the proof of~\cite[Lemma 4.1]{Ar5}.

The case $t=1$ is covered in the last lemma. Let $P_1':=Q_1=MN_1'$, and assume that the corollary is true for $P$ replaced by $P_1'$. 
Let $\pi= u\nu\in U N_1(\R)$, and $a\in A_{M, \text{reg}}$. Let $n\in N_1(\R)$ be the unique element with $\pi=a^{-1}n^{-1}aun$. Write $n=m'n'k'$ with $m'\in M(\R)$, $n'\in N_1'(\R)$, and $k'\in \cpt$, and put $u'=(m')^{-1} u m'$. Let $\pi'=a^{-1} (n')^{-1} a u' n'\in U N_1'(\R)$. Then by the proof of~\cite[Lemma 4.1, p. 241]{Ar5}, we have
\[
 w_P(\lambda,a,\pi)=w_P(\lambda,a,\pi')w_{P_1'}(\lambda,a,\pi).
\]
Suppose that $\pi=\Mid+ Y$. Then $\pi'=\Mid + k'Yk^{\prime-1}$ so that that the map $\pi\mapsto \pi'$ also maps $\pi_s$ to $\pi'_s $ for any $s$. We are further allowed to take the value at $a= 1$ on both sides because of~\cite[Lemma 4.1]{Ar5}. The assertion of the corollary therefore follows from Lemma~\ref{lemma:homog:weight:adjacent} and the induction hypothesis.
\end{proof}

Arthur defines polynomials $W_P(\varpi,a,\pi)\in V_{\varpi}$ for a $P$-dominant weight $\varpi\in \Wt(\ka_M)$ and $(a,\pi)\in A_{M,\text{reg}}\times U N_1(\R)$ by
 \[
  W_P(\varpi,a,\pi)
  =\left(\prod_{\beta\in \Phi_P\cap\Phi_{\overline{P_1}}}  r_{\beta}(\varpi,u,a)\right) \Lambda_{\varpi}(n^{-1}) \phi_{\varpi}
 \]
so that $  w_P(\varpi, a, \pi) =\|W_P(\varpi,a,\pi)\|$ and
\begin{equation}\label{eq:weight:as:prod}
  w_P(\lambda, a, \pi)
  =\prod_{i=1}^r \|W_P(\varpi_i,a,\pi)\|^{\lambda_i}.
\end{equation}
Here $\pi$, $u$, $a$, and $n$ are related as explained after \eqref{eq:weight:function:definition}

\begin{corollary}\label{cor:homog:pol}
If $\varpi\in\Wt(\ka_M)$ is a $P$-dominant weight, the polynomial $W_P(\varpi,a,\pi)$ is defined on all of $A_M\times U N_1$ and does not vanish at $a=1$. Moreover, there is a constant $r_{\varpi}$ depending only on $P$ and $\varpi$ such that for all $\pi\in U N_1(\R)$ and all $s>0$ we have 
\[
 \|W_P(\varpi,1,\pi_s)\|=s^{r_{\varpi}}  \|W_P(\varpi,1,\pi)\|.
\]
\end{corollary}
\begin{proof}
 All assertions except the homogeneity are subject of~\cite[Corollary 4.3]{Ar5}. The homogeneity follows from the previous corollary and the definition of $W_P(\varpi,a,\pi)$.  
\end{proof}

 Let $\theta_P(\lambda)= v_P^{-1}\prod_{\alpha\in\Delta_P} \lambda(\alpha^{\vee})$ for $\lambda\in \ka_M^*$, where $v_P$ denotes the covolume of the lattice spanned by all $\alpha^{\vee}$, $\alpha\in\Delta_P$, in $\ka_M$. Then $\{w_P(\lambda,a,\pi)\}_{P\in\CmP(M)}$ defines a $(G,M)$-family, and one can attach a certain number to this family by defining
 \begin{equation}\label{eq:definition:weight}
  w_M(a,\pi)=\lim_{\lambda\rightarrow 0} \left(\sum_{P\in\CmP(M)} w_P(\lambda,a,\pi)\theta_P(\lambda)^{-1}\right)
 \end{equation}
as in~\cite[\S 6]{Ar5}. By~\cite[(6.5)]{Ar3} this can be computed by
\[
  w_M(a,\pi)
  =\frac{1}{r!} \sum_{P\in\CmP(M)}\left(\lim_{t\rightarrow 0}\frac{d^r}{dt^r} w_P(t\Lambda,a,\pi) \right)\theta_P(\Lambda)^{-1}
\]
with $r=\dim\ka_M^G$ and $\Lambda\in\ka_M^*$ some fixed generic element. (Note that $w_M(a,\pi)$ is independent of the choice of $\Lambda$.)
Using~\eqref{eq:weight:as:prod} we get (cf.~\cite[Lemma 5.4]{Ar5})
\[
  w_M(a,\pi) = \sum_{\Omega} c_{\Omega} \prod_{(P,\varpi)\in\Omega} \log\|W_P(\varpi,a,\pi)\|
\]
where $\Omega$ runs over all finite multisets consisting of elements in  $\mathcal{P}(M)\times \Wt(\ka_M)$, and the $c_{\Omega}\in\C$ are suitable coefficients which vanish for all but finitely many $\Omega$.
Moreover, each $\Omega$ contains at most $r=\dim\ka_M^G$ many elements. 
Note that in a neighborhood of $a=1$ this expression is well-defined for all $\pi$ in an open dense subset of $U N_1$ because of Corollary~\ref{cor:homog:pol}.
Hence we can evaluate $w_M(a,\pi)$ at $a=1$ by means of this expression. Moreover, Corollary~\ref{cor:homog:pol} implies the following result.

\begin{corollary}\label{cor:homog:weight}
For all $\pi$ in an open dense subset of $U N_1(\R)$ and all $s>0$ 
\[
  w_M(1,\pi_s) = \sum_{\Omega} c_{\Omega} \prod_{(P,\varpi)\in\Omega}(r_{\varpi}\log s+ \log\|W_P(\varpi,1,\pi)\|).
\]
In particular, as a function of $s$, $w_M(1,\pi_s)$ is a polynomial in $\log s$ of degree at most $r=\dim\ka_M^G$.
\end{corollary}

If now $Q'\in \CmF(M)$ is an arbitrary subgroup, we can analogously define all the above functions with respect to the Levi component $M_{Q'}$ instead of $G$, in particular we can define the analogue of $w_M(1,\pi)$ which we denote by $w_M^{Q'}(1,\pi)$. The corollary then stays true for $w_M^{Q'}(1,\pi)$ with the necessary changes. Note that a priori $w_M^{Q'}(1,\pi)$ is defined for $\pi$ in a dense open subset of $U N_1^{M_{Q'}}$ for $P_1^{M_{Q'}}=P_1\cap M_{Q'}$. However, we can trivially extend $w_M^{Q'}(1,\pi)$ to a dense open subset of $U^G$. 
By the first equation on \cite[p. 256]{Ar5} the weight function $w_{M,\cV}$ from~\eqref{eq:integral} can then be written as 
\[
 w_{M,\cV}(\pi) = \sum_{Q'\in \CmF(M)} w_M^{Q'}(1,\pi).
\]
This together with Corollary~\ref{cor:homog:weight} implies the Proposition~\ref{prop}.

\subsection{Convergence}

\begin{lem}\label{logconv}
Let $p_i:\R^k\longrightarrow \R$, $i=1,...,l$, be homogeneous polynomials and put
\begin{equation}\label{eq:log:homog:polynom}
\lambda(x):=\prod_{i=1}^l\big|\log|p_i(x)|\big|,\quad x\in\R^k.
\end{equation}
Then for every $a>0$
\begin{equation}\label{eq:log:int2}
 \int_{\R^k}e^{-a\left(\log(1+\|x\|)\right)^2}\lambda(x)\;dx< \infty.
\end{equation}
\end{lem}

\begin{proof}
Since $\prod_{i=1}^l\big|\log|p_i(x)|\big|\le \sum_{i=1}^l \big|\log|p_i(x)|\big|^l$, it suffices to consider the case $\lambda(x)=\big|\log|p(x)|\big|^l$ with $p:\R^k\longrightarrow\R$ a homogeneous polynomial of degree $\kappa$.

Further note that for every $M\in\N$ there exists $r_M>0$ such that $\left(\log(1+\|x\|)\right)^2\ge (M/a) \log \|x\|$ for all $x$ with $\|x\|\ge r_M$. In particular, 
\begin{equation}\label{eq:polynomial:decay}
 e^{-a\left(\log(1+\|x\|)\right)^2}
 \le \|x\|^{-M}
\end{equation}
for all $x$ with $\|x\|\ge r_M$.

We decompose the integral in~\eqref{eq:log:int2} into a sum of integrals over certain subsets of $\R^k$ similarly as in~\cite[\S 7]{Ar5}.
For each $m\in\N_0$ let
\[
 B_m=\{x\in\R^k\mid 2^m\le \|x\|\le 2^{m+1}\}
\]
and let $B=\{x\in\R^k\mid \|x\|\le 2\}$. For $\eps>0$ set
\[
  \Gamma_{\eps}=\{x\in B\mid |p(x)|<\eps\}, 
  \text{ and }
  \Gamma_{m,\eps}=\{x\in B_m\mid |p(x)|<\eps\}.
\]
By \cite[(7.1)]{Ar5} there are constants $C, t>0$ such that for every $\eps>0$ we have
\begin{equation}\label{eq:integral:log}
\int_{\Gamma_\eps} \lambda(x)\, dx
\le C \eps^t.
\end{equation}
For $m\ge0$ we can therefore compute, using that $p$ is homogeneous of degree $\kappa$, that
\begin{align*}
 \int_{\Gamma_{m,\eps}} \lambda(x)\, dx
  &\le 2^{mk}\int_{\Gamma_{2^{-m\kappa}\eps}} \lambda(2^m y)\, dy
 \le c_1 2^{mk} m^l \int_{\Gamma_{2^{-m\kappa}\eps}} (1+\lambda( y))\, dy
 \le c_2 2^{mk} m^l 2^{-m\kappa t_1}\eps^{t_1} \\
 & \le c_3 2^{m (k+l)} \eps^{t_1}
\end{align*}
where $c_1, c_2, c_3, t_1>0$ are suitable constants independent of $m$ and $\eps$ and the second last inequality follows from \eqref{eq:integral:log}. Fix $\eps>0$ and let $Z_\eps=\{x\in\R^k\mid |p(x)|<\eps\}$.

Then, using \eqref{eq:polynomial:decay} for every $M\in\N$ there exist $c_4, r_M>0$ such that
\begin{align*}
&  \int_{Z_\eps}e^{-a\left(\log(1+\|x\|)\right)^2}\lambda(x)\;dx \\
  \le & \int_{\Gamma_\eps} \lambda(x)\, dx
    + \sum_{m\le \log_2 r_M} c_4\int_{\Gamma_{m,\eps}} \lambda(x)\, dx
  + \sum_{m\ge \log_2 r_M} 2^{-mM}\int_{\Gamma_{m,\eps}} \lambda(x)\, dx \\
  \le  & C\eps ^t + c_5 \eps^{t_1}\left( 1+  \sum_{m\ge\log_2 r_M} 2^{-mM} 2^{m(k+l)}\right).
\end{align*}
This is finite if we choose $M$ sufficiently large.

Now on $\R^k\backslash Z_\eps$ the polynomial $p(x)$ is bounded away from $0$ so that $\log|p(x)|$ is bounded from below on $\R^k\backslash Z_\eps$. Since $p(x)$ is of degree $\kappa$, there is a constant $A>0$ such that $|p(x)|\le A (1+\|x\|)^\kappa$ for all $x\in \R^k$. 
Choose $M$ again sufficiently large and let $r_M$ be as in \eqref{eq:polynomial:decay}. Then for some suitable constant $A_1>0$ we get
\[
  \int_{\R^k\backslash Z_\eps} e^{-a\left(\log(1+\|x\|)\right)^2} \lambda(x)\, dx
  \le A_1 \int_{\R^k\backslash Z_\eps} \|x\|^{-M} \left(1+(\log(1+\|x\|))^l\right)\, dx,
\]
which is finite if $M$ was chosen sufficiently large.

\end{proof}

\section{Examples for weight functions in low rank}\label{sec-lowrank}
\setcounter{equation}{0}

\subsection{$G=\GL(2)$}
There are two unipotent conjugacy classes in $\GL(2)$, the trivial class for which  Richardson parabolic subgroup equals $G$, and the regular unipotent conjugacy class  with Richardson parabolic equal to the minimal parabolic subgroup $P_0=M_0U_0$. The archimedean orbital integrals appearing in the fine expansion of $J_{\unip}$ are $J_G(1,f_{\infty})$, $J_G(u_0,f_{\infty})$, and $J_{M_0}(1, f_{\infty})$, where $u_0=\left(\begin{smallmatrix} 1& 1\\ 0&1\end{smallmatrix}\right)$ represents the regular class in $G(\Q)$. The first two integrals are unweighted, and the last integral $J_{M_0}(1, f_0)$ is up to a normalization of Haar measure equal to
\[
\int_{U_0(\R)} f_{\infty}(\left(\begin{smallmatrix} 1& x\\ 0&1\end{smallmatrix}\right)) \log |x| \, dx,
\]
see, e.g., \cite{Gelbart}.

\subsection{$G=\GL(3)$}
There are three unipotent conjugacy class in $\GL(3)$: The trivial, the regular, and the subregular class. Let
\[
 u(x,y,z)=\left(\begin{smallmatrix} 1&x&y\\ \, & 1&z\\\,&\,& 1\end{smallmatrix}\right)
\]
for $x, y, z\in\R$. Then $u(0,1,0)$ is a representative for the subregular class, and $u(1,0,1)$ a representative for the regular class.
Let $M_1$ be the Levi subgroup corresponding to the partition $(2,1)$ of $3$. Every other corank-$1$ Levi subgroup in $\levis$ is conjugate to $M_1$ by some Weyl group element so that it suffices to consider $J_M(u,f_{\infty})$ for $M\in\{M_0, M_1, G\}$. The integrals $J_G(1,f_{\infty})$, $J_G(u(0,1,0), f_{\infty})$, and $J_G(u(1,0,1), f_{\infty})$ are all unweighted. For the other cases we get (up to normalization of the measures)
\begin{equation}\label{unipot3}
J_{M_1}(1, f_{\infty})=\int_{U_1(\R)} f_{\infty}(u(0,y,z)) \log(y^2+z^2) 
\,du(0,y,z),
\end{equation}
\begin{equation}\label{unipot4}
J_{M_1}(u(1,0,0), f_{\infty})  =\int_{U_0(\R)} f_{\infty}(u(x,y,z))  \log|xz| 
\, du(x,y,z),
\end{equation}
and
\begin{equation}\label{unipot5}
J_{M_0}(1, f_{\infty}) 	 =\int_{U_0(\R)} f_{\infty}(u(x,y,z))  
\left(\log|x|\log|z|+ (\log|x|)^2+ (\log|z|)^2\right) \, du(x,y,z),
\end{equation}
cf.\ \cite[p. 67, Lemma 4]{Flicker}.

\section{Bochner Laplace operators}\label{sec-bochlapl}
\setcounter{equation}{0}

In this section we summarize some basic facts about Bochner-Laplace operators
on global Riemannian symmetric spaces.
For simplicity we assume that $G$ is semisimple and $G(\R)$ is 
of noncompact type.  Then $G(\R)$ is a semisimple real Lie group of noncompact 
type. Let
$K_\infty\subset G(\R)$ be a maximal compact subgroup and
\[
\widetilde X=G(\R)/K_\infty
\]
the associated Riemannian symmetric space. 
Let $\Gamma\subset G(\R)$ be a torsion free lattice and let 
$X=\Gamma\bs\widetilde X$. 
Let $\nu$ be a finite-dimensional unitary representation of $K_\infty$ on 
$(V_{\nu},\left<\cdot,\cdot\right>_{\nu})$. Let
\begin{align*}
\widetilde{E}_{\nu}:=G(\R)\times_{\nu}V_{\nu}
\end{align*}
be the associated homogeneous vector bundle over $\tilde{X}$. Then 
$\left<\cdot,\cdot\right>_{\nu}$ induces a $G(\R)$-invariant metric 
$\tilde{h}_{\nu}$ on $\tilde{E}_{\nu}$. Let $\widetilde{\nabla}^{\nu}$ be the 
connection on $\tilde{E}_{\nu}$ induced by the canonical connection on the
principal $K_\infty$-fibre bundle $G(\R)\to G(\R)/K_\infty$. Then 
$\widetilde{\nabla}^{\nu}$ is $G(\R)$-invariant.
Let  
\begin{align*}
E_{\nu}:=\Gamma\bs \widetilde E_\nu
\end{align*}
be the associated locally homogeneous vector bundle over $X$. Since 
$\tilde{h}_{\nu}$ and $\widetilde{\nabla}^{\nu}$ are $G(\R)$-invariant, they push 
down to a metric $h_{\nu}$ and a connection $\nabla^{\nu}$ on $E_{\nu}$. Let
$C^\infty(\widetilde X,\widetilde E_\nu)$ resp. $C^\infty(X,E_\nu)$ denote the 
space of smooth sections of $\widetilde E_\nu$, resp. $E_\nu$. 
Let
\begin{align}\label{globsect}
\begin{split}
C^{\infty}(G(\R),\nu):=\{f:G(\R)\rightarrow V_{\nu}\colon f\in C^\infty,\:
f(gk)=&\nu(k^{-1})f(g),\\
&\forall g\in G(\R), \,\forall k\in K_\infty\},
\end{split}
\end{align}
Let $L^2(G(\R),\nu)$ be the corresponding $L^2$-space. There is a canonical
isomorphism
\begin{equation}\label{iso-glsect}
\widetilde A\colon C^\infty(\widetilde X,\widetilde E_\nu)\cong 
C^\infty(G(\R),\nu)
\end{equation}
(see \cite[p. 4]{Mia}). $\widetilde A$ extends to an isometry of the 
corresponding $L^2$-spaces. Let
\begin{align}\label{globsect1}
C^{\infty}(\Gamma\backslash G(\R),\nu):=\left\{f\in C^{\infty}(G(\R),\nu)\colon 
f(\gamma g)=f(g)\:\forall g\in G(\R), \forall \gamma\in\Gamma\right\}
\end{align}
and let $L^2(\Gamma\bs G(\R),\nu)$ be the corresponding $L^2$-space. The
isomorphism \eqref{iso-glsect} descends to isomorphisms
\begin{equation}\label{iso-glsect1}
 A\colon C^{\infty}(X,E_{\nu})\cong C^{\infty}(\Gamma\backslash G(\R),\nu),\quad
L^2(X,E_\nu)\cong L^2(\Gamma\bs G(\R),\nu).
\end{equation}
Let
$\widetilde{\Delta}_{\nu}={\widetilde{\nabla^\nu}}^{*}{\widetilde{\nabla}}^{\nu}$
be the Bochner-Laplace operator of $\widetilde{E}_{\nu}$. This is a 
$G(\R)$-invariant second order elliptic differential operator whose principal
symbol is given by
\[
\sigma_{\widetilde\Delta_\nu}(x,\xi)=\|\xi\|^2_x\cdot\Id_{E_{\nu,x}},\quad 
x\in\widetilde X, \;\xi\in T^\ast_x(\widetilde X).
\] 
Since 
$\widetilde{X}$ is complete, 
$\widetilde{\Delta}_{\nu}$ with domain the smooth compactly supported sections
is essentially self-adjoint \cite[p. 155]{LM}. Its self-adjoint extension
will be denoted by $\widetilde{\Delta}_\nu$ too. Let $\Omega\in\cZ(\gf_\C)$ 
and $\Omega_{K_\infty}\in\cZ(\kf)$ be the Casimir operators of $\gf$ and
$\kf$, respectively, where the latter is defined with respect to  the
restriction of the normalized Killing form of $\mathfrak{g}$ to $\mathfrak{k}$.
Let  $C^\infty(G(\R),V_\nu)$ be the space of smooth $V_\nu$-valued functions on
$G(\R)$ and $R$ the right regular representation of $G(\R)$ in 
in $C^\infty(G(\R),V_\nu)$. Let $R(\Omega)$ be the differential operator induced 
by $\Omega$. Since $\Ad(g)\Omega=\Omega$, $g\in G(\R)$, it follows that
$R(\Omega)$ preserves the subspace $C^\infty(G(\R),\nu)$.
Then with respect to the isomorphism \eqref{iso-glsect} we have
\begin{align}\label{BLO}
\widetilde{\Delta}_{\nu}=-R(\Omega)+\nu(\Omega_{K_\infty}),
\end{align} 
(see \cite[Proposition 1.1]{Mia}). 
Let $e^{-t\widetilde\Delta_{\nu}}$, $t>0$, be the heat semigroup generated by 
$\widetilde\Delta_\nu$. It commutes with the action of $G(\R)$. With respect
to the isomorphism \eqref{iso-glsect} we may regard $e^{-t\widetilde\Delta_{\nu}}$ as
bounded operator in $L^2(G(\R),\nu)$, which commutes with the action of $G(\R)$.
Hence it is a convolution operator, i.e., there exists a smooth map
\begin{equation}\label{heatker1}
H_t^\nu\colon G(\R)\to \End(V_\nu)
\end{equation}
such that
\[
(e^{-t\widetilde\Delta_{\nu}}\phi)(g)=\int_{G(\R)} H_t^\nu(g^{-1}g^\prime)(\phi(g^\prime))\;
dg^\prime,\quad \phi\in L^2(G(\R),\nu).
\]
The kernel $H_t^\nu$ satisfies
\begin{align}\label{propH}
{H}^{\nu}_{t}(k^{-1}gk')=\nu(k)^{-1}\circ {H}^{\nu}_{t}(g)\circ\nu(k'),
\:\forall k,k'\in K, \forall g\in G.
\end{align}
For $q>0$ let $\ccC^q(G(\R))$ be Harish-Chandra's $L^q$-Schwartz space. 
We briefly recall its definition. Let $\Xi$ and $\|\cdot\|$ be the functions
on $G(\R)$ used to define Harish-Chandra's Schwartz space $\Co(G(\R))$ (see
\cite[7.1.2]{Wal}). Furthermore, for $Y\in\cU(\gf_\C)$ denote by $L(Y)$ (resp.
$R(Y)$) the associated left (resp. right) invariant differential operator on
$G(\R)$. Then $\ccC^q(G(\R))$ consists of all $f\in C^\infty(G(\R))$ such that
\[
\sup_{x\in G(\R)}(1+\|x\|)^m\Xi(x)^{-2/q}|L(Y_1)R(Y_2)f(x)|<\infty
\]
for all $m\ge 0$ and $Y_1,Y_2\in\cU(\gf_\C)$. Note that $\ccC^2(G(\R))$
equals Harish-Chandra's Schwartz space $\Co(G(\R))$. 
 Proceeding as in the proof of \cite[Proposition 2.4]{BM} it follows 
that $H^\nu_t$ belongs to
$(\ccC^{q}(G(\R))\otimes\End(V_{\nu}))^{K_\infty\times K_\infty}$ for all $q>0$.

Let $\pi$ be a unitary representation of $G(\R)$ on a Hilbert space $\cH_\pi$.
Define a bounded operator on $\cH_{\pi}\otimes V_{\nu}$ by
\begin{align}\label{Definiton von pi(tilde(H))}
\tilde{\pi}(H^{\nu}_{t}(g)):=\int_{G(\R)}\pi(g)\otimes H^{\nu}_{t}(g)\;dg.
\end{align}
Then relative to the splitting 
\begin{align*}
\mathcal{H}_{\pi}\otimes V_{\nu}=\left(\mathcal{H}_{\pi}\otimes 
V_{\nu}\right)^{K_\infty}\oplus\left(\left(\mathcal{H}_{\pi}\otimes 
V_{\nu}\right)^{K_\infty}\right)^{\bot},
\end{align*}
$\tilde{\pi}(H^{\nu}_{t})$ has the form
\begin{align*}
\begin{pmatrix}
\pi({H}^{\nu}_{t})&0\\0&0
\end{pmatrix},
\end{align*}
where $\pi(H^{\nu}_{t})$ acts on $\left(\mathcal{H}_{\pi}\otimes 
V_{\nu}\right)^{K_\infty}$.  Assume that $\pi$ is irreducible. Let $\pi(\Omega)$
be the Casimir eigenvalue of $\pi$. Then as in \cite[Corollary 2.2]{BM} it 
follows from \eqref{BLO} that
\begin{equation}\label{integop}
\pi(H_t^\nu)=e^{t(\pi(\Omega)-\nu(\Omega_{K_\infty}))}\Id,
\end{equation}
where $\Id$ is the identity on $\left(\mathcal{H}_{\pi}\otimes 
V_{\nu}\right)^{K_\infty}$. Put
\begin{equation}\label{loctrace}
h_t^\nu(g):=\tr H_t^\nu(g),\quad g\in G(\R).
\end{equation}
Then $h_t^\nu\in\ccC^q(G(\R))$ for all $q>0$. In particular, $h_t^\nu$ belongs
to $\ccC^2(G(\R))$, which equals Harish-Chandra's Schwartz space $\Co(G(\R))$. 
Let $\pi$ be a unitary representation of $G(\R)$. Put
\[
\pi(h_t^\nu)=\int_{G(\R)}h_t^\nu(g)\pi(g)\; dg.
\]
Assume that $\pi(H_t^\nu)$ is a trace class operator. Then it follows as in 
\cite[Lemma 3.3]{BM} that $\pi(h_t^\nu)$ is a trace class operator and
\begin{equation}\label{equ-tr}
\Tr\pi(h_t^\nu)=\Tr\pi(H_t^\nu).
\end{equation}
Now assume that $\pi$ is a unitary admissible representation. Let 
$A:\mathcal{H}_\pi\rightarrow \mathcal{H}_\pi$ be a bounded operator which is an
intertwining operator for $\pi|_{K}$. Then $A\circ\pi(h_t^\nu)$ is again a
finite rank
operator. Define an operator $\tilde{A}$ on $\mathcal{H}_\pi\otimes V_\nu$ by
$\tilde{A}:=A\otimes\Id$.
Then by the same argument as in \cite[Lemma 5.1]{BM} one has
\begin{equation}\label{equtrace}
\Tr \left(\tilde{A}\circ
\tilde{\pi}(H_t^\nu)\right)=\Tr\left(A\circ\pi(h_t^\nu)\right).
\end{equation}
Together with \eqref{integop} we obtain
\begin{equation}\label{TrFT}
\Tr\left(A\circ\pi(h_t^\nu)\right)=e^{t(\pi(\Omega)-\nu(\Omega_{K_\infty}))}
\Tr\left(\tilde{A}|_{(\mathcal{H}_{\pi}\otimes V_{\nu})^{K}}\right).
\end{equation}

\section{Heat kernel estimates}\label{sec-heatkernel}
\setcounter{equation}{0}

Let the notation be as in the previous section. 
In this section we prove some estimations for the function $h_t^\nu$ defined 
by \eqref{loctrace}.  Let $\widetilde K^\nu(t,x,y)$ be 
the kernel of $e^{-t\widetilde\Delta_\nu}$. Observe that 
$\widetilde K^\nu(t,x,y)\in\Hom((\widetilde E_\nu)_y,(\widetilde E_\nu)_x)$. 
Denote by $|\widetilde K^\nu(t,x,y)|$ the norm of this homomorphism.
Furthermore, let $r(x,y)$ denote the geodesic distance of $x,y\in\widetilde X$.
\begin{prop}\label{prop-estim}
Let $d=\dim\widetilde X$. For every $T>0$ there exists $C>0$ such that we have
\[
|\widetilde K^\nu(t,x,y)|\le C t^{-d/2}  exp\left(-\frac{r^2(x,y)}{4t}\right)
\]
for all $0<t\le T$ and $x,y\in\widetilde X$.
\end{prop}
\begin{proof}
If $\nu$ is irreducible, this is proved in \cite[Proposition 3.2]{Mu1}. 
However, the proof does not make any use of the irreducibility of $\nu$. 
So it extends without any change to the case of finite-dimensional 
representations.
\end{proof}
Let $x_0:=eK_\infty\in\widetilde X$ be the base point. For $g\in G(\R)$ and 
$x\in\widetilde X$ let
$L_g\colon\widetilde E_x\to\widetilde E_{gx}$ be the isomorphism induced by the
left translation. The kernel $\widetilde K^\nu$ is related to
the convolution kernel $H_t^\nu\colon G(\R)\to\End(V_\nu)$ by
\begin{equation}\label{heat-conv}
H_t^\nu(g_1^{-1}g_2)=L_{g_1}^{-1}\circ \widetilde K^\nu(t,g_1x_0,g_2x_0)\circ L_{g_2},
\quad g_1,g_2\in G(\R).
\end{equation}
Thus we get
\begin{equation}\label{heat-conv1}
h_t^\nu(g):=\tr H_t^\nu(g)=\tr(\widetilde K^\nu(t,x_0,gx_0)\circ L_g),
\quad g\in G(\R).
\end{equation}
Using Proposition \ref{prop-estim} and the fact that $L_g$ is an isometry, 
 we obtain the following corollary.
\begin{corollary}\label{estim3}
Let $d=\dim\widetilde X$. For all $T>0$ there exists $C>0$ such that we have
\[
|h_t^\nu(g)|\le C  t^{-d/2} \exp\left(-\frac{r^2(gx_0,x_0)}{4t}\right)
\]
for all $0<t\le T$ and $g\in G(\R)$.
\end{corollary}
Next we turn to the asymptotic expansion of the heat kernel. For 
$x_0\in\widetilde X$ and $\bx\in T_{x_0}\widetilde X$ let $d_{\bx}\exp_{x_0}$ be 
the differential of the exponential map $\exp_{x_0}\colon T_{x_0}\widetilde 
X\to\widetilde X$ at the point $\bx$.  It is a map
from $T_{x_0}\widetilde X$ to $T_x\widetilde X$, where $x=\exp_{x_0}(\bx)$. Let
\begin{equation}\label{jacob1}
j_{x_0}(\bx):=|\det(d_{\bx}\exp_{x_0})|
\end{equation}
be the Jacobian, taken with respect to the inner products in the tangent spaces.
We use $\exp_{x_0}$ to introduce normal coordinates centered at $x_0$. Let
$g_{ij}(\bx)$ denote the components of the metric tensor in these coordinates.
Then on has
\begin{equation}\label{jacob2}
j_{x_0}(\bx)=|\det(g_{ij}(\bx))|^{1/2}
\end{equation}
(see \cite[(1.22)]{BGV}).
Given $y\in\widetilde X$ and $\bx\in T_y\widetilde X$, let $x=\exp_y(\bx)$.
Put
\begin{equation}\label{jacob3}
j(x;y):=j_y(\bx).
\end{equation}
Let $\ve>0$ be 
sufficiently
small. Let $\psi\in C^\infty(\R)$ with $\psi(u)=1$ for $u<\ve$ and 
$\psi(u)=0$ for $u>2\ve$.

\begin{prop}\label{prop-asympexp}
Let $d=\dim\widetilde X$. Let $(\nu,V_\nu)$
be a finite-dimensional unitary representation of $K_\infty$. There
exist smooth sections $\Phi_i^\nu\in C^\infty(\widetilde X\times\tilde X,
\widetilde E_\nu\boxtimes \widetilde E^*_\nu)$, $i\in\N_0$, such that for 
every $N\in\N$
\begin{equation}\label{asympexp}
\begin{split}
\widetilde K^\nu(t,x,y)=(4\pi t)^{-d/2}\psi(d(x,y))
\exp\left(-\frac{r^2(x,y)}{4t}\right)\sum_{i=0}^N\Phi_i^\nu(x,y&) j(x;y)^{-1/2}t^i\\
&+O(t^{N+1-d/2}),
\end{split}
\end{equation}
uniformly for  $0< t\le 1$. Moreover the leading term $\Phi_0^\nu(x,y)$ is
equal to the parallel transport $\tau(x,y)\colon (\widetilde E_\nu)_y\to 
(\widetilde E_\nu)_x$ with respect to the connection $\nabla^\nu$ along the 
unique geodesic joining $x$ and $y$. 
\end{prop}
\begin{proof}
Let $\Gamma\subset G$ be a co-compact torsion free lattice. It exists by
\cite{Bo}. Let $X=\Gamma\bs \widetilde X$ and $E_\nu=\Gamma\bs 
\widetilde E_\nu$. As in \cite[Sect 3]{Do}, the proof can be reduced to the 
compact case, which follows from \cite[Theorem 2.30]{BGV}. 
\end{proof}

Let $\gf=\kf\oplus \pf$ be the Iwasawa decomposition. We recall that the 
mapping 
\[
\varphi\colon  \pf\times K_\infty\to G(\R),
\]
defined by $\varphi(Y,k)=\exp(Y)\cdot k$ is a diffeomorphism 
\cite[Ch. VI, Theorem 1.1]{He}. Thus each $g\in G(\R)$ can be uniquely
written as 
\begin{equation}\label{cartan2}
g=\exp(Y(g))\cdot k(g),\quad Y(g)\in\pf,\; k(g)\in K_\infty.
\end{equation} 
Using \eqref{heat-conv1} and Proposition \ref{prop-asympexp}, we obtain the 
following corollary.
\begin{corollary}
There exist $a^\nu_i\in C^\infty(G(\R))$, $i\in\N_0$,  such that for every
$N\in\N$ we have
\begin{equation}\label{asympexp1}
h^\nu_t(g)=(4\pi t)^{-d/2}\psi(d(gx_0,x_0))
\exp\left(-\frac{r^2(gx_0,x_0)}{4t}\right)
\sum_{i=0}^N a_i^\nu(g) t^i+O(t^{N+1-d/2})
\end{equation}
which holds for  $0<t\le1$. Moreover the leading 
coefficient $a_0^\nu$ is given by
\begin{equation}\label{leadcoeff}
a_0^\nu(g)=\tr(\nu(k(g)))\cdot j(x_0,gx_0)^{-1/2}.
\end{equation}
\end{corollary}
\begin{proof}
By \eqref{heat-conv} we have
\[
H^\nu_t(g)=\widetilde K^\nu(t,x_0,gx_0)\circ L_g,\quad g\in G(\R).
\]
Put
\begin{equation}\label{trace5}
a_i^\nu(g):=\tr(\Phi_i^\nu(x_0,gx_0)\circ L_g)\cdot j(x_0,gx_0)^{-1/2},
\quad g\in G(\R).
\end{equation}
Then \eqref{asympexp1} follows immediately from \eqref{asympexp} and the 
definition of $h_t^\nu$. To prove the second statement, we recall that
$\Phi_0^\nu(x,y)$ is the parallel transport $\tau(x,y)$ 
with respect to the canonical
connection of $\widetilde E_\nu$ along the geodesic connecting $x$ and $y$.
Let $g=\exp(Y)\cdot k$, $Y\in\pf$, $k\in K_\infty$. Then the geodesic connecting 
$x_0$ and $gx_0$ is the curve $\gamma(t)=\exp(tY)x_0$, $t\in [0,1]$
(see \cite[Ch. IV, Theorem 3.3]{He}). The parallel transport along $\gamma(t)$
equals $L_{\exp(Y)}$. Thus $\Phi_0^\nu(x_0,gx_0)=L_{\exp(Y)}^{-1}$. Hence we get.
\[
\Phi_0^\nu(x_0,gx_0)\circ L_g=L_k=\nu(k).
\]
Together with \eqref{trace5} the claim follows.
\end{proof}

\section{Regularized traces}\label{sec-regtr}
\setcounter{equation}{0}

Let $G$ be a reductive algebraic group over $\Q$. For simplicity, we assume 
that the center $Z_G$ splits over $\Q$, i.e., we have $Z_G=A_G$. Let 
\[
G(\R)^1=G(\A)^1\cap G(\R).
\] 
Then $G(\R)^1$ is semisimple and
\begin{equation}\label{iso4}
G(\R)=G(\R)^1\cdot\Ag.
\end{equation}
Let  $K_\infty\subset G(\R)^1$ be a maximal compact subgroup and let $K_\infty^0$
be the connected component of the identity. Let 
\[
\widetilde X=G(\R)^1/K_\infty^0
\]
be the associated Riemannian symmetric space. Let $\Gamma\subset G(\Q)$
be an arithmetic subgroup. For simplicity we assume that $\Gamma$ is torsion 
free. Note that $\Gamma\subset G(\R)^1$.
Let $X=\Gamma\bs \widetilde X$ be the associated locally symmetric 
manifold. 
Let $\nu\colon K_\infty\to \GL(V_\nu)$ an irreducible unitary representation of
$K_\infty$ and let $E_\nu\to X$ be the associated locally homogeneous vector 
bundle over $X$. Let 
$\Delta_\nu$ be the corresponding Bochner-Laplace operator acting in 
$C^\infty(X,E_\nu)$. Our goal is to define a regularized trace of the heat 
operator $e^{-t\Delta_\nu}$.

Recall that $e^{-t\Delta_\nu}$ is an integral operator with
a smooth kernel $K_\nu(t,x,y)$. By definition, $K_\nu(t,x,y)\in\Hom((E_\sigma)_y,
(E_\sigma)_x)$.
Especially, $K(t,x,x)$ is an endomorphism of $(E_\sigma)_x$. Let 
$\tr K_\nu(t,x,x)$ be
the trace of this endomorphism. If $X$ is compact, then we have
\begin{equation}\label{regtrace0}
\Tr\left(e^{-t\Delta_\nu}\right)=\int_X\tr K_\nu(t,x,x)\,dx.
\end{equation}
To begin with we rewrite \eqref{regtrace0}.
Let $\widetilde\Delta_\nu$ be the Bochner-Laplace operator on the universal
covering $\widetilde X$. Let $H_t^\nu\colon G(\R)^1\to \End(V_\nu)$
be the convolution kernel of $e^{-t\widetilde\Delta_\nu}$ and
\begin{equation}\label{ptw-trace}
h_t^\nu(g):=\tr H_t^\nu(g).
\end{equation}
Assume that  $\Gamma\bs G(\R)^1$ is compact. Then one has  
\begin{equation}\label{trace4}
\Tr\left(e^{-t\Delta_\nu}\right)=\int\limits_{\Gamma\bs G(\R)^1}\sum_{\gamma\in\Gamma}
h_t^\nu(g^{-1}\gamma g)\,dg.
\end{equation}
For the proof see \cite[(3.13)]{MP2}. Let $R_\Gamma$ denote the right regular
representation of $G(\R)^1$ on $L^2(\Gamma\bs G(\R)^1)$. Recall that for 
any $f\in C^\infty_c(G(\R)^1)$, $R_\Gamma(f)$ is the integral operator with kernel 
$\sum_{\gamma\in\Gamma}f(g_1^{-1}\gamma g_2)$. Thus the right hand side of
\eqref{trace4} equals $\Tr R_\Gamma(h_t^\nu)$ and \eqref{regtrace0} can be 
rewritten as
\begin{equation}\label{trace4a}
\Tr\left(e^{-t\Delta_\nu}\right)=\Tr R_\Gamma(h_t^\nu).
\end{equation}

If $X$ is not compact, we choose an appropriate height function $h$ on $X$, that is, a function which measures how far out in the cusp a point is.
For $Y>0$ let $X_Y=\{x\in X\colon h(x)\le Y\}$. If $X=\Gamma\backslash \tilde X$ with $\Gamma\subseteq G(\Q)$ a congruence subgroup, there is a canonical choice of height function: Recall the function $H_0=H_{P_0}: G(\R)^1\longrightarrow\ka_0$ and fix a norm $\|\cdot\|$ on $\ka_0$. Then $X\ni x\mapsto \max_{\gamma\in \Gamma}\|H_0(\gamma x)\|$ defines a height function on $X$.  Assume that $X_Y$ is compact.
Then the integral
\begin{equation}\label{regtrace1}
\int_{X_Y}\tr K_\nu(t,x,x)\,dx
\end{equation}
is well defined. Suppose that this integral has an asymptotic expansion in $Y$.
Then it is natural to define the regularized trace 
$\Tr_{\reg}(e^{-t\Delta_\nu})$ as the 
finite part of the integral as $Y\to\infty$. For hyperbolic manifolds this
has been carried out in \cite{MP1}. The regularized trace defined in this
way depends of course on the choice of the height function. 
 To choose the height function, we pass to the 
adelic setting. In fact, we will not define it explicitly. Instead we use
directly the truncated manifold.

Let $K_f\subset G(\A_f)$ be a neat open
compact subgroup. Let $X(K_f)$ be the arithmetic manifold defined
by \eqref{adel-quot2} and let $E_\nu\to X(K_f)$ be the locally homogeneous 
vector bundle defined by \eqref{vectbdl}. By \eqref{g-modules} we have
\[
L^2(X(K_f),E_\nu)=\bigoplus_{i=1}^l L^2(\Gamma_i\bs\widetilde X,E_{i,\nu}).
\]
Using \eqref{iso4} we extend $h_t^\nu$ to a $C^\infty$-function on $G(\R)$ by
\begin{equation}\label{extension1}
\widetilde h_t^\nu(g_\infty z)=h_t^\nu(g_\infty),\quad g_\infty\in G(\R)^1,\; 
z\in \Ag.
\end{equation}
Let $\1_{K_f}$ denote the characteristic function of $K_f$ in $G(\A_f)$. We 
normalize the characteristic function by
\begin{equation}\label{charact1}
\chi_{K_f}:=\frac{\1_{K_f}}{\vol(K_f)}.
\end{equation}
Define $\phi_t^\nu\in C^\infty(G(\A))$ by 
\begin{equation}\label{extension2} 
\phi_t^\nu(g_\infty g_f)=\widetilde h_t^\nu(g_\infty)\chi_{K_f}(g_f)
\end{equation}
for $ g_\infty\in G(\R)$, $g_f\in G(\A_f)$.
Let $R$  denote the right regular representation of 
$G(\A)$  on $L^2(\Ag G(\Q)\bs G(\A))$,
and let $\Pi_{K_f}$ denote the orthogonal projection of 
$L^2(\Ag G(\Q)\bs G(\A))$ onto $L^2(\Ag G(\Q)\bs G(\A))^{K_f}$. 
Using the isomorphism
\eqref{g-modules} of $G(\R)$-modules, it follows that
\begin{equation}\label{trace6}
R(\phi^\nu_t)=\left[\bigoplus_{i=1}^lR_{\Gamma_i}(h_t^\nu)\right]\circ \Pi_{K_f}.
\end{equation}

Let $\Co(G(\R)^1)$ be Harish-Chandra's Schwartz space (see
\cite[7.1.2]{Wal}). As explained in section \ref{sec-bochlapl}, we have
 $h_t^\nu\in\Co(G(\R)^1)$ for all $t>0$. 
This implies $\phi_t^\nu\in\Co(G(\A),K_f)$.
Thus by \cite[Theorem 7.1]{FL1}, $J^T(\phi_t^\nu)$ is defined for all
$T\in\af_0$.  Let $G(\A)^1_{\le T}$
be defined by \eqref{trunc1} and for $T\in\af_0$ let $d(T)$ be defined 
by \eqref{dT}. Let $C\subset\af_0^+$ be a positive cone for which there 
exists $c>0$ such that
\[
d(T)\ge c\|T\|,\quad \text{for}\;\text{all}\;T\in C.
\]
Then by Theorem \ref{theo-trunc} it follows that for every $t>0$ we have
\begin{equation}\label{trunc3}
\int\limits_{G(\Q)\bs G(\A)^1_{\le T} }\sum_{\gamma\in G(\Q)} 
\phi_t^\nu(x^{-1}\gamma x)\,dx=J^T(\phi_t^\nu)+O\left(e^{-\|T\|/2}\right)
\end{equation}
for all $T\in C$ with $\|T\|>d_0/c$, where $d_0>0$ is as in Theorem 
\ref{theo-trunc} and the implied constant in the remainder term depends on 
$t$. Now by \cite[Theorem 7.1]{FL1}, $J^T(\phi_t^\nu)$ is a polynomial in $T$. 
Therefore
\eqref{trunc3} suggests to define the regularized trace of $e^{-t\Delta_\nu}$ as
the constant term of this polynomial. In order to relate it to the trace
formula, we need an additional assumption. Let $T_0\in\af_0$ be the unique
point determined by \cite[Lemma 1.1]{Ar3}. Then the distribution $J_{\geo}$ 
is defined 
by 
\[
J_{\geo}(f):=J^{T_0}(f),\quad f\in C^\infty_c(G(\A)^1).
\]
and by \cite[Theorem 7.1]{FL1}, $J_{\geo}$ extends continuously to $\Co(G(\A)^1)$.
As proved in \cite[p. 19]{Ar3}, $J_{\geo}$ depends only
on the choice of a maximal compact subgroup $K$ of $G(\A)$ and $M_0$, but is 
independent of the choice of the minimal
parabolic subgroup $P_0$ with Levi component $M_0$. Let $W_0$ be the Weyl group
of $(G,A_0)$. For $s\in W_0$ let $w_s\in G(\Q)$ be a representative of $s$. 
As shown in \cite[p. 10]{Ar3}, $w_s$ belongs to $KM_0(\A)$ for all $s\in W_0$.
Now assume that each $s\in W_0$ has a representative in $G(\Q)\cap K$. Then it
follows from \cite[Lemma 1.1]{Ar3} that $sT_0=T_0$ for all $s\in W_0$, and
therefore $T_0=0$. Thus in this case, the constant term of $J^T(f)$ equals
$J_{\geo}(f)$. Let  $G=\GL(n)$. Then $A_0$ consists of diagonal matrices and 
$W_0$ is equal to the symmetric group $S_n$ which acts by permutations. 
A permutation matrix 
$P_\pi$ is a $n\times n$-matrix where in row $i$ the entry $\pi(i)$ is equal to
$1$ and all other entries are equal to $0$. Such a matrix belongs to
$\GL(n,\Q)\cap K$. The case $G=\SL(n)$ is similar. Thus for 
$G=\GL(n)$ and $G=\SL(n)$ the constant term of $J^T(f)$ equals $J_{\geo}(f)$ 
and the above discussion suggests to define the regularized trace to be
$J_{\geo}(\phi^\nu_t)$. In general $T_0\neq 0$ and therefore, 
$J_{\geo}(\phi^\nu_t)$ is not the 
constant term of the polynomial $J_{\geo}^T(\phi^\nu_t)$. Nevertheless we choose 
$J_{\geo}(\phi^\nu_t)$ as the definition of the regularized trace in general,
because of its independence on the choice of the minimal parabolic subgroup
$P_0$.  
 
\begin{definition}\label{def-regtrace}
Let $G$ be a reductive algebraic group over $\Q$ with $\Q$-split center.
The regularized trace of $e^{-t\Delta_\nu}$ is defined to be
\[
\Tr_{\reg}\left(e^{-t\Delta_\nu}\right):=J_{\geo}(\phi^\nu_t).
\]
\end{definition}

\section{The asymptotic expansion of the regularized trace for $\GL(n)$ and 
$\SL(n)$}
\label{sec-asymp}
\setcounter{equation}{0}

It is well known that on a compact manifold, the trace of the heat operator 
$\Tr\left(e^{-t\Delta}\right)$ of a generalized Laplacian $\Delta$ admits an 
asymptotic expansion as $t\to 0$ (see \cite{BGV}). We wish
to establish a similar result for the regularized trace 
$\Tr_{\reg}(e^{-\Delta_\nu})$ introduced in section \ref{sec-regtr}.
For technical reasons we have to restrict to the group $G=\GL(n)$ or 
$G=\SL(n)$. 

\subsection{Auxiliary results}

We fix an open compact subgroup $K_f\subset G(\A_f)$. 
Let $\phi_t^\nu\in C^\infty(G(\A)^1)$ be defined by \eqref{extension2}.
To begin with we replace $\phi_t^\nu$ by a compactly supported function.
Let $0<a<b$. Let $h\in C^\infty(\R)$ such that $h(u)=1$ if 
$|u|\le a$, and $h(u)=0$, if $|u|\ge b$. Let $d(x,y)$ denote the 
geodesic distance of $x,y\in\widetilde X$. Put
\[
r(g_\infty):=d(g_\infty K_\infty,K_\infty).
\]
Let $\varphi\in C^\infty_c(G(\R))$ be defined by
\[
\varphi(g_\infty):=h(r(g_\infty)).
\]
Define $\tilde\phi_t^\nu\in C^\infty(G(\A))$ by
\begin{equation}\label{extension3}
\tilde\phi_t^\nu(g_\infty g_f):=\varphi(g_\infty)h_t^\nu(g_\infty)\chi_{K_f}(g_f).
\end{equation}
for $g_\infty\in G(\R)$ and $g_f\in G(\A_f)$. Then the restriction of 
$\tilde\phi_t^\nu$ to $G(\A)^1$ belongs to $C^\infty_c(G(\A)^1)$.
\begin{prop}\label{prop-asympexp5}
There exist $C,c>0$ such that
\[
|J_{\geo}(\phi_t^\nu)-J_{\geo}(\tilde\phi_t^\nu)|\le C e^{-c/t}
\]
for $0<t\le 1$.
\end{prop}
\begin{proof}
Let $J_{\spec}(\Phi)$, $\Phi\in C_c^\infty(G(\A)^1)$, be the spectral side of the 
trace
formula. By \cite{FLM1}, $J_{\spec}(\Phi)$ converges absolutely for 
$\Phi\in \Co(G(\A)^1;K_f)$ and by the trace formula we have
\begin{equation}\label{traceform3}
J_{\geo}(\Phi)=J_{\spec}(\Phi),\quad \Phi\in \Co(G(\A)^1;K_f).
\end{equation}
Put  $\psi_t^\nu:=\phi_t^\nu-\tilde\phi_t^\nu$ and $f:=1-\varphi$. Let $\Omega$
(resp. $\Omega_{K_\infty}$) denote the Casimir operator of $G(\R)$ (resp. 
$K_\infty$). Let
\[
\Delta_{G}=-\Omega+2\Omega_{K_\infty}.
\]
By the proof of Theorem 3 of \cite{FLM1} (see \cite[Sect. 5]{FLM1}), it 
follows that there
exists $k\in \N$ such that
\[
|J_{\spec}(\phi_t^\nu)-J_{\spec}(\tilde\phi_t^\nu)|=|J_{\spec}(\psi_t^\nu)|\le C
\|(\Id+\Delta_{G})^k(\psi_t^\nu)\|_{L^1(G(\A)^1)}
\]
for some $C>0$. Now note that by definition
\[
\psi_t^\nu(g_\infty g_f)=f(g_\infty)h_t^\nu(g_\infty)\chi_{K_f}(g_f).
\]
Hence
\[
\|(\Id+\Delta_{G})^k(\psi_t^\nu)\|_{L^1(G(\A)^1)}
=\vol(K_f)\|(\Id+\Delta_{G})^k(fh_t^\nu)\|_{L^1(G(\R)^1)}.
\]
Let $\gf$ be the Lie algebra of $G(\R)$ and let $Y_1,\dots,Y_r$ be an 
orthonormal basis of $\gf$. Then $\Delta_G=-\sum_i Y_i^2$. Denote by $\nabla$
the canonical connection on $G(\R)$. Then it follows that there exists $C_1>0$ 
such that
\[
|(\Id+\Delta_G)^kF(g)|\le C\sum_{l=0}^k\|\nabla^lF(g)\|, \quad g\in G(\R),
\]
for all $F\in C^\infty(G(\R))$. Let $m=\dim G(\R)$. and $j\in\N$.
By \cite[Proposition 2.1]{Mu1}, for every $j\in\N$  there exist
$C_2,c>0$ such that
\begin{equation}\label{connect}
\|\nabla^j h_t^\nu(g)\|\le C_2 t^{-(m+j)/2} e^{-cr^2(g)/t},\quad g\in G(\R),
\end{equation}
for all $0<t\le 1$. Since $f$ vanishes in neighborhood of $1\in G(\R)$,
it follows that there exist $C_4,c_4,c_5>0$ such that 
\[
\sum_{l=0}^k\|\nabla^l(fh_t^\nu)(g)\|\le C_4 e^{-c_4/t}e^{-c_5r^2(g)}
\]
for all $g\in G(\R)$ and $0<t\le 1$. Since $G(\R)\ni g\mapsto e^{-c_5r^2(g)}$ is
integrable on $G(\R)$, we obtain
\[
\|(\Id+\Delta_{G})^k(fh_t^\nu)\|_{L^1(G(\R)^1)}\le C_5 e^{-c_4/t}
\]
for some constant $C_5>0$, which completes the proof.
\end{proof}

Let $K(N)\subset G(\A_f)$ be the principal congruence subgroup of level 
$N\in\N$. From now on we assume that $K_f$ is contained in $K(N)$ for some 
$N\ge 3$. By Proposition \ref{prop-asympexp5} it suffices to show that
$J_{\geo}(\widetilde\phi_t^\nu)$  admits an asymptotic expansion as $t\to 0$.
Now by our assumption on $K_f$ it follows that if the support of 
$f$ is a sufficiently small neighborhood of $0$, then by 
\eqref{unip-contr1} we have 
\[
J_{\geo}(\widetilde\phi_t^\nu)=J_{\unip}(\widetilde\phi_t^\nu).
\]
Now we use the fine geometric expansion \eqref{finegeoexp} by which we express
$J_{\unip}(\widetilde\phi_t^\nu)$ in terms of a finite linear
combination of weighted orbital integrals. Next we use \eqref{orbint11}. Since
at the finite places our test function is fixed, we are reduced to the 
consideration of real weighted orbital integrals in some Levi subgroup of $G$. Since the weighted orbital integrals $J_M^L$ for $L\subseteq G$ a semistandard Levi subgroup can be treated analogously to the case of $L=G$ (they can be reduced to weighted orbital integrals of the form $J_{M'}^{\GL(m)}$ or $J_{M'}^{\SL(m)}$ for suitable $m$ and $M'$), it suffices to 
consider the weighted orbital integrals for $L=G$. 
Now observe that the kernel $H_t^\nu\colon G(\R)^1\to
\End(V_\sigma)$ of $e^{-t\widetilde \Delta_\nu}$ satisfies
\[
H_t^\nu(k^{-1}gk^\prime)=\nu(k)^{-1}\circ H_t^\nu(g)\circ \nu(k^\prime),\quad
\forall k,k^\prime\in K,\forall\; g\in G(\R)^1.
\]
(see \cite[\S 3]{MP4}). Therefore the function $h_t^\nu=\tr H_t^\nu$ 
is invariant under conjugation by $k_\infty\in K_\infty$. Define 
$F_t^\nu$ by
\[
F_t^\nu(g)=\varphi(g)h_t^\nu(g),\quad g\in G(\R).
\]
Then by \eqref{realorbint} the orbital integral that we need to consider is 
given by
\begin{equation}\label{orbint1}
J_M^G(U, F_t^\nu)=c\int_{N(\R)}F_t^\nu(n)w_{M}(n)\; dn,
\end{equation}
where $M\in\cL$, $U\in\left(\cU_M(\R)\right)_{M(\R)}$, $w_{M}(n):=w_{M,U}(n)$ is the weight function described in section \ref{sec-wfct} 
and $N$ is the unipotent radical of some parabolic subgroup $Q\in\cF$. 
Our goal is to determine the asymptotic behavior of the
integral on the right hand side as $t\to 0^+$. 
To study this integral, we identify $N(\R)$ with its Lie algebra $\nf$ via the 
map $n\in N(\R)\mapsto n-\Id$. Furthermore $\nf\cong\R^k$ for some $k\in\N$. 
Let
\[
x\in\R^k\mapsto n(x)\in N(\R)
\]
be the inverse map. With respect to the isomorphism  the invariant measure 
$dn$ is identified with Lebesgue measure $dx$ in $\R^k$. Thus \eqref{orbint1}
equals
\begin{equation}\label{orbint2}
\int_{\R^k} F_t^\nu(n(x))w_{M}(n(x))\; dx.
\end{equation} 
To determine the asymptotic behavior of this integral as $t\to 0^+$, we will
use the asymptotic expansion \eqref{asympexp1} of $h_t^\nu$. To this end we
need to estimate the function
\begin{equation}
r(x):=r(n(x)x_0,x_0),\quad x\in\R^k,
\end{equation}
where $x_0=eK_\infty\in\widetilde X$. Note that $\widetilde X$ is a
Hadamard manifold of nonpositive curvature and the orbit $N(\R)x_0$ is a 
horosphere in $\widetilde X$. Then it follows by \cite[Theorem 4.6]{HH} that 
there exist constants $C,c>0$ such that
\begin{equation}\label{horobd}
r(x)\ge C\arcsinh(c\|x\|),\quad x\in\R^k.
\end{equation}
Now note that 
\[
\arcsinh(x)=\ln\left(x+\sqrt{x^2+1}\right).
\]
Thus we get
\begin{equation}\label{bound2}
r(x)\ge C\ln\left(1+\|x\|\right),\quad x\in\R^k.
\end{equation}
We also need the Taylor expansion of $r(x)^2$ at $x=0$. This is described by the
following lemma.
\begin{lem}\label{lem-taylor}
We have
\[
r(x)^2=\frac{1}{4}\|x\|^2+O(\|x\|^3)
\]
as $x\to 0$. 
\end{lem}
\begin{proof}
Let $H= (H_1,\ldots, H_n)\in\ka$, $H_1+\ldots, H_n=0$,  with 
$n(x)\in K_\infty e^H K_\infty$.  Now note that $r(e^Hx_0,x_0)=\|H\|$ (see
\cite[Corollary 10.42]{BH}). Thus it follows that $r(x)^2= \|H\|^2$. Moreover, 
\[
n + \|x\|^2=\tr \left(n(x)^t n(x)\right) = \tr e^{2H}
= n + 4 \|H\|^2 + O(\|H\|^3). 
\]
If $x\rightarrow 0$, then also $H\rightarrow 0$ so that this equation implies 
$\|H\|^3 = O(\|x\|^3)$ for small $x$. Hence $r(x)^2=\|H\|^2 = 
\frac{1}{4}\|x\|^2 + O(\|x\|^3)$ as $x\rightarrow0$. 
\end{proof}

\subsection{Asymptotics for $t\to0$}

We can now turn to the estimation of the weighted orbital integral
 \eqref{orbint2}. For $\varepsilon>0$ let $B(\varepsilon)\subset \R^k$ denote 
the ball of radius $\varepsilon$ centered at the origin and let
$U(\varepsilon)=\R^k\setminus B(\varepsilon)$.
Let $\psi$ be the function occurring in \eqref{asympexp1}. 
Choose $\varepsilon>0$ so small such that $\varphi(n(x))=1$ for 
$x\in B(\varepsilon)$ and $\supp \psi(n(\cdot))\subset B(\varepsilon)$.
Let $0<t\le 1$. By Corollary~\ref{estim3} we have 
\[
 \bigg|\int_{U(\varepsilon)}\varphi(n(x))h_t^\nu(n(x))w_M(n(x))\;dx\bigg|\le 
C t^{-d/2}\int_{U(\varepsilon)}\exp\left(-\frac{r^2(x)}{4t}\right)\left|w_M(n(x))\right|\;dx
\]
for some absolute constant $C>0$. Using the lower bound~\ref{bound2} for $r(x)$, and the result on the weight function from Proposition~\ref{prop} we can find $c, C_1, C_2>0$ such that
\begin{multline*}
 t^{-d/2}\int_{U(\varepsilon)}\exp\left(-\frac{r^2(x)}{4t}\right)\left|w_M(n(x))\right|\;dx \\
 \le C_1 \exp\left(-\frac{c(\varepsilon)}{t}\right) \int_{\R^k} \exp\left(- c\left(\log(1+ \|x\|)\right)^2\right)\lambda(x)\;dx
\end{multline*}
where $c(\varepsilon)=C_2\log(1+\varepsilon)$ and $\lambda:\R^k\longrightarrow\C$ is a function of the form~\eqref{eq:log:homog:polynom}. By~\eqref{eq:log:int2} of  Lemma~\ref{logconv} the last integral is bounded by a constant so that we finally obtain
\begin{equation}\label{bound3}
\bigg|\int_{U(\varepsilon)}\varphi(n(x))h_t^\nu(n(x))w_M(n(x))\;dx\bigg|
\le C_3\exp\left(-\frac{c(\varepsilon)}{t}\right)
\end{equation}
for some absolute constant $C_3>0$.
To 
deal with the integral over $B(\varepsilon)$, we use \eqref{asympexp1},
which gives 
\begin{equation}\label{asympexp3}
\begin{split}
\int_{B(\varepsilon)}h_t^\nu(n(x))w_M(n(x))\;dx= &t^{-d/2}\sum_{i=0}^N t^i
\int_{B(\varepsilon)}\exp\left(-\frac{r^2(x)}{4t}\right)
a_i^\nu(x)w_M(n(x))\;dx\\
&+O(t^{N+1-d/2})
\end{split}
\end{equation}
for $0<t\le 1$, where $a_i^\nu\in C_c^\infty(B(\varepsilon))$. Each of the 
integrals is of the form
\begin{equation}\label{integr1}
\int_{B(\varepsilon)}\exp\left(-\frac{r^2(x)}{4t}\right)f(x)w_M(\Id+x)\;dx,
\end{equation}
where $f\in C_c^\infty(B(\varepsilon))$. 
We expand $r^2(x)$ and $f(x)$ in their Taylor series at $0$. Let $N\in \N$,
$N\ge 3$. 
By Lemma \ref{lem-taylor} we have
\begin{equation}\label{taylor5}
r^2(x)=a\|x\|^2+\psi_N(x),\quad \psi_N(x)
=\sum_{3\le|\alpha|\le N}a_\alpha x^\alpha+R_N(x),
\end{equation}
where
\begin{equation}\label{remainder}
R_N(x)=\sum_{|\alpha|=N+1}\frac{D^\alpha r^2(\theta x)}{\alpha !}x^\alpha
\end{equation}
for $x\in B(\varepsilon)$ and some $0\le\theta\le 1$.
Now we change variables by $x\mapsto \sqrt{t}x$. Then \eqref{integr1} equals
\begin{equation}\label{integr2}
t^{k/2}\int_{B(\varepsilon t^{-1/2})}\exp(-a\|x\|^2)
\exp\left(-\frac{\psi_N(t^{1/2}x)}{t}\right)f(t^{1/2}x)w_M(1,\Id+t^{1/2}x)
\end{equation}
Now observe that by \eqref{taylor5} we have
\[
t^{-1}\psi_N(t^{1/2}x)=\sum_{3\le|\alpha|\le N}a_\alpha t^{|\alpha|/2-1} x^\alpha+
t^{-1}R_N(t^{1/2}x)
\]
and
\[
t^{-1}R_N(t^{1/2}x)=t^{(N-1)/2}\sum_{|\alpha|=N+1}
\frac{D^\alpha r^2(t^{1/2}\theta x)}{\alpha !}x^\alpha.
\]
There is $C>0$ such that 
\[
|D^\alpha r^2(t^{1/2}\theta x)|\le C
\]
for all $x\in B(t^{-1/2}\varepsilon)$, $0<t\le 1$, and $0\le\theta\le1$. Hence 
it follows 
that there is $C_1>0$ such that for all $0<t\le1$ 
\[
|t^{-1}R_N(t^{1/2}x)|\le C_1 t^{(N-1)/2}\|x\|^{N+1}, \quad x\in 
B(t^{-1/2}\varepsilon).
\]
Using the Taylor expansion of $\exp(u)$, we get for $n\ge 3$
\begin{equation}\label{taylor6}
\exp\left(-\frac{\psi_N(t^{1/2}x)}{t}\right)=\sum_{j=0}^N t^{j/2}p_j(x)+R_N(t,x),
\end{equation}
where $p_j(x)$ is a polynomial of degree $\le N^2$ and the remainder term
satisfies
\begin{equation}\label{remaind3}
|R_N(t,x)|\le C_2 t^{(N-1)/2} (1+\|x\|)^{N^2}
\end{equation}
for some constant $C_2>0$, $0<t\le 1$ and $x\in B(t^{-1/2}\varepsilon)$. 
Similarly, using the Taylor expansion of $f(x)$ we get
\begin{equation}\label{taylor7}
f(t^{1/2}x)=\sum_{|\alpha|\le N}b_\alpha t^{|\alpha|/2}x^\alpha+Q_N(t,x)
\end{equation}
with
\[
|Q_N(t,x)|\le C_3 t^{(N+1)/2}(1+\|x\|)^{N+1}
\]
for $0<t\le 1$ and $x\in B(t^{-1/2}\varepsilon)$. Using \eqref{integr2}, \eqref{taylor6}, 
\eqref{taylor7}, and Proposition \ref{prop}, it follows that
\begin{equation}\label{int-exp}
\begin{split}
\int_{B(\varepsilon)}\exp\left(-\frac{r^2(x)}{4t}\right)f(x)w_M(\Id+x)\;dx=
t^{k/2}\sum_{j=0}^N\sum_{i=0}^{r_j}a_{ij}(t)(\log t)^i t^{j/2}+\phi_N(t),
\end{split}
\end{equation}
where each $a_{ij}(t)$ is of the form
\[
\int_{B(t^{-1/2}\varepsilon)} e^{-a\|x\|^2}p(x)\prod_{l=1}^h\big|\log|p_l(x)|\big|\;dx,
\]
if $i<r_j$, or
\[
\int_{B(t^{-1/2}\varepsilon)} e^{-a\|x\|^2}p(x)\;dx,
\] 
if $i=r_j$, with homogeneous polynomials $p(x)$, $p_1(x),...,p_h(x)$. The fact 
that for $i=r_j$ no logarithm appears in the integral for $a_{ij}(t)$  can 
easily be seen by changing variables 
from $x$ to $t^{1/2} x$ in the integral on the left hand side of \eqref{int-exp}
and collecting all $\log t$ terms coming from the weight function.  

Finally, 
$\phi_N(t)$ satisfies
\[
|\phi_N(t)|\le C t^{(N+k+1)/2}\int_{B(t^{-1/2}\varepsilon)}e^{-a\|x\|^2}(1+\|x\|)^{N^2}
\prod_{i=1}^m\big\|\log|p_i(x)|\big\|\;dx.
\]
Let $U(r)=\R^k\setminus B(r)$. Since $\log(1+\|x\|)\le\|x\|$ for all $x$, it 
follows from  Lemma \ref{logconv} that
\[
\Big|\int_{U(t^{-1/2}\varepsilon)} e^{-a\|x\|^2}p(x)
\prod_{l=1}^h\big|\log|p_l(x)|\big|\;dx\Big|\le C e^{-a\varepsilon^2/(2t)},
\]
for $0<t\le 1$. Thus there are constants $c_{ij}\in\R$ and $c>0$ such that
\[
a_{ij}(t)=c_{ij}+O(e^{-c/t})
\]
for $0<t\le 1$. By the considerations above, for each pair $(i,j)$, 
$0\le i\le r_j$, there exist homogeneous polynomials $p,p_1,\dots,p_h$, such 
that 
\begin{equation}\label{coeffi}
c_{ij}=\begin{cases}
\int_{\R^k} e^{-a\|x\|^2}p(x)\prod_{l=1}^h\big|\log|p_l(x)|\big|\;dx,&
\text{if}\; i<r_j,\\
{}&{}\\
\int_{\R^k} e^{-a\|x\|^2}p(x)\;dx,&\text{if}\; i=r_j.
\end{cases}
\end{equation}
In the same way we get 
\[
|\phi_N(t)|\le C t^{(N+k+1)/2},\quad 0<t\le1.
\]
Putting everything together, we get
\begin{prop}\label{prop-asympexp1}
Let $M\in\cL$, $M\neq G$. For every $N\in\N$, $N\ge 3$, there is an expansion
\begin{equation}\label{unipo-asympexp}
J_M(u,\widetilde\phi_t^\nu)=t^{-(d-k)/2}\sum_{j=0}^N\sum_{i=0}^{r_j}c_{ij}(\nu)t^{j/2}
(\log t)^i+O(t^{(N-d+k+1)/2})
\end{equation}
as $t\to 0^+$.
\end{prop}
Now we come to the first term in \eqref{finegeoexp}, where 
$f=\widetilde\phi_t^\nu$. Then we have to determine the asymptotic behavior of
$h_t^\nu(1)$ as $t\to +0$. Let
$\Gamma^\prime\subset G(\R)$ be a cocompact torsion free lattice. Such a
lattice exists by \cite{Bo1}. Let $X^\prime=\Gamma^\prime\bs \widetilde X$
and $E_\nu^\prime\to X^\prime$ the locally homogeneous vector bundle associated
to $\nu$. Let $\Delta_{X^\prime,\nu}$ be the corresponding Bochner-Laplace
operator. The kernel of $e^{-t\Delta_{X^\prime,\nu}}$, regarded as operator in
$L^2(\Gamma\bs G(\R),\nu)$,  is given by
\[
K^\nu(t,g_1,g_2):=\sum_{\gamma\in\Gamma^\prime} H_t^{\nu}(g_1^{-1}\gamma g_2).
\]
Hence
\[
\begin{split}
\Tr\left(e^{-t\Delta_{X^\prime,\nu}}\right)&=\int_{\Gamma\bs G(\R)}\tr K^\nu(t,g,g)\;dg
=\int_{\Gamma\bs G(\R)}\sum_{\gamma\in\Gamma^\prime} h_t^\nu(g^{-1}\gamma g)\;  dg\\
&=\vol(\Gamma^\prime\bs G(\R))h_t^\nu(1) + \int_{\Gamma^\prime\bs G(\R)}
\sum_{\gamma\in\Gamma^\prime\setminus\{1\}} h_t^\nu(g^{-1}\gamma g)\;dg.
\end{split}
\]
As in \cite[(5.10)]{MP2}, the last term on the right can be estimated by
$C_1e^{-c_1/t}$ for $0<t\le 1$ and some constants $C_1,c_1>0$. Thus we get
\[  
h_t^{\nu}(1)=\frac{1}{\vol(\Gamma^\prime\bs G(\R)}\Tr\left(e^{-t\Delta_{X^\prime,\nu}}
\right)+O(e^{-c_1/t})
\]
for $0<t\le 1$. Now the trace of the heat operator on a compact manifold
has an asymptotic expansion as $t\to +0$ (see \cite{Gi}). Hence, it follows
that there is an asymptotic expansion
\[
h_t^\nu(1)\sim \sum_{j=0}^\infty a_j t^{-d/2+j}
\]
as $t\to +0$. Combined with Proposition \ref{prop-asympexp1} we obtain
Theorem \ref{prop-asymp3}.

\section{The analytic torsion}\label{sec-analtor}
\setcounter{equation}{0}

In this section we assume that $G=\GL(n)$ or $G=\SL(n)$. We consider the case
$G=\SL(n)$. The case $G=\GL(n)$ is similar. We choose $K_\infty=\SO(n)$
as maximal compact subgroup of $G(\R)=\SL(n,\R)$. 
Then $\widetilde X=\SL(n,\R)/\SO(n)$ and $X(K_f)=\Gamma\bs \widetilde X$,
where $\Gamma=(G(\R)\times K_f)\cap G(\Q)$.  

\subsection{The Hodge-Laplace operator and heat kernels}
Let $\tau$ be an irreducible finite-dimensional representation of $G(\R)$ on 
$V_{\tau}$. Let $E_{\tau}$ be the flat vector bundle over $X$ associated to the 
restriction of $\tau$ to $\Gamma$. Let $\widetilde E^\tau$ be the homogeneous
vector bundle associated to $\tau|_{K_\infty}$ and let 
$E^\tau:=\Gamma\bs \widetilde E^\tau$. There is a canonical isomorphism
\begin{equation}\label{iso-vb}
E^\tau\cong E_\tau
\end{equation}
\cite[Proposition 3.1]{MM}. By \cite[Lemma 3.1]{MM}, there exists a positive
definite inner product $\left<\cdot,\cdot\right>$ on $V_{\tau}$ such that 
\begin{enumerate}
\item $\left<\tau(Y)u,v\right>=-\left<u,\tau(Y)v\right>$ for all 
$Y\in\mathfrak{k}$, $u,v\in V_{\tau}$
\item $\left<\tau(Y)u,v\right>=\left<u,\tau(Y)v\right>$ for all 
$Y\in\mathfrak{p}$, $u,v\in V_{\tau}$.
\end{enumerate}
Such an inner product is called admissible. It is unique up to scaling. Fix an 
admissible inner product. Since $\tau|_{K_\infty}$ is unitary with respect to 
this inner product, it induces a metric on $E^{\tau}$, and by \eqref{iso-vb} 
on $E_\tau$, which we also call 
admissible. Let $\Lambda^{p}(E_{\tau})=\Lambda^pT^*(X)\otimes E_\tau$. Let
\begin{align}\label{repr4}
\nu_{p}(\tau):=\Lambda^{p}\Ad^{*}\otimes\tau:\:K_\infty\rightarrow\GL
(\Lambda^{p}\mathfrak{p}^{*}\otimes V_{\tau}).
\end{align}
Then by \eqref{iso-vb} there is a canonical isomorphism
\begin{align}\label{pforms}
\Lambda^{p}(E_{\tau})\cong\Gamma\backslash(G(\R)\times_{\nu_{p}(\tau)}
(\Lambda^{p}\mathfrak{p}^{*}\otimes V_{\tau}))
\end{align}
of locally homogeneous vector bundles. 
Let  $\Lambda^{p}(X,E_{\tau})$ be the space the smooth $E_{\tau}$-valued 
$p$-forms on $X$. 
The isomorphism \eqref{pforms} induces an isomorphism
\begin{align}\label{isoschnitte}
\Lambda^{p}(X,E_{\tau})\cong C^{\infty}(\Gamma\backslash G(\R),\nu_{p}(\tau)),
\end{align}
where the latter space is defined as in \eqref{globsect1}. A corresponding 
isomorphism also holds for the spaces of $L^{2}$-sections.
Let $\Delta_{p}(\tau)$ be the 
Hodge-Laplacian on $\Lambda^{p}(X,E_{\tau})$ with respect to the admissible 
metric in $E_\tau$. Let $R_\Gamma$ denote the right regular representation 
of $G(\R)$ in $L^2(\Gamma\bs G(\R))$. By \cite[(6.9)]{MM} it follows that with
respect to the isomorphism \eqref{isoschnitte} one has
\begin{equation}\label{laplace1}
\Delta_{p}(\tau)=-R_\Gamma(\Omega)+\tau(\Omega)\Id.
\end{equation}
Let $\widetilde E_\tau\to \widetilde X$ be the lift of $E_\tau$ to 
 $\widetilde X$.
There is a canonical isomorphism 
\begin{equation}\label{iso-vbcov}
\Lambda^p(\widetilde X,\widetilde E_\tau)\cong C^\infty(G(\R),\nu_p(\tau)).
\end{equation}
Let $\widetilde\Delta_p(\tau)$ be the lift of $\Delta_p(\tau)$ to 
$\widetilde X$.
Then again it follows  from \cite[(6.9)]{MM} that with respect to the
isomorphism \eqref{iso-vbcov} we have
\begin{equation}\label{kuga}
\widetilde \Delta_p(\tau)=-R(\Omega)+\tau(\Omega)\Id.
\end{equation}
Let $e^{-t\widetilde\Delta_p(\tau)}$ be the 
corresponding heat semigroup. Regarded as an operator in
$L^2(G(\R),\nu_p(\tau))$, it is a convolution operator with kernel
\begin{align}\label{DefHH}
H^{\tau,p}_t\colon G(\R)\to\End(\Lambda^p\mathfrak p^*\otimes
V_\tau)
\end{align} 
which belongs to $C^\infty\cap L^2$ and  satisfies the covariance property
\begin{equation}\label{covar}
H^{\tau,p}_t(k^{-1}gk')=\nu_p(\tau)(k)^{-1} H^{\tau,p}_t(g)\nu_p(\tau)(k')
\end{equation}
with respect to the representation \eqref{repr4}. Moreover, for all $q>0$ we 
have 
\begin{equation}\label{schwartz1}
H^{\tau,p}_t \in (\mathcal{C}^q(G(\R))\otimes
\End(\Lambda^p\pf^*\otimes V_\tau))^{K_\infty\times K_\infty}, 
\end{equation}
where $\mathcal{C}^q(G(\R))$ denotes Harish-Chandra's $L^q$-Schwartz space
(see \cite[Sect. 4]{MP2}). We note that the kernel $H^{\tau,p}_t$ can be expressed
in terms of the kernel $H_t^{\nu_p(\tau)}$ of the heat semigroup $e^{-t\widetilde
\Delta_{\nu_p(\tau)}}$ associated to the Bochner-Laplace operator 
$\widetilde\Delta_{\nu_p(\tau)}$ acting in 
$C^\infty(\widetilde X,\widetilde E_{\nu_p(\tau)})$. For $p=0,\dots,n$ put
\[
E_p(\tau):=\nu_p(\tau)(\Omega_{K_\infty}),
\]
which we regard as an endomorphism of $\Lambda^p\pg^*\otimes V_\tau$. It defines
an endomorphism of $\Lambda^pT^\ast(X)\otimes E_\tau$. By \eqref{BLO} and
\eqref{kuga} we have 
\[
\widetilde\Delta_{p}(\tau)=\widetilde\Delta_{\nu_p(\tau)}+
\tau(\Omega)\Id-E_p(\tau).
\]
Let $\nu_p(\tau)=\oplus_{\sigma\in\Pi(K_\infty)} m(\sigma)\sigma$ be the decomposition
of $\nu_p(\tau)$ into irreducible representations. This induces a
corresponding decomposition of the homogeneous vector bundle
\begin{equation}\label{vb-decomp}
\widetilde E_{\nu_p(\tau)}=\bigoplus_{\sigma\in\Pi(K_\infty)}m(\sigma)\widetilde E_\sigma.
\end{equation}
With respect to this decomposition we have 
\[
E_p(\tau)=\bigoplus_{\sigma\in\Pi(K_\infty)}m(\sigma)
\sigma(\Omega_{K_\infty})\Id_{V_\sigma},
\]
where $\sigma(\Omega_{K_\infty})$ is the Casimir eigenvalue of $\sigma$ and 
$V_\sigma$ the corresponding  representation space. Let $\widetilde\Delta_\sigma$
be the Bochner-Laplace operator associated to $\sigma$. By \eqref{vb-decomp}
we get a corresponding decomposition of 
$C^\infty(\widetilde X,\widetilde E_{\nu_p(\tau)})$ and with respect to this
decomposition we have
\[
\widetilde\Delta_{\nu_p(\tau)}=\bigoplus_{\sigma\in\Pi(K_\infty)} m(\sigma)
\widetilde\Delta_\sigma.
\]
This shows that $\widetilde\Delta_{\nu_p(\tau)}$ commutes with $E_p(\tau)$. Hence
we get
\begin{equation}\label{equ-kernel}
H^{\tau,p}_t=e^{-t(\tau(\Omega)-E_p(\tau))}\circ H^{\nu_p(\tau)}_t.
\end{equation}
Let $h^{\tau,p}_t\in C^\infty(G(\R))$ be defined by
\begin{equation}\label{tr-kern}
h^{\tau,p}_t(g)=\tr H^{\tau,p}_t(g),\quad g\in G(\R).
\end{equation}
Then by \eqref{equ-kernel} we get
\begin{equation}\label{equ-kernel2}
h^{\tau,p}_t=e^{t(\tau(\Omega)-\tr E_p(\tau))}h^{\nu_p(\tau)}_t,
\end{equation}
where $h^{\nu_p(\tau)}_t=\tr H^{\nu_p(\tau)}_t$. As in \eqref{extension2} we define
$\phi^{\tau,p}_t\in C^\infty(G(\A))$ by
\begin{equation}\label{extension5}
\phi^{\tau,p}_t(g_\infty g_f):=h^{\tau,p}_t(g_\infty)\chi_{K_f}(g_f)
\end{equation}
for $g_\infty\in G(\R)$ and $g_f\in G(\A_f)$. 
Following Definition \ref{def-regtrace}, we define the regularized trace of
$e^{-t\Delta_p(\tau)}$ by
\begin{equation}\label{regtr-heat}
\Tr_{\reg}\left(e^{-t\Delta_p(\tau)}\right):=J_{\geo}(\phi^{\tau,p}_t).
\end{equation}

\subsection{Decay for the continuous spectrum}
The next goal is to determine the asymptotic behavior of 
$\Tr_{\reg}\left(e^{-t\Delta_p(\tau)}\right)$ as $t\to\infty$ and $t\to 0^+$.
To study the asymptotic behavior as $t\to\infty$ we use the trace formula 
\eqref{tracef1}. 
By Theorem \ref{thm-specexpand}, $J_{\spec}$  is a distribution on 
$\Co(G(\A);K_f)$ and by \cite[Theorem 7.1]{FL1}, $J_{\geo}$ is continuous on
$\Co(G(\A);K_f)$. This implies  that
\eqref{tracef1} holds for $\phi_t^{\tau,p}$ and we have
\begin{equation}\label{regtrace1a}
\Tr_{\reg}\left(e^{-t\Delta_p(\tau)}\right)=J_{\spec}(\phi_t^{\tau,p}).
\end{equation}
Now we apply Theorem \ref{thm-specexpand} to study the asymptotic behavior
as $t\to\infty$ of the right hand side. Let $M\in\cL$ and $P\in\cP(M)$.
Recall that 
$L^2_{\di}(A_M(\R)^0 M(\Q)\bs M(\A))$ splits as the completed direct sum of its
$\pi$-isotypic components for $\pi\in\Pi_{\di}(M(\A))$. We have a corresponding
decomposition of $\bar{\cA}^2(P)$ as a direct sum of Hilbert spaces
$\hat\oplus_{\pi\in\Pi_{\di}(M(\A))}\bar{\cA}^2_\pi(P)$. Similarly, we have the
algebraic direct sum decomposition
\[
\cA^2(P)=\bigoplus_{\pi\in\Pi_{\di}(M(\A))}\cA^2_\pi(P),
\]
where $\cA^2_\pi(P)$ is the ${\bf K}$-finite part of $\bar{\cA}^2_\pi(P)$. For 
$\sigma\in\widehat{K_\infty}$ let $\cA^2_\pi(P)^\sigma$ be the $\sigma$-isotypic
subspace. Then $\cA^2_\pi(P)$ decomposes as
\[
\cA^2_\pi(P)=\bigoplus_{\sigma\in\widehat{K_\infty}}\cA^2_\pi(P)^\sigma.
\]
Let $\cA^2_\pi(P)^{K_f}$ be the subspace of $K_f$-invariant functions in
$\cA^2_\pi(P)$, and for any $\sigma\in\widehat{K_\infty}$ let 
$\cA^2_\pi(P)^{K_f,\sigma}$ be the $\sigma$-isotypic subspace of 
$\cA^2_\pi(P)^{K_f}$. Recall that $\cA^2_\pi(P)^{K_f,\sigma}$ is finite dimensional.
Let $M_{Q|P}(\pi,\lambda)$ denote the restriction of $M_{Q|P}(\lambda)$ to
$\cA^2_\pi(P)$. Recall that the operator $\Delta_{\mathcal{X}}(P,\lambda)$, which 
appears in the formula \eqref{specside2}, is defined by \eqref{intertw3}.
Its definition involves the intertwining operators $M_{Q|P}(\lambda)$. If we 
replace $M_{Q|P}(\lambda)$ by its restriction $M_{Q|P}(\pi,\lambda)$ to
$\cA^2_\pi(P)$, we obtain the restriction $\Delta_{\mathcal{X}}(P,\pi,\lambda)$ of
$\Delta_{\mathcal{X}}(P,\lambda)$ to $\cA^2_\pi(P)$. Similarly, let $\rho_\pi(P,\lambda)$
be the induced representation in $\bar{\cA}^2_\pi(P)$. 
Fix $\beta\in\bases_{P,L_s}$ and $s\in W(M)$.  Then for the integral 
on the right of \eqref{specside2} with $h=\phi_t^{\tau,p}$ we get
\begin{equation}\label{specside3}
\sum_{\pi\in\Pi_{\di}(M(\A))}\int_{i(\af^G_{L_s})^*}\Tr\left(
\Delta_{\dtup_{L_s}(\bss)}(P,\pi,\lambda)M(P,\pi,s)\rho_\pi(P,\lambda,\phi^{\tau,p}_t)
\right)\;d\lambda.
\end{equation}
Let $P,Q\in\cP(M)$ and $\nu\in\Pi(K_\infty)$.  Denote by 
$\widetilde M_{Q|P}(\pi,\nu,\lambda)$ the restriction of 
\[
M_{Q|P}(\pi,\lambda)\otimes\Id\colon \cA^2_\pi(P)\otimes V_{\nu}\to 
\cA^2_\pi(P)\otimes V_{\nu}
\]
to $(\cA^2_\pi(P)^{K_f}\otimes V_{\nu})^{K_\infty}$. Denote by 
$\widetilde\Delta_{\dtup_{L_s}(\bss)}(P,\pi,\nu,\lambda)$ and 
$\widetilde M(P,\pi,\nu,s)$ the corresponding restrictions. 
Let $m(\pi)$ denote the multiplicity with which $\pi$ occurs in the regular
representation of $M(\A)$ in $L^2_{\di}(A_M(\R)^0 M(\Q)\bs M(\A))$. Then
\begin{equation}\label{iso-ind}
\rho_\pi(P,\lambda)\cong \oplus_{i=1}^{m(\pi)}\Ind_{P(\A)}^{G(\A)}(\pi,\lambda).
\end{equation}
Fix positive restricted roots of $\af_P$ and let $\rho_{\af_P}$ denote the 
corresponding half-sum of these roots. For $\xi\in \Pi(M(\R))$ and
$\lambda\in\af^\ast_P$ let 
\[
\pi_{\xi,\lambda}:=\Ind_{P(\R)}^{G(\R)}(\xi\otimes e^{i\lambda})
\]
be the unitary induced representation. Let $\xi(\Omega_M)$ be the Casimir
eigenvalue of $\xi$. Define a constant $c(\xi)$ by
\begin{equation}\label{casimir4}
c(\xi):=-\langle\rho_{\af_P},\rho_{\af_P}\rangle+\xi(\Omega_M).
\end{equation}
Then for $\lambda\in\af^\ast_P$ one has
\begin{equation}\label{casimir5}
\pi_{\xi,\lambda}(\Omega)=-\|\lambda\|^2+c(\xi)
\end{equation}
(see \cite[Theorem 8.22]{Kn}). Let 
\begin{equation}\label{def-F}
\cT:=\{\nu\in\Pi(K_\infty)\colon [\nu_p(\tau)\colon\nu]\neq 0\}.
\end{equation}
Using \eqref{equ-kernel}, \eqref{iso-ind} and \eqref{TrFT}, it follows that 
\eqref{specside3} is equal to
\begin{equation}\label{specside4}
\begin{split}
\sum_{\pi\in\Pi_{\di}(M(\A))}\sum_{\nu\in\cT}e^{-t(\tau(\Omega)-c(\pi_\infty))}
\int_{i(\af^G_{L_s})^*}e^{-t\|\lambda\|^2}\Tr\left(
\widetilde\Delta_{\dtup_{L_s}(\bss)}(P,\pi,\nu,\lambda)
\widetilde M(P,\pi,\nu,s)\right)\;d\lambda.
\end{split}
\end{equation}

In order to estimate \eqref{specside4} from above, we need the following two 
preparatory results.
\begin{lem}\label{est-casim}
Let $(\tau,V_\tau)\in\Rep(G(\R))$. Assume that $\tau\not\cong\tau_\theta$. Let
$P=MAN$ be a proper parabolic subgroup of $G$ and let 
$K_\infty^M=M(\R)\cap K_\infty$.  Let $\xi\in\widehat{M(\R)}$ and 
assume that $\dim(W_\xi\otimes\Lambda^p\pg^\ast\otimes V_\tau)^{K^M_\infty}\neq 0$.
Then one has
\[
\tau(\Omega)-c(\xi)>0.
\]
\end{lem}
\begin{proof}
Let $\xi\in\widehat{M(\R)}$ with $\dim(W_{\xi}\otimes\Lambda^p\pg^\ast
\otimes V_\tau)^{K^M_\infty}\neq 0$. 
Assume that $\tau(\Omega)-c(\xi)\le 0$. Then by \eqref{casimir5} there
exists $\lambda_0\in\af^\ast$ such that
\[
\pi_{\xi,\lambda_0}(\Omega)=\tau(\Omega).
\]
By Frobenius reciprocity we have
\[
\dim\left(W_\xi\otimes\Lambda^p\pg^\ast\otimes V_\tau\right)^{K^M_\infty}=
\dim\left(\H_{\xi,\lambda_0}\otimes\Lambda^p\pg^\ast\otimes V_\tau\right)^{K_{\infty}}.
\]
Combined with our assumption and \cite[Proposition II,3.1]{BW} it follows
that
\[ 
\dim H^p(\gf,K_\infty;\H_{\xi,\lambda_0,K_{\infty}}\otimes V_\tau)\neq 0,
\]
where $\H_{\xi,\lambda_0,K_{\infty}}$ denotes the subspace of $K_\infty$-finite vectors
of $\H_{\xi,\lambda_0}$. Since $\tau\neq\tau_\theta$, this is a contradiction to the
first statement of \cite[Proposition II. 6.12]{BW}. Thus it follows that
\[
\tau(\Omega)-c(\xi)>0
\]
for all $\xi\in \widehat{M(\R)}$ satisfying 
$\dim(W_\xi\otimes\Lambda^p\pg^\ast\otimes V_\tau)^{K^M_\infty}\neq 0$. 
\end{proof}
\begin{lem}\label{lem-finite}
For every $R\ge 0$, the number of $\pi\in \Pi_{\di}(M(\A))$ with 
$\lambda_{\pi_\infty}\ge -R$ and $\cA_\pi^2(P)^{K_f,\nu}\neq 0$ for some 
$\nu\in\cT$ is finite.
\end{lem}
\begin{proof}
To prove the lemma, it suffices to show that for every $R\ge 0$ we have
\begin{equation}\label{estim11}
\sum_{\substack{\pi\in\Pi_{\di}(M(\A))\\-\lambda_{\pi_\infty}\le R}}
\dim(\cA_\pi^2(P)^{K_f,\nu})<\infty.
\end{equation}
By passing to a subgroup of finite index, we may assume that 
$K_f=\prod_{p<\infty}K_p$. Let $K_{M,f}=K_f\cap M(\A_f)$ and 
$K_{M.\infty}=K_\infty\cap M(\R)$. For $\pi\in\Pi(M(\A))$ and 
$\tau\in\Pi(K_{M,\infty})$ let $\H_{\pi_\infty}(\tau)$ denote the $\tau$-isotypical
subspace of the representation space $\cH_{\pi_\infty}$. Arguing as in the proof 
of Proposition 3.5 in \cite{Mu1}, it follows that in order to establish 
\eqref{estim11}, it suffices to show that for every $\tau\in\Pi(K_{M,\infty})$ 
\[
\sum_{\substack{\pi\in\Pi_{\di}(M(\A))\\-\lambda_{\pi_\infty}\le R}} 
\dim(\cH_{\pi_f}^{K_{M,f}})\cdot\dim(\cH_{\pi_\infty}(\tau))<\infty.
\]
Let $\Gamma_M\subset M(\R)$ be an arithmetic subgroup. Let $\Omega_{M(\R)^1}$ be 
the Casimir element of $M(\R)^1$ and let $A_\tau$ be the
differential operator in $C^\infty(\Gamma_M\bs M(\R)^1;\tau)$ which is
induced by $-\Omega_{M(\R)^1}$. Let $\bar A_\tau$ be its self-adjoint extension
of $A_\tau$ in $L^2$. Proceeding as in the proof of Lemma 3.2 of \cite{Mu1}, it 
follows that
it suffices to show that for every $R\ge 0$, the number of eigenvalues 
$\lambda_i$ of $\bar A_\tau$ (counted with multiplicities), 
satisfying $\lambda_i\le R$ is finite. Let $\Delta_\tau$ be the Bochner-Laplace
operator and let $\Lambda_\tau$ be the Casimir eigenvalue of $\tau$. Then
$\Delta_\tau=A_\tau+\Lambda_\tau\Id$. Since $\Delta_\tau\ge 0$ and by \cite{Mu3},
the counting function of the eigenvalues has a polynomial bound, the lemma
follows.
\end{proof}

Now we can begin with the estimation of \eqref{specside4}.  
Using that $M(P,\pi,s)$ is unitary, it follows that \eqref{specside4} can be
estimated by
\begin{equation}\label{est-specside}
\begin{split}
\sum_{\pi\in\Pi_{\di}(M(\A))}
\sum_{\nu\in\cT}&
\dim\left(\cA^2_\pi(P)^{K_f,\nu}\right)\\
&\cdot e^{-t(\tau(\Omega)-c(\pi_\infty))}\int_{i(\af^G_{L_s})^*}e^{-t\|\lambda\|^2}\|
\widetilde\Delta_{\dtup_{L_s}(\bss)}(P,\pi,\nu,\lambda)\|\;d\lambda.
\end{split}
\end{equation}
For $\pi\in \Pi(M(\A))$ denote by $\lambda_{\pi_\infty}$ the Casimir 
eigenvalue of the restriction of $\pi_\infty$ to $M(\R)^1$. Given $\lambda>0$, 
let
\[
\Pi_{\di}(M(\A);\lambda):=\left\{\pi\in\Pi_{\di}(M(\A))\colon 
|\lambda_{\pi_\infty}|\le\lambda\right\}.
\]
Let $d=\dim M(\R)^1/K_\infty^M$. As in \cite[Proposition 3.5]{Mu1} it follows 
that for every 
$\nu\in\Pi(K_\infty)$ there exists $C>0$ such that
\begin{equation}\label{estim10}
\sum_{\pi\in\Pi_{\di}(M(\A);\lambda)}\dim\cA^2_\pi(P)^{K_f,\nu}\le C(1+\lambda^{d/2})
\end{equation}
for all $\lambda\ge 0$. 
Next we estimate the integral in \eqref{est-specside}.
Let $\bss=(\beta_1^\vee,\dots,\beta_m^\vee)$ and $\dtup_{L_s}(\bss)=
(Q_1,\dots,Q_m)\in\Xi_{L_s}(\bss)$ with with $Q_i=\langle P_i,P_i'\rangle$, 
$P_i|^{\beta_i}P_i'$, $i=1,\dots,m$.
 Using the definition \eqref{intertw3} of 
$\Delta_{\dtup_{L_s}(\bss)}(P,\pi,\nu,\lambda)$, it follows that we can bound the
integral by a constant multiple of
\begin{equation}\label{est-integral}
\dim(\nu)\int_{i(\af^G_{L_s})^*}e^{-t\|\lambda\|^2}\prod_{i=1}^m\left\| 
\delta_{P_i|P_i^\prime}(\lambda)\Big|_{\cA^2_\pi(P_i^\prime)^{K_f,\nu}}\right\|
\;d\lambda.
\end{equation}
We introduce new coordinates $s_i:=\langle\lambda,\beta_i^\vee\rangle$,
$i=1,\dots,m$, on $(\af^G_{L_s,\C})^\ast$. Using \eqref{normalization}, we
can write
\begin{equation}\label{delta}
\delta_{P_i|P_i^\prime}(\lambda)=\frac{n^\prime_{\beta_i}(\pi,s_i)}{n_{\beta_i}(\pi,s_i)}
+j_{P_i^\prime}\circ(\Id\otimes R_{P_i|P_i^\prime}(\pi,s_i)^{-1}
R^\prime_{P_i|P_i^\prime}(\pi,s_i))\circ j_{P_i^\prime}^{-1}.
\end{equation}
Put
\[
\cA^2_\pi(P)^{K_f,\cT}=\bigoplus_{\nu\in\cT}\cA^2_\pi(P)^{K_f,\nu},
\]
where $\cT$ is defined by \eqref{def-F}. 
It follows from \cite[Theorem 5.3]{Mu2} that there
exist $N,k\in\N$ and $C>0$ such that 
\begin{equation}\label{log-deriv}
\int_{i\R}\left|\frac{n^\prime_{\beta_i}(\pi,s)}{n_{\beta_i}(\pi,s)}\right|
(1+|s|^2)^{-k}\;ds\le C(1+\lambda_{\pi_\infty}^2)^N,\;i=1,\dots,m,
\end{equation}
for all $\pi\in \Pi_{\di}(M(\A))$ with $\cA^2_\pi(P)^{K_f,\cT}\neq 0$. 
Furthermore, for $G=\GL(n)$ it follows from \cite[Proposition 0.2]{MS} that
there exist $k, C>0$ such that
\begin{equation}\label{log-deriv1}
\int_{i\R}\left\|R_{P_i|P_i^\prime}(\pi,s)^{-1}
R^\prime_{P_i|P_i^\prime}(\pi,s)\Big|_{\cA^2_\pi(P_i^\prime)^{K_f,\nu}}\right\|(1+|s|^2)^{-k}\, ds \le C,\;
i=1,\dots,m,
\end{equation}
for all $\nu\in\cT$ and $\pi\in \Pi_{\di}(M(\A))$ with 
$\cA^2_\pi(P)^{K_f,\nu}\neq 0$. To show that \eqref{log-deriv1} also holds for
$G=\SL(n)$, we proceed as in the proof of \cite[Lemma 5.14]{FLM2}.
Combining \eqref{delta}, \eqref{log-deriv} and
\eqref{log-deriv1}, it follows that for $t\ge 1$ we have
\[
\int_{i(\af^G_{L_s})^*}e^{-t\|\lambda\|^2}\prod_{i=1}^m\left\| 
\delta_{P_i|P_i^\prime}(\lambda)\Big|_{\cA^2_\pi(P_i^\prime)^{K_f,\nu}}\right\|
\;d\lambda\ll (1+\lambda_{\pi_\infty}^2)^{mN}
\]
for all $\pi\in \Pi_{\di}(M(\A))$ with $\cA^2_\pi(P)^{K_f,\cT}\neq 0$. 
Thus \eqref{est-specside} can be estimated by a constant multiple of
\begin{equation}\label{est-specside1}
\sum_{\pi\in\Pi_{\di}(M(\A))}\sum_{\nu\in\cT}\dim\left(\cA^2_\pi(P)^{K_f,\nu}\right)
(1+\lambda_{\pi_\infty}^2)^{mN} e^{-t(\tau(\Omega)-c(\pi_\infty))}.
\end{equation}
First assume that $M$ is a proper Levi subgroup. Note that by 
\eqref{casimir4} one has 
\begin{equation}\label{eigenv}
\tau(\Omega)-c(\pi_\infty)=\tau(\Omega)+\|\rho_\af\|^2-\lambda_{\pi_\infty}.
\end{equation}
Together with Lemma \ref{lem-finite}, it follows that there 
exists $\lambda_0> 0$ such that
\[
\tau(\Omega)-c(\pi_\infty)\ge |\lambda_{\pi_\infty}|/2
\]
for all $\pi\in \Pi_{\di}(M(\A))$ with $\cA^2_\pi(P)^{K_f,\cT}\neq 0$ and
$|\lambda_{\pi_\infty}|\ge \lambda_0$. Decompose the sum over $\pi$ in 
\eqref{est-specside1} in two summands $\Sigma_1(t)$ and $\Sigma_2(t)$,
where in $\Sigma_1(t)$ the summation runs over all $\pi$ with 
$|\lambda_{\pi_\infty}|\le \lambda_0$. Using \eqref{estim10}, it follows that
for $t\ge 1$
\[
\Sigma_2(t)\ll e^{-t|\lambda_0|/2}.
\]
Since $\Sigma_1(t)$ is a finite sum by Lemma \ref{lem-finite}, both in $\pi$ and $\nu$, it follows from
Lemma \ref{est-casim} that there exists $c>0$ such that
\[
\Sigma_1(t)\ll e^{-ct}
\]
for $t\ge 1$. Putting everything together it follows that for every 
$\tau\in \Rep(G(\R))$ such that $\tau\not\cong\tau_\theta$ and every proper
Levi subgroup $M$ of $G$ there exists $c>0$ such that
\begin{equation}\label{trspec}
J_{\spec,M}(\phi_t^{\tau,p})=O(e^{-ct})
\end{equation}
for $t\ge 1$.

Now consider the case $M=G$. Then $c(\pi_\infty)=\pi_\infty(\Omega)$ and we need
to show that 
\begin{equation}\label{positiv1}
\tau(\Omega)-\pi_\infty(\Omega)>0
\end{equation}
for all $\pi\in\Pi_{\di}(G(\A))$ with $\dim\cH_\pi^{K_f,\cT}\neq 0$. This follows
from \cite[Lemma 4.1]{BV}, and we can proceed as in the
case $M\neq G$ to prove that
\[
J_{\spec,G}(\phi_t^{\tau,p})=O(e^{-ct})
\]
for $t\ge 1$. Combined with \eqref{trspec} we obtain
\begin{prop}\label{asympinf}
Let $\tau\in\Rep(G(\R))$. Assume that $\tau\not\cong\tau_\theta$. Then there 
exists $c>0$ such that 
\[
J_{\spec}(\phi_t^{\tau,p})=O\left(e^{-ct}\right)
\]
for all $t\ge 1$ and $p=0,\dots,n$.
\end{prop}

\subsection{Definition of analytic torsion}
Applying the trace formula \eqref{tracef1}, we get
\[
\Tr_{\reg}\left(e^{-t\Delta_p(\tau)}\right)=O(e^{-ct}),\quad\text{as}\;
 t\to\infty,
\]
which is the proof of Theorem \ref{theo-lt}. 
Using \eqref{regtr-heat}, \eqref{equ-kernel2} and Theorem 
\ref{prop-asymp3}, it follows that as $t\to +0$, there is an asymptotic 
expansion of the form
\[
\Tr_{\reg}\left(e^{-t\Delta_p(\tau)}\right)\sim t^{-d/2}\sum_{j=0}^\infty a_jt^{j}+
t^{-(d-1)/2}\sum_{j=0}^\infty\sum_{i=0}^{r_j} b_{ij}t^{j/2} (\log t)^i.
\]
Thus the corresponding zeta function $\zeta_p(s;\tau)$, defined by the Mellin 
transform 
\begin{equation}\label{zetafct}
\zeta_p(s;\tau):=\frac{1}{\Gamma(s)}\int_0^\infty 
\Tr_{\reg}\left(e^{-t\Delta_p(\tau)}\right) t^{s-1}\; dt.
\end{equation}
is holomorphic in the half-plane $\Re(s)>d/2$ and admits a meromorphic
extension to the whole complex plane. It may have a pole at $s=0$.  Let $f(s)$ 
be a meromorphic function on $\C$. For $s_0\in\C$ let 
\[
f(s)=\sum_{k\ge k_0}a_k(s-s_0)^k
\]
be the Laurent expansion of $f$ at $s_0$. Put $\FP_{s=s_0}:=a_0$. Now we define
the analytic torsion $T_X(\tau)\in\R^+$ by
\begin{equation}\label{analtor3}
\log T_X(\tau)=\frac{1}{2}\sum_{p=0}^d (-1)^p p 
\left(\FP_{s=0}\frac{\zeta_p(s;\tau)}{s}\right).
\end{equation}
Put
\begin{equation}\label{altheat}
K(t,\tau):=\sum_{p=1}^d (-1)^p p \Tr_{\reg}\left(e^{-t\Delta_p(\tau)}\right).
\end{equation}
Then $K(t,\tau)=O(e^{-ct})$ as $t\to\infty$ and the Mellin transform
\[
\int_0^\infty K(t,\tau)t^{s-1} dt
\]
converges absolutely and uniformly on compact subsets of $\Re(s)>d/2$ and 
admits a meromorphic extension to $\C$. Moreover, by \eqref{analtor3} we
have
\begin{equation}\label{analtor1}
\log T_X(\tau)=\FP_{s=0}\left(\frac{1}{\Gamma(s)}\int_0^\infty
K(t,\tau)t^{s-1} dt\right).
\end{equation}
Let
\[
\phi_t^\tau:=\sum_{p=1}^d (-1)^p p \phi_t^{\tau,p} \quad\text{and}\quad 
k_t^\tau:=\sum_{p=1}^d (-1)^p p h_t^{\tau,p}.
\]
Then by \eqref{regtr-heat} we have
\begin{equation}\label{alt-reg-tr}
K(t,\tau)=J_{\spec}(\phi_t^\tau).
\end{equation}
For $\pi\in\Pi(G(\R))$ let $\Theta_\pi$ be the global character. Then we get
\begin{equation}\label{alt-reg-tr1}
J_{\spec,G}(\phi^\tau_t)=\sum_{\pi\in\Pi_{\di}(G(\A))} m(\pi)\dim\left(\cH_{\pi_f}^{K_f}
\right)\Theta_{\pi_\infty}(k_t^\tau).
\end{equation}
For $n\in\N$, $n\ge 2$, let $\delta_n:=\rk_\C \SL(n)- \rk_\C\SO(n)$ be the 
fundamental rank of $\SL(n)$. 
\begin{lem}
For $G=\GL(n)$ and $n\ge 5$ we have $J_{\spec,G}(\phi^\tau_t)=0$.
\end{lem}
\begin{proof}
Let $Q$ be a standard cuspidal parabolic subgroup of $G(\R)$. Let 
$Q=M_QA_QN_Q$ be the Langlands decomposition of $Q$. Let $(\xi,W_\xi)$ be a
discrete series representation of $M_Q$ and let $\nu\in\af_{Q,\C}^\ast$. Let
$\pi_{\xi,\nu}$ be the induced representation. By \cite[Proposition 4.1]{MP2}
we have $\Theta_{\xi,\nu}(k_t^\tau)=0$, if $\dim \af_Q\ge 2$. If $\delta_n\ge 2$, it
follows that $\dim\af_Q\ge 2$ for every cuspidal parabolic subgroup $Q$ of 
$G(\R)$. Thus $\Theta_{\xi,\nu}(k_t^\tau)=0$ for all cuspidal parabolic
subgroups $Q$ and pairs $(\xi,\nu)$ as above. Now observe that for $\GL(n)$
the $R$-group is trivial. Therefore, it follows from  \cite[section 2.2]{De}
that the Grothendieck group of all admissible representations of $G(\R)$ is 
generated by the induced representations $\pi_{\xi,\nu}$ as above. Hence
$\Theta_\pi(k_t^\tau)=0$ for all $\pi\in\Pi(G(\R))$. If $n\ge 5$, then $\delta_n
\ge 2$ and the lemma follows from \eqref{alt-reg-tr1}.
\end{proof}
\begin{remark} 
If $\Gamma$ is cocompact and $n\ge 5$, then it follows that $T_X(\tau)=1$. 
In the noncompact
case this need not be true. In \cite{MP1} the case of finite volume hyperbolic
manifolds has been studied. It has been shown that in even dimensions, the 
renormalized analytic torsion has a simple  expression, but is not trivial.
This includes the case of $\SL(2)$.
\end{remark}

\section{The case $G=\GL(3)$}\label{sect-gl3}
\setcounter{equation}{0}

If $G=\GL(3)$, the weight functions are explicitly given by \eqref{unipot3}-
\eqref{unipot5}. Using the explicit form of the weight function, 
we can extract  more precise information about the pole at $s=0$. To this
end we need to show that the coefficients $c_{ij}(\nu)$ in 
\eqref{unipo-asympexp} with $j=d-k$ and $i=1,\dots,r_{d-k}$ vanish for the 
corresponding orbital integrals. In our case $d=5$. Now consider the first 
integral \eqref{unipot3}. Then $k=2$ and the weight function is $\log(y^2+z^2)$.
Thus the highest power with which $\log t$ occurs in the asymptotic expansion 
of \eqref{unipot3} is 1. This means that $c_{13}(\nu)$ is the only coefficient
that we need to consider. It is of the form \eqref{coeffi}. We are in the case
$i=r_j$. Hence 
\[
c_{13}(\nu)=\int_{\R^2}e^{-a\|x\|^2}p(x)\;dx,
\]
where $p(x)$ is a homogeneous polynomial. Moreover, from its construction it 
follows that $p(x)$ is odd, i.e., $p(-x)=-p(x)$. 
Hence $c_{13}(\nu)=0$. Thus the asymptotic expansion of the first integral 
has the form
\begin{equation}\label{unipot6}
J_{M_1}(1,h_t^\nu)\sim t^{-3/2}\sum_{j=0}^\infty a_j(\nu) t^{j/2}
+t^{-3/2}\sum_{k=0}^\infty b_k(\nu) t^{k/2}\log t,
\end{equation}
as $t\to +0$, and $b_3(\nu)=0$. 

Now consider the second integral \eqref{unipot4}. Then $k=3$. 
By \eqref{unipot4} we  need only to consider $c_{12}(\nu)$, which we denote by 
$c_2(\nu)$. Let $p_1(x)$ and
$p_2(x)$ be the polynomials occurring on the right hand side of \eqref{taylor6}
and $a^\nu_j(g)$ the coefficients on the right hand side of \eqref{asympexp1}.
If we collect all possible contributions, we get
\begin{equation}\label{const1}
\begin{split}
c_{2}(\nu)=&a_0^\nu(1)\int_{\R^3}p_2(x) e^{-\|x\|^2}\;dx+\sum_{i=1}^3
\frac{\partial}{\partial x_i}a_0^\nu(n(x))\big|_{x=0}
\int_{\R^3} x_ip_1(x)e^{-\|x\|^2}\;dx\\
&+ \sum_{i,j=1}^3\frac{\partial^2}{\partial x_i\partial x_j}
a_0^\nu(n(x))\big|_{x=0}\int_{R^3} x_i x_j e^{-\|x\|^2} dx+ 
a_1^\nu(1)\int_{\R^3} e^{-\|x\|^2}\;dx.
\end{split}
\end{equation}
By definition we have
\[
p_1(x)=\sum_{|\alpha|=3} D^\alpha r^2(x)\big|_{x=0} x^\alpha.
\]
Now recall that for $g\in\SL(n,\R)$ the distance $r(g(x_0),x_0)$ is given as
follows. Let $\lambda_1,\dots,\lambda_n$ be the eigenvalues of the positive
definite matrix $g^\top\cdot g$. Then
\[
r(g(x_0),x_0)^2=\sum_{i=1}^n (\log\lambda_i)^2.
\]
An explicit computation shows that 
\[
r^2(x_1,x_2,0)=2\log^2\left(1+\frac{x_1^2+x_2^2}{2}+\sqrt{x_1^2+x_2^2+
\frac{(x_1^2+x_2^2)^2}{4}}\right).
\]
Thus $r(x_1,x_2,0)$ is even in $x_1$ and $x_2$. The same holds for $r(x_1,0,x_3)$
and $r(0,x_2,x_3)$. This implies that for $\alpha\neq(1,1,1)$ we have
$D^\alpha r^2(x)\big|_{x=0}=0$. Finally note that
\[
\int_{\R^3} x_i x_1 x_2 x_3 e^{-\|x\|^2} dx=0,\quad\text{and}\quad \int_{\R^3} x_i x_j
e^{-\|x\|^2}dx=0,\; i\neq j.
\]
Thus \eqref{const1} is reduced to
\begin{equation}\label{const2}
\begin{split}
c_{2}(\nu)=&a_0^\nu(1)\int_{\R^3}p_2(x) e^{-\|x\|^2}\;dx + 
a_1^\nu(1)\int_{\R^3} e^{-\|x\|^2}\;dx\\
&+\sum_{i=1}^3\frac{\partial^2}{\partial x_i^2}
a_0^\nu(n(x))\big|_{x=0}\int_{\R^3}x_i^2 e^{-\|x\|^2} dx.
\end{split}
\end{equation}
Thus for the second integral we get an asymptotic expansion of the form
\begin{equation}\label{unipot7}
J_{M_1}(u(1,0,0), h_t^\nu)\sim t^{-1}\sum_{j=0}^\infty a_j(\nu) t^{j/2} + 
t^{-1}\sum_{k=0}^\infty c_k(\nu) t^{k/2}\log t
\end{equation}
with $c_2(\nu)$ given by \eqref{const2}. Finally consider the integral 
\eqref{unipot5}. Again $k=3$. By \eqref{unipot5}
we only need to consider $c_{12}(\nu)$ and $c_{22}(\nu)$. By the same 
considerations as in 
the previous case, it follows that $c_{22}(\nu)=c_2(\nu)$. Furthermore, 
$c_{12}(\nu)$ has the 
same form as $c_2(\nu)$, except that the integrals contain in addition some 
factors $\log|x_i|$ for $i=1,2,3$. Thus we obtain
\begin{equation}\label{unipot8}
J_{M_0}(1,h_t^\nu)\sim t^{-1}\sum_{j=0}^\infty a_j(\nu) t^{j/2}+ 
t^{-1}\sum_{k=0}^\infty c_{1k}(\nu)t^{k/2}\log t
+t^{-1}\sum_{l=0}^\infty c_{2l}(\nu)t^{l/2}(\log t)^2,
\end{equation}
with $c_{22}(\nu)=c_2(\nu)$, where $c_2(\nu)$ is given by \eqref{const2}, and 
$c_{12}(\nu)$ is given by a similar formula as described above. Now we 
specialize $\nu$ to $\nu_p(\tau)$, which is defined by \eqref{repr4}. 
\begin{lem}\label{lem-res1}
Let $(\tau,V_\tau)$ be a finite dimensional representation of $G(\R)$. We have
\[
\sum_{p=1}^5 (-1)^p p\cdot a_0^{\nu_p(\tau)}(1)=0, \quad \sum_{p=1}^5 (-1)^pp\cdot 
a_1^{\nu_p(\tau)}(1)=0.
\]
\end{lem}
\begin{proof}
By \eqref{leadcoeff} we have 
$a_0^{\nu_p(\tau)}=\dim(\Lambda^p\pg^\ast\otimes V_\tau)=
\binom{5}{p}\cdot\dim V_\tau$. Now observe that $\sum_{p=1}^5 (-1)^p p\binom{5}{p}
=0$. This proves the first statement. For the second statement we note that by
\eqref{trace5} we have $a_1^{\nu_p(\tau)}(1)=\tr(\phi_1^{\nu_p(\tau)}(x_0,x_0))$,
and by \eqref{asympexp}, $\phi_1^{\nu_p(\tau)}(x_0,x_0)$ is the second 
coefficient of the asymptotic expansion as $t\to +0$ of 
$\tr K^{\nu_p(\tau)}(t,x_0,x_0)$. Using the known structure of the coefficient, we
get
\begin{equation}\label{heat-coeff}
a_1^{\nu_p(\tau)}(1)=-\frac{R\cdot\dim(\tau)}{6}\left\{\binom{5}{p}
-6\binom{3}{p-1}\right\},
\end{equation}
where $R$ is the scalar curvature (which is constant) and it is understood that
$\binom{m}{p}=0$, if $p<0$ or $p>m$. For $\tau=1$, this follows from
\cite[Theorem 4.1.7, (b)]{Gi}. It is easy to extend this to the twisted case.
Using \eqref{heat-coeff}, the second statement follows.
\end{proof}
\begin{lem}\label{lem-res2}
For every  finite dimensional representation $(\tau,V_\tau)$ of $G(\R)$ we have
\[
\frac{\partial^2}{\partial x_i^2}\bigg|_{x=0}\left(\sum_{p=1}^5 (-1)^p p\;
a_0^{\nu_p(\tau)}(n(x))\right)=0.
\]
for $i=1,2,3$.
\end{lem}
\begin{proof}
We consider the derivative with respect to $x_1$. Let 
\[
n_1(u)=\begin{pmatrix} 1&u&0\\0&1&0\\0&0&1\end{pmatrix}.
\]
Then
\[
\frac{\partial^2}{\partial x_1^2}a_0^{\nu_p(\tau)}(n(x))\big|_{x=0}=
\frac{\partial^2}{\partial u^2}a_0^{\nu_p(\tau)}(n_1(u))\big|_{u=0}.
\]
By \eqref{leadcoeff} we have
\[
a_0^{\nu_p(\tau)}(n_1(u))=\tr(\nu_p(\tau)(k(u))\cdot j(x_0,n_1(u)x_0),
\]
where $k(u):=k(n_1(u))\in\SO(3)$ is determined by \eqref{cartan2}. 
Furthermore, by
\eqref{repr4} we have
\[
\tr(\nu_p(\tau)(k(u))=
\tr(\Lambda^p\Ad_\pg^\ast(k(u))\cdot\tr(\tau(k(u)).
\]
Let 
\[
S:=\left\{A\in\Mat_3(\R)\colon A=A^t,\;\tr(A)=0\right\},
\]
equipped with the inner product
\[
\langle Y_1,Y_2\rangle=\Tr(Y_1Y_2),\quad Y_1,Y_2\in S.
\]
Then $\pg\cong S$ as inner product spaces. Moreover, the adjoint representation
$\Ad_{\pg}$ of $\SO(3)$ on $S$ is given by 
\begin{equation}\label{adjoint}
\Ad_{\pg}(k)Y=k\cdot Y\cdot k^\ast,\quad k\in\SO(3),\; Y\in S.
\end{equation}
With respect to  this isomorphism, $k(u)$ is determined as follows. 
Let $A(u):=n_1(u)n_1(u)^\ast$. Then $A(u)=A(u)^t$ and $A(u)>0$. 
Let $S(u)=A(u)^{-1/2}$. Then $k(u)=S(u)\cdot n_1(u)$. Note that $k(u)$ is a 
block diagonal matrix of the form
\[
\begin{pmatrix} r(\theta)&0\\0&1\end{pmatrix},
\]
where $r(\theta)\in\SO(2)$ is the rotation by the angle $\theta$. Let
\[
Y_1=\begin{pmatrix}-1/2&0&0\\0&-1/2&0\\0&0&1\end{pmatrix}.
\]
Then with respect to \eqref{adjoint} we have $\Ad_{\pg}(k(u))(Y_1)=Y_1$. Let 
$S_1:=\R Y_1$ and $S_0=S_1^\perp$. Then the decomposition $S=S_0\oplus S_1$ is
invariant under $\Ad_{\pg}(k(u))$ and $\Ad_{\pg}(k(u))|_{S_1}=\Id$. Let
$T(u):=\Ad_{\pg}(k(u))|_{S_0}$. Then we have
\begin{equation}\label{det}
\sum_{p=1}^5 (-1)^p p \;\tr(\Lambda^p\Ad_{\pg}^\ast(k(u)))=\sum_{p=0}^5 (-1)^p
\tr(\Lambda^p T(u))=\det(\Id-T(u)).
\end{equation}
For $\lambda\in\C$ let
\begin{equation}\label{det1}
f(\lambda,u):=\det(\lambda\Id-T(u)),\quad \lambda\in\C,\;u\in\R.
\end{equation}
Recall that $T(u)$ is unitary. So every eigenvalue $\mu$ of $T(u)$ satisfies
$|\mu|=1$. Assume that $|\lambda|\neq 1$. Then $f(\lambda,u)\neq 0$ for all
$u\in\R$ and
\begin{equation}\label{logder1}
\frac{\partial}{\partial u}\log f(\lambda,u)
=-\tr(T^\prime(u)(\lambda\Id-T(u))^{-1}),
\end{equation}
where $T^\prime(u)=\frac{d}{du}T(u)$. 
Note that $f(\lambda,0)=\det(\lambda\Id-T(0))=(\lambda-1)^4$. Thus
\[
\frac{\partial}{\partial u}f(\lambda,u)\big|_{u=0}
=-(\lambda-1)^3\tr(T^\prime(0)).
\]
Since $T(u)$ is orthogonal, it follows that $\tr(T^\prime(0))=0$,
and therefore
\begin{equation}\label{firstder}
\frac{\partial}{\partial u}f(\lambda,u)\big|_{u=0}=0.
\end{equation}
Using \eqref{logder1}, we get 
\begin{equation}\label{deriv}
\begin{split}
\frac{\partial^2}{\partial u^2}f(\lambda,u)&=-\frac{\partial}{\partial u}
f(\lambda,u)\cdot\tr(T^\prime(u)(\lambda\Id-T(u))^{-1})\\
&-f(\lambda,u)\tr(T^{\prime\prime}(u)(\lambda\Id-T(u))^{-1})\\
&-f(\lambda,u)\tr(T^\prime(u)(\lambda\Id-T(u))^{-1}T^\prime(u)
(\lambda\Id-T(u))^{-1}).
\end{split}
\end{equation}
Using \eqref{firstder}, we obtain
\[
\frac{\partial^2}{\partial u^2}f(\lambda,u)\big|_{u=0}=
-(\lambda-1)^3\tr(T^{\prime\prime}(0))-(\lambda-1)^2\tr(T^\prime(0)^2).
\]
Since $f(\lambda,u)$ is a polynomial in $\lambda$, it follows that this
equality holds for all $\lambda\in\C$. In particular, we get
\[
\frac{\partial^2}{\partial u^2}f(1,u)\big|_{u=0}=0.
\]
Combined with \eqref{det} and the definition of $f(\lambda,u)$, the statement 
follows for $i=1$. The proof of the other cases is similar. 
\end{proof}
Using \eqref{unipot7}, \eqref{unipot8} and Lemmas \ref{lem-res1} and 
\ref{lem-res2}, it follows that
\[
\sum_{p=1}^5 (-1)^p p\;\zeta_p(s;\tau)
\]
is holomorphic at $s=0$. Thus in this case we can define $\log T_{X(K_f)}(\tau)$
by
\[
\log T_{X(K_f)}(\tau)=\frac{1}{2}\frac{d}{ds}\left(\sum_{p=1}^5 (-1)^p p\;
\zeta_p(s;\tau)\right)\bigg|_{s=0}.
\]

\section{Example: Classes of finite order for $\GL(2)$ and $\GL(3)$}\label{sec:examples}
\label{sec-finite-ord}
\setcounter{equation}{0}
In order to remove the assumption that $\Gamma\subseteq \Gamma(N)$ for some $N\ge 3$, we need to understand distributions $J_{\of}$ appearing in the coarse geometric expansion of the trace formula for which the equivalence classes $\of$ which are not necessarily unipotent. Let $K_f$ be an arbitrary subgroup of $G(\widehat{\Z})$ of finite index and let $f=f_{\infty}\cdot 1_{K_f}\in C_c^{\infty}(G(\A)^1)$ with $f_{\infty}\in C_c^{\infty} (G(\R)^1)$ and $1_{K_f}\in C_c^{\infty}(G(\A_f))$ the characteristic function of $K_f$. In this situation, more than just the unipotent orbit may contribute non-trivially to the coarse geometric expansion.
The equivalence classes $\of\in\CmO$ are in bijection with semisimple orbits in $G(\Q)$. Hence there is a canonical bijection between $\CmO$ and monic polynomial of degree $n$ with rational coefficients and non-vanishing constant term if $G=\GL(n)$ by sending the semisimple conjugacy class to its characteristic polynomial. We may therefore speak of the characteristic polynomial and the eigenvalues of a class $\of$. 
The following lemma is explained in the proof of \cite[Lemma 5.1]{LM}.
\begin{lem}
 We can choose a $K_{\infty}$-bi-invariant neighborhood $\omega\subseteq G(\R)^1$ of $K_{\infty}$ such that if $\of\in\CmO$ is such that there exists $f_{\infty}\in C_c^{\infty}(G(\R)^1)$ supported in $\omega$ with $J_\of(f_{\infty}\cdot 1_{K_f})\neq0$, then the eigenvalues of $\of$ are all roots of unity (over some algebraic closure of $\Q$).
\end{lem}

Let $\CmO_1$ denote the set of all $\of\in\CmO$ whose eigenvalues (in some algebraic closure of $\Q$) are all roots of unity. Note that this set is finite. By the preceding lemma we can choose a bi-$K_{\infty}$-invariant $f_{\infty}\in C_c^{\infty}(G(\R)^1)$ with $f(1)=1$ and 
\[
 J_{\geo}(f_{\infty}\cdot 1_{K_f})=\sum_{\of\in \CmO_1} J_{\of}(f_{\infty}\cdot 1_{K_f}).
\]
Let $\of\in\CmO_1$, and let $\sigma\in G(\Q)\cap \of$ be a semisimple representative for $\of$. Then $\sigma$ is in $G(\R)$ conjugate to some element $\sigma_\infty$ in $\rO(n)$.
For each $\of$ and $f\in C_c^{\infty}(G(\A)^1)$ we have the fine expansion
\[
 J_{\of}(f)=\sum_{(M,\gamma)} a^M(\gamma, S) J_M(\gamma, f),
\]
where $S$ is a sufficiently large finite set of places of $\Q$ with $\infty\in S$, $a^M(\gamma, S)$ are certain global coefficients as defined in \cite{Ar7}, $(M, \gamma)$ runs over all pairs of Levi subgroups $M$ containing $M_0$ and $\gamma$ over representatives of the $M(\Q)$-conjugacy classes in $M(\Q)\cap \of$, and $J_M(\gamma, f)$ are $S$-adic weighted orbital integrals. Since the (finite) set $\CmO_1$ and the set $S$ are fixed in our setting, the value of the coefficients $a^M(\gamma, S)$ is not relevant for us.

\subsection{Orbits of finite order for $\GL(2)$}
If $G=\GL(2)$, then each $\of\in\CmO_1$ is represented by one of the following semisimple elements:
\[
\sigma_0^\pm=\pm \begin{pmatrix}1&0\\0&1\end{pmatrix}, \quad
\sigma_1= \begin{pmatrix}1&0\\0&-1\end{pmatrix}, \quad
\sigma_i^\pm= \pm\begin{pmatrix}\cos\theta_i&\sin\theta_i\\-\sin\theta_i&\cos\theta_i\end{pmatrix},\quad i=2,3,
\]
with $\theta_2=\pi/2$ and  $\theta_3=\pi/3$. We accordingly write $\of_i$ or $\of_i^\pm$ for the associated equivalence classes. Note that $\sigma_1, \sigma_2^\pm, \sigma_3^\pm$ are all regular semisimple so that the associated equivalence class is in fact equal to the conjugacy class of the respective element. (In fact, $\of_2^+=\of_2^-$, but we keep the superscript to make notation more uniform.) Moreover, since we assume that our test function $f$ is $K_\infty$-invariant, $J_{\unip}(f)=J_{\of_0^-}(f)$ for $\of_0^-$ the class attached to $\sigma_0^-$. Hence we only need to consider the regular elements. 

The element $\sigma_1$ is the only of the remaining elements which is split over $\R$. Since it is regular, the distribution $\of_1$ is of a simple form, namely,
\[
 J_{\of_1}(f)=\int_{U_0(\A)} f(u^{-1} \sigma_1 u) v_T(u)\, du
\]
for every bi-$K_{\infty}$-invariant $f\in C_c^{\infty}(G(\A)^1)$.
It follows that if $f= f_{\infty}\cdot 1_{K_f}$, then
\[
 J_{\of_1}(f)=a_1\int_{\R} f_{\infty}(u(x)) \log(1+x^2) \, dx
 + a_2 \int_{\R}f_{\infty} (u(x))\, dx
\]
where $u(x)=\left(\begin{smallmatrix}1&x\\0&1\end{smallmatrix}\right)$, and $a_1, a_2\in\R$ are suitable constants depending only on $K_f$. Using Taylor expansion of $\log(1+x^2)$ around $x=0$, we get, as in \S~\ref{sec-asymp}, that for every $N>0$,
\[
 J_{\of_1}(\tilde\phi_t^\nu)
 = t^{-(d-1)/2}\sum_{k=0}^N c_k t^{k/2} + O_N(t^{(N-d+1)/2})
\]
for suitable coefficients $c_k$.

The remaining classes are regular elliptic and non-split over $\R$. In particular, for each $i\in\{2,3\}$ we have
\[
 J_{\of_i^\pm}(f)= a_i\int_{G_{\sigma_i}(\R)\backslash G(\R)} f_{\infty}(g^{-1}\sigma_i^\pm g)\, dg
\]
for a suitable constant $a_i\in\R$ again depending only on $K_f$ and $\of_i^\pm$. We have for $i=2,3$ that  $G_{\sigma_i^\pm}(\R)=Z(\R) K_{\infty}$ so that using $KAK$ decomposition we get
\[
  J_{\of_i^\pm}(f)= a_i\int_{0}^{\infty} f_{\infty}(a^{-1}\sigma_i^\pm a) \sinh(2X)\, dX
\]
where $a=e^X$. 
We can write
\[
\pm\begin{pmatrix} \cos\theta_i & e^{-2X}\sin\theta_i\\-e^{-2X}\sin\theta_i & \cos\theta_i\end{pmatrix}
= a^{-1}\sigma_i^\pm a
 = k_1 \begin{pmatrix} e^Y&0\\0&e^{-Y}\end{pmatrix} k_2
\]
with suitable $k_1, k_2\in K_{\infty}$ and $Y\ge0$. Hence
\[
 \sinh(Y)= \alpha_i\sinh(2X)
\]
with $\alpha_i=\sqrt{2}\sin\theta_i$. Hence
\[
 Y^2= 4\alpha_i^2 X^2+ O_{\of_i^\pm}(X^4)
\]
around $0$. Since $r(a^{-1}\sigma_i^\pm a)= \|(Y, -Y)\|$, we therefore get
\[
 r^2(a^{-1}\sigma_i^\pm a)
 = 8\alpha_i^2 X^2 + O_{\of_i^\pm}(X^4).
\]
Using the Taylor expansion of $\sinh(2X)$, we get, as in \S~\ref{sec-asymp}, that for any $N$,
\[
 J_{\of_i^\pm}(\tilde\phi_t^\nu) = t^{-(d-2)/2}\sum_{k=0}^N c_k t^{k/2}+ O_{N,\of_i^\pm}(t^{(N-d+2)/2})
\]
as $t\rightarrow0^+$ for suitable coefficients $c_k$ depending on $\of_i^\pm$.

\subsection{Orbits of finite order for $\GL(3)$}
For $G=\GL(3)$ the real weighted orbital integrals associated to $G(\R)$-conjugacy classes of elements in the classes in $\CmO_1$ can have a more complicated form. Each $\of\in\CmO_1$ has a semisimple element which in $G(\R)$ is conjugate to one of the following matrices:
\[
 \sigma_0^\pm=\pm\begin{pmatrix}1&0&0\\0&1&0\\0&0&1\end{pmatrix},
 \sigma_1^\pm=\pm\begin{pmatrix}1&0&0\\0&1&0\\0&0&-1\end{pmatrix},
 \sigma_i^{\pm,\pm}=\pm \begin{pmatrix}\cos\theta_i&\sin\theta_i&0\\-\sin\theta_i&\cos\theta_i&0\\0&0&\pm1\end{pmatrix},\quad i=2,3,
\]
with $\theta_2$ and $\theta_3$ as for $\GL(2)$. Again, we write $\of_i^\pm$, $i=0,1$, and $J_{\of_i^{\pm,\pm}}$, $i=2,3$, for the associated equivalence classes.
We already understand the distributions $J_{\of_0^\pm}$.

If $i=2,3$, the only semi-standard Levi subgroups of $G(\R)$ containing $\sigma_i^{\pm,\pm}$ are $M(\R):=\GL_2(\R)\times\GL_1(\R)$ (diagonally embedded in $G(\R)$) and $G(\R)$ itself. Let $P(\R)$ be the standard parabolic subgroup of $G(\R)$ with Levi component $M(\R)$. Moreover, $\of_i^{\pm,\pm}$ in fact equals the conjugacy class of $\sigma_i^{\pm,\pm}$. Hence we need to understand the real weighted orbital integrals $J_M(\sigma_i^{\pm,\pm},\tilde\phi_{t,\infty}^\nu)$ and $J_G(\sigma_i^{\pm,\pm},\tilde\phi_{t,\infty}^\nu)$.
The latter integral equals
\begin{align*}
 J_G(\sigma_i^{\pm,\pm},\tilde\phi_{t,\infty}^\nu)
 & =\int_{M(\R)_{\sigma_i^{\pm,\pm}}\backslash M(\R)} \int_{U(\R)} \tilde\phi_{t,\infty}^\nu(u^{-1}m^{-1}\sigma_i^{\pm,\pm} mu)\, du\, dm \\
 & = \int_{M(\R)_{\sigma_i^{\pm,\pm}}\backslash M(\R)} \int_{U(\R)} \tilde\phi_{t,\infty}^\nu(m^{-1}\sigma_i^{\pm,\pm} mu)\, du\, dm \\
\end{align*}
where we used the $O(n)$-conjugation invariance of $\tilde\phi_{t,\infty}^\nu$ and that the centralizer of $\sigma_i^{\pm,\pm}$ in $G(\R)$ and $M(\R)$ coincide. 
Hence for $t\rightarrow0$ we get an asymptotic expansion
\[
 J_G(\sigma_i^{\pm,\pm},\tilde\phi_{t,\infty}^\nu)
 =t^{-(d-4)/2}\sum_{k=0}^N C_k t^{k/2}+ O_N(t^{(N-d+4)/2})
\]
for any $N>0$ where $C_k$ are certain coefficients depending on $\of_i^{\pm,\pm}$.

The other weighted orbital integral is of the form
\[
J_M(\sigma_i^{\pm,\pm},\tilde\phi_{t,\infty}^\nu)
=\int_{M(\R)_{\sigma_i^{\pm,\pm}}\backslash M(\R)} \int_{U(\R)} \tilde\phi_{t,\infty}^\nu(u^{-1}m^{-1}\sigma_i^{\pm,\pm} mu) v_M(u)\, du\, dm 
\]
where the weight function is given by
\[
 v_M(u)=\log(1+x^2+y^2)
\]
for $u=\begin{pmatrix}1&0&x\\0&1&y\\0&0&1\end{pmatrix}\in U(\R)$. A change of variables therefore gives
\[
J_M(\sigma_i^{\pm,\pm},\tilde\phi_{t,\infty}^\nu)
=\int_{M(\R)_{\sigma_i^{\pm,\pm}}\backslash M(\R)} \int_{U(\R)} \tilde\phi_{t,\infty}^\nu(m^{-1}\sigma_i^{\pm,\pm} mu) \log\left(1+\|w\|^2\right)\, du\, dm 
\]
where
\[
 w=\left({\rm id} - m^{-1}\sigma_i^{\pm,\pm} m\right)^{-1}\begin{pmatrix}x\\ y\end{pmatrix}\in\R^2.
\]
Using the series expansion of $\log$ around $1$, this integral also has an asymptotic expansion in $t$ as $t\rightarrow 0$. Altogether, we get 
\[
 J_{\of_i^{\pm,\pm}}(\tilde\phi_t^\nu)
= t^{-(d-4)/2}\sum_{k=0}^N B_k t^{k/2}+ O_N(t^{(N-d+4)/2})
\]
for suitable constants $B_k$ and any $N>0$.

The remaining two classes $\of_1^\pm$ contain more elements than just the conjugates of $\sigma_1^\pm$. Hence we have more orbital integrals to consider. Let $u_1=\left(\begin{smallmatrix}1&1&0\\0&1&0\\0&0&1\end{smallmatrix}\right)$. Then we need to consider the weighted orbital integrals $J_{M_0}(\sigma_1^\pm, f_\infty)$, $J_L(\sigma_1^\pm, f_\infty)$, and $J_L(\sigma_1^\pm u_1, f_\infty)$, $L=M, G$, for $f_\infty=\tilde\phi_{t,\infty}^\nu$. For the invariant integrals we get
\[
 J_G(\sigma_1^\pm, f_\infty)
 =\int_{U(\R)} f_\infty(u^{-1}\sigma_1^\pm u) \, du
 = c_1\int_{U(\R)} f_\infty(\sigma_1^\pm u)\, du,
\]
and similarly, after a change of variables,
\[
 J_G(\sigma_1^\pm u_1, f_\infty)
 =c_2\int_{U_0(\R)} f_\infty(\sigma_1^\pm u) \, du.
\]
Here $c_1, c_2>0$ are suitable constants.
The weighted orbital integrals can also be written as integrals over $U_0(\R)$, but against a non-invariant measure. For $J_{M_0}(\sigma_1^\pm,f_{\infty})$ it involves a weight function of the form $\log(1+x^2+y^2)$ as above and a linear function in $\log|a|$ if we write $u=\left(\begin{smallmatrix}1&a&0\\0&1&0\\0&0&1\end{smallmatrix}\right)\left(\begin{smallmatrix}1&0&x\\0&1&y\\0&0&1\end{smallmatrix}\right)$.
Similarly, $J_M(\sigma_1^\pm u_1,f_\infty)$ equals an integral over $U_0(\R)$ against $\log(1+x^2+y^2)$ times the invariant measure, and $J_M(\sigma_i^\pm, f_\infty)$ equals the integral over $U(\R)$ against $\log(1+x^2+y^2)$ times the invariant measure on $U(\R)$.
Proceeding similarly as before, one can then show that
\[
 J_{\of_1^\pm}(\tilde\phi_t^\nu)
 = t^{-(d-3)/2}\sum_{k=0}^N C_k t^{k/2}+ t^{-(d-3)/2}\sum_{k=0}^NB_kt^{k/2} \log t + O_N(t^{(N-d+3)/2})
\]
for suitable constants $C_k, B_k$ and any $N$.

\end{document}